\documentclass{article}
\usepackage{
amsmath,
amsfonts,
latexsym, 
amsrefs,
amssymb
}
\usepackage{mathtools}

\usepackage{hyperref}

\usepackage{bm}

\usepackage{tocloft}

\usepackage{pstricks}

\usepackage{subcaption}

\usepackage{tikz}
\usepackage{tikz,fullpage}
\usetikzlibrary{shapes,arrows,%
                petri,%
                topaths, automata}%
\usepackage{tkz-berge}

\usepackage{graphicx}
\usepackage{enumerate}
\usepackage{mathrsfs}
\usepackage{wrapfig}

\usepackage{color}

\numberwithin{equation}{section}

\newcommand{\rd}{{\rm d}}

\newcommand{\bZ}{{\mathbb Z}}

\newcommand{\beq}{\begin{equation}}
\newcommand{\bEq}{\end{equation}}

\newcommand{\bu}{{\bf{u}}}

\newcommand{\al}{\alpha}

\newcommand{\be}{\begin{equation}}
\newcommand{\ee}{\end{equation}}

\newcommand{\e}{{\varepsilon}}

\usepackage{amsmath} 
\usepackage{amssymb}
\usepackage{amsthm}

\renewcommand{\cal}{\mathcal}

\newcommand{\wt}{\widetilde}

\newcommand{\ii}{\mathrm{i}} 
\newcommand{\dd}{\mathrm{d}}

\newcommand{\col}{\mathrel{\mathop:}}

\newcommand{\deq}{\mathrel{\mathop:}=}

\renewcommand{\epsilon}{\varepsilon}
\renewcommand{\leq}{\leqslant}
\renewcommand{\geq}{\geqslant}

\renewcommand{\le}{\leq}
\renewcommand{\ge}{\geq}

\renewcommand{\P}{\mathbb{P}}
\newcommand{\E}{\mathbb{E}}
\newcommand{\R}{\mathbb{R}}
\newcommand{\C}{\mathbb{C}}
\newcommand{\N}{\mathbb{N}}
\newcommand{\Z}{\mathbb{Z}}

\DeclareMathOperator{\re}{Re}
\DeclareMathOperator{\im}{Im}

\DeclareMathOperator{\OO}{O}
\DeclareMathOperator{\oo}{o}

\theoremstyle{plain} 
\newtheorem{theorem}{Theorem}[section]
\newtheorem*{theorem*}{Theorem}
\newtheorem{lemma}[theorem]{Lemma}
\newtheorem{assumption}[theorem]{Assumption}
\newtheorem*{lemma*}{Lemma}

\newtheorem*{corollary*}{Corollary}

\newtheorem*{proposition*}{Proposition}
\newtheorem{definition}[theorem]{Definition}
\newtheorem*{definition*}{Definition}

\theoremstyle{remark} 

\newtheorem*{example*}{Example}
\newtheorem{remark}[theorem]{Remark}

\newtheorem*{remark*}{Remark}
\newtheorem*{remarks*}{Remarks}

\makeatletter
\renewcommand{\subsection}{\@startsection
{subsection}
{2}
{0mm}
{-\baselineskip}
{0 \baselineskip}
{\normalfont\bf\itshape}} 
\makeatother

\usepackage{dsfont}
\usepackage{stmaryrd}

\newcommand{\nc}{\normalcolor}

\def\bZ{{\mathbb Z}}

\def\@empty{}

\def\author#1{\par
    {\centering{\authorfont#1}\par\vspace*{0.05in}}
}

\def\titlefont{\fontsize{13}{15}\bfseries\boldmath\selectfont\centering{}}
\def\authorfont{\fontsize{13}{15}}
\def\abstractfont{\fontsize{8}{10}}

\let\affiliationfont\rhfont

\def\address#1{\par
    {\centering{\affiliationfont#1\par}}\par\vspace*{11pt}
}

\def\body{
\setcounter{footnote}{0}
\def\thefootnote{\alph{footnote}}
\def\@makefnmark{{$^{\rm \@thefnmark}$}}
}

\def\title#1{
    \thispagestyle{plain}
    \vspace*{-14pt}
    \vskip 79pt
    {\centering{\titlefont #1\par}}%
    \vskip 1em
}

\renewenvironment{abstract}{\par%
    \vspace*{6pt}\noindent 
    \abstractfont
    \noindent\leftskip18pt\rightskip18pt
}{%
  \par}

\usepackage{tocloft}

\makeatletter
\renewcommand{\section}{\@startsection
{section}
{1}
{0mm}
{-2\baselineskip}
{2\baselineskip}
{\normalfont\large\scshape\centering}} 
\makeatother

\newcommand{\tnorm}[1]{{\left\vert\kern-0.25ex\left\vert\kern-0.25ex\left\vert #1 
    \right\vert\kern-0.25ex\right\vert\kern-0.25ex\right\vert}}

\begin{document}

\title{Random band matrices in the delocalized phase, III: Averaging fluctuations}
{\let\thefootnote\relax\footnotetext{\noindent 
The work of J. Yin is partially supported by the NSF grant DMS-1552192.}}
\vspace{0.4cm}
%
%
\begin{minipage}[b]{0.5\textwidth}
\author{F. Yang }
\address{University of Pennsylvania\\
  fyang75@wharton.upenn.edu}
 \end{minipage}
\begin{minipage}[b]{0.5\textwidth}
 \author{J. Yin}
\address{University of California, Los Angeles\\
    jyin@math.ucla.edu}
 \end{minipage}

\begin{abstract}
 We consider a general class of symmetric or Hermitian random band matrices $H=(h_{xy})_{x,y \in \llbracket 1,N\rrbracket^d}$ in any dimension $d\ge 1$, where the entries are independent, centered random variables with variances $s_{xy}=\mathbb E|h_{xy}|^2$. We assume that $s_{xy}$ vanishes if $|x-y|$ exceeds the band width $W$, and we are interested in the mesoscopic scale with $1\ll W\ll N$.
Define the {\it{generalized resolvent}} of $H$ as $G(H,Z):=(H - Z)^{-1}$, where $Z$ is a deterministic diagonal matrix with entries $Z_{xx}\in  \mathbb C_+$ for all $x$. Then we establish a precise high-probability bound on certain averages of polynomials of the resolvent entries. 
As an application of this fluctuation averaging result, we give a self-contained proof for the delocalization of random band matrices in dimensions $d\ge 2$. More precisely, for any fixed $d\ge 2$, we prove that the bulk eigenvectors of $H$ are delocalized in certain averaged sense if $N\le W^{1+\frac{d}{2}}$. This improves the corresponding results in \cite{HeMa2018} under the assumption $N\ll W^{1+\frac{d}{d+1}}$, and in \cite{ErdKno2013,ErdKno2011} under the assumption $N\ll W^{1+\frac{d}{6}}$. For 1D random band matrices, our fluctuation averaging result was used in \cite{PartII,PartI} to prove the delocalization conjecture and bulk universality for random band matrices with $N\ll W^{4/3}$. 

\end{abstract}

\tableofcontents

\section{Introduction }

\subsection{Random band matrices.}\ 
Random band matrices $H=(h_{xy})_{x,y\in \Gamma}$ model interacting quantum systems on a large finite graph $\Gamma$ of scale $N$ with random transition amplitudes effective up to scale of order $W\ll N$. More precisely, we consider random band matrix ensembles with entries being centered and independent up to the symmetry condition $h_{xy}=\overline h_{yx}$. The variance $s_{xy}:=\mathbb E|h_{xy}|^2$ typically decays with the distance between $x$ and $y$ on a characteristic length scale $W$, called the {\it band width} of $H$. For the simplest one-dimensional model with graph $\Gamma=\{1,2,\cdots, N\}$ and $h_{xy}=0$ for $|x-y|\ge W$, we have a band matrix in the usual sense that only the matrix entries in a narrow band of width $2W$ around the diagonal can be nonzero. In particular, if $W=N/2$ and all the variances are equal, we recover the famous Wigner matrix ensemble, which corresponds to a mean-field model. 

In this paper, we consider the case where $\Gamma$ is a $d$-dimensional torus $\Z_N^d:=\{1,2, \cdots, N\}^d$ with $d\ge 1$, so that the dimension of the matrix is $N^d$ (with an arbitrary ordering of the lattice points). Typically, we take the band width $W$ to be of mesoscopic scale $1\ll W\ll N$. The band structure is imposed by requiring that the variance profile is given by 
\be\label{fxy}
s_{xy} = \frac{1}{W^d}f\left( \frac{x-y}{W}\right) \quad \text{with } \quad \sum_{x}s_{xy}=\sum_{y}s_{xy}=1,
\ee
for some non-negative symmetric functions $f$ that decays sufficiently fast at infinity. As $W$ varies, the random band matrices naturally interpolate between two classes of quantum systems: the random Schr\"odinger operator with short range transitions such as the Anderson model \cite{Anderson}, and mean-field random matrices such as Wigner matrices \cite{Wigner}. A basic conjecture about random band matrices is that a sharp {\it Anderson metal-insulator phase transition} occurs at some critical band width $W_c$. More precisely, the eigenvectors of band matrices satisfy a localization-delocalization transition in the bulk of the spectrum \cite{ConJ-Ref1, ConJ-Ref2,ConJ-Ref6}, with a corresponding sharp transition for the eigenvalues distribution \cite{ConJ-Ref4}:
\begin{itemize}
\item for  $W \gg W_c$, delocalization of eigenstates (i.e.\;conductor phase) and Gaussian orthogonal/unitary ensemble (GOE/GUE) spectral statistics hold;
\item  for  $W \ll W_c$, localization of eigenstates (i.e.\;insulator phase) holds and the eigenvalues converge to a Poisson point process. 
\end{itemize}
Based on numerics \cite{ConJ-Ref1, ConJ-Ref2} and nonrigorous supersymmetric calculations \cite{fy}, the transition is conjectured to occur at $W_c\sim \sqrt N$ in $d=1$ dimension. In higher dimensions, the critical band width is expected to be $W_c\sim \sqrt{\log N}$ in $d=2$ and $W_c=\OO(1)$ in $d\ge 3$. For more details about the conjectures, we refer the reader to \cite{Sch2009,Spencer1,Spencer2}. The above features make random band matrices particularly attractive from the physical point of view as a model to study large quantum systems of high complexity.  

So far, there have been many partial results concerning the localization-delocalization conjecture for band matrices. In $d=1$ and for general distribution of the matrix entries, localization of eigenvectors was first proved for $W\ll N^{1/8}$ \cite{Sch2009}, and later improved to $W\ll N^{1/7}$ for band matrices with Gaussian entries \cite{PelSchShaSod}. The Green's function was controlled
down to the scale $\im z\gg W^{-d}$ in \cite{Semicircle, Bulk_generalized}, implying a lower bound of order $W$ for the localization length of all eigenvectors. For 1D random band matrices with general distributed entries, the {\it weak} delocalization of eigenvectors in some {\it averaged sense} (see the definition in Theorem \ref{comp_delocal}) was proved under $W\gg N^{6/7}$ in  \cite{ErdKno2013}, $W\gg N^{4/5}$  in \cite{delocal}, and $W\gg N^{7/9}$ in \cite{HeMa2018}. The strong delocalization and bulk universality for 1D random band matrices was first rigorously proved in \cite{BouErdYauYin2017} for $W=\Omega(N)$. {In the series \cite{PartI}, \cite{PartII} and this paper, we relax the condition on band width to $W\gg N^{3/4}$. In particular, the main results were stated as Theorems 1.2-1.5 in \cite{PartI}, and this paper contains the last piece of the proof. We refer the reader to Section \ref{sec relation} for more details.}
We mention also that at the edge of the spectrum, the transition of the eigenvalue statistics for 1D band matrices at the critical band width $W_c\sim N^{5/6}$ was understood in \cite{Sod2010}, thanks to the method of moments. For a special class of random band matrices, whose entries are Gaussian with some specific covariance profile, some powerful supersymmetry techniques can be used (see \cite{Efe1997,Spencer2} for  overviews). With this method, precise estimates on the density of states  \cite{DisPinSpe2002} were first obtained for $d=3$. Then random matrix local spectral statistics were proved for $W=\Omega(N)$ \cite{Sch2014}, and delocalization was obtained for all eigenvectors when $W\gg N^{6/7}$ and the first four moments of the matrix entries match the Gaussian ones \cite{BaoErd2015} (these results assume complex entries and hold in part of the bulk). Moreover, a transition at the critical band width $W_c\sim N^{1/2}$ was proved in \cite{SchMT,Sch1,Sch2,1Dchara}, concerning the second order correlation correlation function of bulk eigenvalues.

The purpose of this paper is two-fold. {First, we will complete the proof of the strong delocalization and bulk universality for 1D random band matrices under the condition $W\gg N^{3/4}$ together with \cite{PartI, PartII}. This is also the main goal of this series of papers; see Section \ref{sec relation} below for more detailed discussions. More precisely, in this paper we will develop a novel graphical scheme for the fluctuation averaging estimates on the generalized resolvents to complete the proof of local law in \cite{PartII} under $W\gg N^{3/4}$, which is further used in \cite{PartI} for the proofs of the main results.}
Second, using the same fluctuation averaging estimate, we shall give a self-contained proof for the weak delocalization of random band matrices in dimensions $d\ge2$ under the assumption $W\gg N^{\frac{2}{d+2}}$. This kind of weak delocalization of bulk eigenvectors was proved under the assumptions $W\gg N^{\frac{6}{d+6}}$ in \cite{ErdKno2013,ErdKno2011}, $W\gg N^{\frac{d+2}{2d+2}}$ in \cite{delocal}, and $W\gg N^{\frac{d+1}{2d+1}}$ in \cite{HeMa2018}. (In \cite{delocal}, the authors claimed they can prove the weak delocalization under the condition $W\gg N^{\frac{4}{d+4}}$, which turns out to be wrong as pointed out in \cite{HeMa2018}.) One can see that our results strictly improve these previous {results}. Moreover, {as observed in \cite{ErdKno2013,ErdKno2011} the exponent $(d/6+1)^{-1}$ is closer to being sharp ($W_c=\OO(N^0)$ for $d\ge 3$) when $d$ increases, while our result improves the coefficient $1/6$ to $1/2$.} 
We remark that our proof can be also applied to 1D band matrix and gives a weak delocalization of the eigenvectors under $W\gg N^{\frac{3}{4}}$, however it is strictly weaker than the result in \cite{PartI}, where the {\it strong} delocalization of the bulk eigenvectors was proved under the same assumption. 

\subsection{Averaging fluctuations.}\ \label{sec_model0}
 In this subsection, we give the informal statements of the main results of this paper, and explain why we need a decent fluctuation averaging bound. 
One main result is the following weak delocalization of random band matrices in dimensions $d\ge2$. 

\begin{theorem}[Informal statement of Theorem \ref{comp_delocal}]\label{main informal1}
If the band width satisfies $N\ll W^{1+\frac{d}{2}}$, then most of the bulk eigenvectors cannot be localized sub-exponentially on any scale $l\ll N$.
\end{theorem}
 

Our basic tool for the proof of Theorem \ref{main informal1} is the resolvent (Green's function) defined as
\be\label{greens}G(z)=(H-z)^{-1},\quad z\in \mathbb C_+:=\{z\in \mathbb C: \im z>0\}.\ee
The Green's function was shown to satisfy that for any fixed $\e>0$,
\be\label{intro_semi}
\max_{x,y}|G_{xy}(z)-m(z)\delta_{xy}| \le \frac{W^\e}{\sqrt{W^d \eta}} ,\quad z= E+ \ii \eta,
\ee
with high probability for all $\eta\gg W^{-d}$ in \cite{Semicircle, Bulk_generalized} (see Theorem \ref{semicircle}), where $m$ is the Stieltjes transform of Wigner's semicircle law
\be\label{msc}
m(z):=\frac{-z+\sqrt{z^2-4}}{2} = \frac{1}{2\pi}\int_{-2}^2 \frac{\sqrt{4-\xi^2}}{\xi-z}\dd\xi,\quad z\in \C_+.
\ee
The bound \eqref{intro_semi} already implies a lower bound of order $W$ for the localization length, but is not strong enough to give the delocalization on any scale larger than $W$. The bound we need is that for any scale $W\ll l \ll N$, 
\be\label{intro_aver} \max_y \Big(\eta \sum_{x:|x-y|\le l} |G_{xy}(z)|^2\Big)=\oo(1) \quad \text{ with high probability}\ee
for some $W^{-d} \ll \eta \ll 1$. To get an improvement over the estimate \eqref{intro_semi}, as in \cite{delocal,HeMa2018}, we introduce the so-called $T$-matrix, whose entries 
\be \label{def: T}
T_{xy}:=\sum_{w}s_{xw}|G_{wy}|^2, \quad s_{xw}:=\mathbb E|h_{xw}|^2,
\ee
are local averages of $|G_{xy}|^2$. The importance of $T$ lies in the following facts: 
\begin{itemize}
\item[(i)] by a self-consistent equation estimate (see Lemma \ref{bneng}), we can bound 
\be\label{intro_Tself}
\max_{x,y}|G_{xy}-m\delta_{xy}|^2 \le N^\e \max_{x,y}T_{xy} 
\ee
with high probability for any small constant $\e>0$;

\item[(ii)] for any scale $l\gg W$, we can bound 
\be\label{intro_Tself2}
\sum_{x:|x-y|\le l}|G_{xy}|^2 \le \sum_{x:|x-y|\le l+\OO(W)}T_{xy}.
\ee

\end{itemize}

The key tool to estimate the $T$-matrix is a self-bounded equation for $T$ (see \eqref{Tequation}):
\be\label{T_equation}
T_{xy}=(1+\OO(1))\left(\frac{|m|^2S}{1-|m|^2S}\right)_{xy} +\sum_{w\ne y}\left(\frac{|m|^2S}{1-|m|^2S}\right) _{xw}\left(  |G_{wy}|^2-|m|^2T_{wy} \right),
\end{equation}
where $S=(s_{xy})$ in the matrix of variances. One observation of \cite{delocal} is that the behavior of $T$ is essentially given by 
\be\label{diffu}
\Theta:=\frac{|m|^2S}{1-|m|^2S}=\sum_{k=1}^\infty |m|^{2k}S^{k},
\ee
and the second term in \eqref{T_equation} can be regarded an error under proper assumptions on $W$ and $\eta$. By the translation invariance of $S$ in \eqref{fxy}, $S^k$ can be understood through a $k$-step random walk $\sum_{i=1}^k X_i$ on the torus $\{1,2,\cdots, N\}^d$ with single step distribution $\mathbb P(X_1 = y-x)=s_{xy}$. 
Also with
$$|m(z)|^2= 1- \alpha \eta + \OO(\eta^2), \quad z= E+\ii\eta, \quad \alpha:= 4/{\sqrt{4-E^2}},$$ 
we only need to keep the terms with $k=\OO(\eta^{-1})$ in \eqref{diffu}.
Then there is a natural threshold at $\eta_c=W^2/N^2$. For $k\ll \eta_c^{-1}$, the $k$-step random walk is almost a random walk on the free space $\Z^d$ without boundary, and behaves diffusively by CLT.  This consideration gives the following diffusion approximation of $T$ (see Appendix \ref{appd}):
\be\label{T_behav}
T_{xy}(z)\approx \Theta_{xy}(z) \sim \frac{1}{W^d + W^2 |x-y|^{d-2}}, \quad \text{for}\quad \eta=\im z\ge \frac{W^2}{N^2}.
\ee

The main part of the proof is to estimate the error term in \eqref{T_equation}. It turns out that for $x\ne y$, $T_{xy}\approx |m|^2 \mathbb E_x |G_{xy}|^2$, where $\E_x$ is the partial expectation with respect to the $x$-th row and column of $H$. 
Hence the error term is approximately a sum over 
fluctuations: $\sum_{w}\Theta_{xw} Q_w |G_{wy}|^2$, where $Q_w:=1-\mathbb E_w$. The main difficulty is that $|G_{xy}|^2$ and $|G_{x'y}|^2$ for $x\ne x'$ are not independent; actually they are strongly correlated for $\eta\ll 1$. Estimating the high moments of these sums requires an unwrapping of the hierarchical correlation structure among many resolvent entries. In \cite{delocal} and \cite{HeMa2018}, the authors adopted different strategies. For the proof in \cite{delocal}, a so-called {\it fluctuation averaging mechanism} in \cite{EKY_Average} was used. 
It relies on intricate resolvent expansions to explore the cancellation mechanism in sums of monomials of $G$ entries. In \cite{HeMa2018}, however, the authors performed a careful analysis of the error term in Fourier space, where certain {\it cumulant expansions} are used. In this paper, we will follow the line of \cite{EKY_Average,delocal} and prove a much finer fluctuation averaging estimate as we will outline below.

Let $b_x$ be any sequence of deterministic coefficients of order $\OO(1)$. 
Suppose we have some initial (rough) estimates on the $G$ entries: for some constant $\delta>0$ and deterministic parameters $\Phi $ and $ \Gamma$, we have
\be\label{initial}
\max_{x,y}|G_{xy}-\delta_{xy} m | \le \Phi , \quad  \max_y \sum_{ x\in \Z_N^d} \left(|G_{xy}|^2+|G_{yx}|^2\right) \le \Gamma^2 , \quad W^{-d/2} \le \Phi \le N^{-\delta}, \quad \Gamma\ge 1,
\ee
with high probability. The state of the art fluctuation averaging estimate was proved in \cite{EKY_Average}:
\be\label{rough1}
  \Big| \sum_{x:x\ne y} b_x \left(  |G_{xy}|^2-|m|^2T_{xy} \right)\Big| \le N^\epsilon\left(N^{d/2}\Phi^2 + N^d\Phi^4\right), \quad \eta \gg W^{-d},
\ee
with high probability for any constant $\e>0$. {In this paper, we prove the following stronger fluctuation averaging estimate, which is another main result of this paper. 

\begin{theorem}[Informal statement of Theorem \ref{YEniu}]\label{main informal2}
Under \eqref{initial}, we have that with high probability, 
\be\label{rough3}
 \Big| \sum_{x:x\ne y} b_x \left(  |G_{xy}|^2-|m|^2T_{xy} \right) \Big| \le N^\epsilon\left(1 +  \Gamma^{2}\Phi^2\right), \quad \eta \gg W^{-d}.
\ee
It also holds for a more general type of resolvents, called generalized resolvents, see Definition \ref{defn Genr}.
\end{theorem}
\nc Recall that $\Gamma^2\sim \eta^{-1}$ by the following Ward's identity for the resolvent entries:
\be\label{Ward}
\sum_{x}|G_{xy}|^2  = \sum_{x}|G_{yx}|^2  = \frac{\im G_{yy}}{\eta}, \quad z= E+ \ii \eta,
\ee
which can be proved using the spectral decomposition of $G$. Together with the initial input $\Phi=W^\e(W^d\eta)^{-1/2}$ by \eqref{intro_semi}, it is obvious that \eqref{rough3} is better than \eqref{rough1} by a factor of $W^d/N^d$. 
{ For 1D random band matrices, the gaining of the $W/N$ factor is essential to reduce the band width to $W\gg N^{3/4}$ in \cite{PartI}. We remark that Theorem \ref{YEniu} was stated as Lemma 2.14 in \cite{PartII}, but the full proof was not given there.} 
{
We refer the reader to Section \ref{sec relation} below for more detailed discussions. For the application to 1D random band matrices in \cite{PartI, PartII}, we shall use a slightly more general band matrix model, and a more general type of resolvent, called the generalized resolvent, which is an extension of the regular resolvent defined in \eqref{greens}. The notations will be introduced in Section \ref{sec_model}. }

On the other hand, for $d\ge 2$ \eqref{rough3} allows us to establish the weak delocalization of random band matrices under the assumption $W\gg N^{\frac{2}{d+2}}$ as we shall explain now. 
With $\max_{x,y}\Theta_{xy}=\OO(W^{-d})$ by \eqref{T_behav}, we obtain from \eqref{rough3} and \eqref{intro_Tself} that with high probability,
\be\label{fine1}
\max_{x,y}|G_{xy}-m\delta_{xy}|^2 \le N^{\e/2}T_{xy} \le \frac{N^\e}{W^d} \left(1 +  \eta^{-1}\Phi^2\right).
\ee
If we assume $W\gg N^{\frac2{d+2}}$, then $W^d\eta \gg 1$ for $\eta=W^2/N^2$ and the above bootstrapping estimate gives an improved estimate 
\be\label{intro_entry}
\max_{x,y}|G_{xy}-m\delta_{xy}|^2 \le \frac{N^\e}{W^d} \quad \text{with high probability. }
\ee
It seems that this estimate is still not good enough to conclude \eqref{intro_aver}. 
However, using \eqref{T_equation} and \eqref{T_behav}, we can obtain that for some sequence of deterministic coefficients $b_w$ of order $\OO(1)$,
$$\eta \sum_{x:|x-y|\le l}T_{xy} = \eta\frac{l^2}{W^2}\sum_{w} b_w Q_w |G_{wy}|^2\le \frac{N^\e l^{2}}{W^2} \left(\eta + \Phi^2\right)=\oo(1),$$
if we take $\eta=W^2/N^2$ and $\Phi^2=N^\e W^{-d}$ by \eqref{intro_entry}. This gives \eqref{intro_aver} by \eqref{intro_Tself2}, which implies the delocalization on any scale $l\ll N$, and hence concludes Theorem \ref{main informal1}.

{
The starting point of the proof for Theorem \ref{main informal2} is the same as the one in \cite{EKY_Average}, that is, we try to bound the high moments of the left-hand side of \eqref{rough3}, and use a graphical tool to organize the calculations. Here the indices of the resolvents are the vertices of the graphs and the $G$ entries are represented by the edges between vertices. However, some key new ideas are needed in order to improve \eqref{rough1} to \eqref{rough3}.  Notice that there are two natural scales for the band model---the global scale, $N$, and the local scale, $W$. Correspondingly, we observe a two-level structure of the graphs, that is, a global level structure, on which the distances between vertices are as large as $N$, plus many local level components, in which the distances between vertices are at most of order $W$. 
In particular, we find that for our model, while the local structures can be handled in similar ways as \cite{EKY_Average}, the global structures cause a lot of trouble. One subtle issue is that for the local structures the vertices can be summed in arbitrary ways, but for the global structure the summation order of the vertices matters a lot. To handle this issue, we define a novel graphical property, called the nested property, and show that the nested order gives the correct order for the summation over the vertices. Conversely, in order to preserve the nested property, we need to perform the graph operations in a specific order, which makes the local structures to be also harder to deal with than \cite{EKY_Average}. For more details of the main ideas, we refer the reader to Section \ref{sec_idea}.} 


\subsection{Relation with \cite{PartI} and \cite{PartII}.}\label{sec relation}\ 
This paper is the third part of a series of papers with \cite{PartI} and \cite{PartII} being the first two parts.
The main goal of this series is to prove Theorem \ref{mainthmstrong} below, which was listed as Theorems 1.2-1.5 in \cite{PartI}, and is completely proved by combining the key ingredients provided by each part of the series. 
More precisely, Theorem \ref{mainthmstrong} was proved in \cite{PartI} using the mean-field reduction method introduced in \cite{BouErdYauYin2017} and a novel quantum unique ergodicity estimate. Both the proofs for the main result and the quantum unique ergodicity estimate are based on a generalized resolvent estimate, that is, Theorem 4.5 of \cite{PartI}. This estimate was partially proved as Theorem 1.4 in \cite{PartII}. In fact, a full proof was only given in \cite{PartII} under the condition $W\gg N^{6/7}$ based on a weak fluctuation averaging estimate, Lemma 2.8 of \cite{PartII}. In order to relax the condition to $W\gg N^{3/4}$, one needs a stronger fluctuation averaging estimate, which is provided by Theorem \ref{YEniu} of this paper. Combining Theorem \ref{YEniu} and the arguments in \cite{PartII}, we are able to conclude the proof of Theorem 4.5 in \cite{PartI}, and hence fill in the last piece of the whole proof.   

As explained in Remark 2.13 of \cite{PartII}, without using the result of this paper, it is possible to prove the main result under $W\gg N^{4/5}$ using the methods in \cite{EKY_Average}, although a full proof was not given there because our setting is a little different from the one in \cite{EKY_Average} and checking all the necessary details will be rather lengthy. 
An alternative approach to regular resolvent estimate was developed in \cite{HeMa2018} under $W\gg N^{7/9}$ based on a Fourier space analysis method instead of the fluctuation averaging mechanism. But it is not clear whether the method can be extended to generalized resolvent defined in \eqref{defGzetag}. 

In addition, we remark that this paper is not just a ``supplement" to \cite{PartI} and \cite{PartII}. 
We believe that the new ideas and techniques developed for proving Theorem \ref{YEniu} (see Section \ref{sec_idea}) will be useful in   the study of other types of non-mean-field random matrices. To illustrate the applicability of Theorem \ref{YEniu}, we apply it to random band matrices in dimensions $d\ge 2$, 
and give a self-contained proof for delocalization of bulk eigenvectors under $W \gg N^{\frac{2}{d+2}}$ in Theorem \ref{comp_delocal}. This result is completely independent from \cite{PartI} and \cite{PartII}. 

\vspace{5pt}

The rest of this paper is organized as follows. In Section \ref{sec_main}, we introduce our model and present the main results Theorem \ref{comp_delocal} and Theorem \ref{YEniu}. With Theorem \ref{YEniu}, we prove the weak delocalization of random band matrices, Theorem \ref{comp_delocal}, in dimensions $d\ge 2$ in Section \ref{sec mainaddpf}. On the other hand, the proof of Theorem \ref{YEniu} is mainly based on two averaging fluctuation lemmas---Lemma \ref{Ppart} and Lemma \ref{Qpart}.  In Section \ref{sec_tool}, we introduce the notations and collect some tools that will be used in the proof of Lemma \ref{Ppart} and Lemma \ref{Qpart}. In Section \ref{sec_Ppart}, we reduce Lemma \ref{Ppart} to another averaging fluctuation lemma, i.e.\;Lemma \ref{Q1}, which has a similar form as Lemma \ref{Qpart}. Sections \ref{sec_graph1} and \ref{sec_graph2} consist of the main proof for Lemma \ref{Q1} and Lemma \ref{Qpart}.


\vspace{5pt}

\noindent{\bf Conventions.} The fundamental large parameter is $N$ and we regard $W$ as a parameter depending on $N$. All quantities that are not explicitly constant may depend on $N$, and we usually omit $N$ from our notations. We use $C$ to denote a generic large positive constant, which may depend on fixed parameters and whose value may change from one line to the next. Similarly, we use $c$, $\epsilon$ or $\delta$ to denote a generic small positive constant. If a constant depend on a quantity $a$, we use $C_a$ or $c_a$ to indicate this dependence. Also, in the lemmas and theorems of this paper, we often use notations $\tau,D$ when we want to state that the conclusions hold for {\it{any fixed}} small constant $\tau>0$ and large constant $D>0$. For two quantities $A_N$ and $B_N>0$ depending on $N$, we use the notations $A_N = \OO(B_N)$ and $A_N \sim B_N$ to mean $|A_N| \le CB_N$ and $C^{-1}B_N \le |A_N| \le CB_N$, respectively, for some constant $C>0$. We use $A_N=\oo(B_N)$ to mean $|A_N| \le c_N B_N$ for some positive sequence $c_N \downarrow 0$ as $N\to \infty$. 
For any matrix $A$, we use the notations 
$$\|A\|:=\|A\|_{l^2 \to l^2}, \quad \|A\|_{\max} := \max_{i,j}|A_{ij}|, \quad \|A\|_{\min} := \min_{i,j}|A_{ij}|.$$ 
In particular, for a vector $\bf v$, we shall also use the notation $\|\bf v\|_\infty \equiv  \|\bf v\|_{\max}$. 

\vspace{5pt}

\noindent{\bf Acknowledgements.} The second author would like to thank Benedek Valk{\'o} and L. Fu for fruitful discussions and valuable suggestions.


 \section{Main results}\label{sec_main}
 
 \subsection{The model.}\label{sec_model} \
All the results in this paper apply to both real symmetric and complex Hermitian random band matrices. For the definiteness of notations, we only consider the real symmetric case. 
We always assume that $N,W$ are integers satisfying 
\begin{equation}\label{WN}
N^c \le W \le N
\end{equation}
for some constant $c>0$. Moreover, all the statements in this paper only hold for sufficiently large $N$ and we will not repeat it everywhere.

We define the $d$-dimensional discrete torus  
$$ \Z_N^d:=\left(\Z \cap (-N/2, N/2]\right)^d.$$
The $d$-dimensional random band matrix is indexed by the lattice points, where we fix an arbitrary ordering of $\Z_N^d$. For any $x\in \Z^d$, we always identify it with its canonical representative 
\be\label{cani}
[x]_N: = (x + N\bZ^d) \cap \Z_N^d.
\ee
Moreover, for simplicity, we will always use the $l^\infty$ norm on $\Z_N^d$ lattice:
\be\label{simplei}
|x-y| \equiv |[x-y]_N|: = \max_{1\le i \le d} \left|[x_i-y_i]_N\right|, \quad x,y\in \bZ_N^d.
\ee


Keeping the application of our results to \cite{PartI,PartII} in mind, we shall use a slightly more general model than the one in the introduction. Let $H=(H_{xy})$ be an $N^d\times N^d$ real symmetric random matrix with centered matrix entries that are independent up to the symmetry constraint. We assume that that variances $\mathbb E |H_{xy}|^2 = s_{xy}$ satisfy
\be\label{bandcw1}
\frac{c_s}{W^d} \cdot {\bf 1} _{|x-y| \le c_sW} \le s_{xy}  \le \frac{C_s}{W^d} \cdot {\bf 1} _{|x-y| \le C_sW}, \quad x,y\in \bZ_N^d,
\ee
for some constants $c_s, C_s>0$. 
Then $H$ is a random band matrix with band width of order $W$. Moreover, up to a rescaling, we assume that
\be\label{bandcw2}
1-\zeta \le \sum_{y\in \Z_N^d} s_{xy} \le 1+ \zeta ,\quad x\in \Z_N^d,
\ee
for some $\zeta\in [0,1)$.  


\begin{assumption}[Band matrix $H_N$]\label{jyyuan} 
Let $ H\equiv H_N$ be an $N^d\times N^d$ real symmetric random matrix whose entries $(H_{xy}: x,y \in \Z_N^d)$ are independent random variables satisfying
\be\label{bandcw0}
\mathbb E H_{xy} = 0, \quad \E |H _{xy}|^2 = s_{xy}, \quad x , y \in \bZ_N^d,
\ee
where the variances $s_{xy}$ satisfy \eqref{bandcw1} and \eqref{bandcw2}. Then we say that $H$ is a random band matrix with (typical) bandwidth $W\equiv W_N$. Moreover, we define the $N^d \times N^d$ symmetric matrix of variances $S\equiv S^\zeta: = (s_{xy})_{x,y\in \Z_N^d}$.

We assume that the random variables $H_{xy}$ have arbitrarily high moments, in the sense that for any fixed $p\in \mathbb N$, there is a constant $\mu_p>0$ such that
\begin{equation}\label{high_moment}
\left(\mathbb E|H_{xy}|^p\right)^{1/p} \le \mu_p s_{xy}^{1/2},\quad x,y \in \Z_N^d,
\end{equation}
for all $N$. Our result in this paper will depend on the parameters $C_s$ and $\mu_p$, 
but we will not track the dependence on these parameters in the proof. 
\end{assumption}

An important type of band matrices satisfying the above assumptions is the periodic random band matrices studied in e.g.\;\cite{delocal, Semicircle, EKY_Average}, where the variances are given by \eqref{fxy}.

\begin{assumption}[Periodic band matrix $H_N$]\label{jyyuan2} 
We say that $H\equiv H_N$ is a periodic random band matrix with bandwidth $W\equiv W_N$ if it satisfies the Assumption \ref{jyyuan} and \eqref{fxy}.
\end{assumption}

Again for the applications in \cite{PartI,PartII}, we state our results for the following {\it{generalized resolvent}} (or {\it{generalized Green's function}}) of $H$. 


\begin{definition}[Generalized resolvent]\label{defn Genr}
Given a sequence of spectral parameters $z_x \in \mathbb C_+$, $x\in \Z_N^d$, we define the following generalized resolvent (or generalized Green's function)  $G( H,Z)$ as
\be\label{defGzetag}
 G(H,Z)=(H-Z)^{-1}, \quad Z_{xy}=\delta_{xy} z_x.
\ee
\end{definition}

If $z_x=z$ for all $x$, then we get the normal Green's function $G(z)$ as in \eqref{greens}.  
The key point of the generalized resolvent is the freedom to choose different $z_x$. In particular, the following choice of $Z$ is used in \cite{PartI,PartII} for 1D random band matrices: 
$$ z_x = z \cdot \mathbf 1_{1\le x \le W} + \tilde z \cdot \mathbf 1_{ x > W}$$ 
for some $z, \tilde z\in \mathbb C_+$ with $\im \tilde z < \im z$. 

\vspace{5pt}
For $z_x$'s with fixed imaginary parts, one can show that $(G_{xx}(H,Z))_{x\in \Z_N^d}$ satisfies asymptotically the following system of self-consistent equations for $(M_x)_{x\in \Z_N^d}\equiv (M_x(S,Z))_{x\in \Z_N^d}$:
\be\label{falvww}
M_x^{-1} =- z_x - \sum_{y\in \Z_N^d} s_{xy}M_y .
\ee
If $\zeta$ is small and the $z_x$'s are close to some $z\in \mathbb C_+$, then the above equations are perturbations of the self-consistent equation for $m(z)$ defined in \eqref{msc}: 
$$m^{-1}(z) = - z - m(z).$$
In particular, the following Lemma \ref{UE} shows that the solution $(M_x)_{x\in \Z_N^d}$ exists and is unique as long as $\zeta$ and $\max_x |z_x-z|$ are small enough. It is proved in Lemma 1.3 of \cite{PartII}.


\begin{lemma}\label{UE}
Suppose $z \in \mathbb C_+$ satisfies $|\re z|\le 2-\kappa$ and $|z|\le \kappa^{-1}$ for some (small) constant $\kappa>0$. Then there exist constants $c_0, C_0>0$ such that the following statements hold.
\begin{itemize}
\item (Existence) 
If  
\be\label{heiz}
 \zeta+\max_x |z_x - z| \le c_0, 
\ee
then there exists $(M_x(S,Z))_{x\in \Z_N^d}$ that solves \eqref{falvww} and satisfies 
\be\label{bony}
 \max_x\left|M_x(S,Z) - m(z)\right|\le C_0 \left(\zeta + \max_x |z_x - z|\,\right).
\ee


\item (Uniqueness) The solution $(M_x(S,Z))_{x\in \Z_N^d}$ is unique under \eqref{heiz} and the condition
\be\label{heiz2}
\max_x\left|M_x(S,Z)-m(z)\right|\le c_0.
 \ee
for parameters satisfying \eqref{heiz},  
\end{itemize}
 
\end{lemma}


In the rest of this paper, we always assume that \eqref{heiz} holds for sufficiently small $c_0>0$. In particular, for $z \in \mathbb C_+$ with $|\re z|\le 2-\kappa$ and $|z|\le \kappa^{-1}$, we have $\im m(z) \ge c$ for some constant $c>0$ depending on $\kappa$. Thus we can choose $c_0$ to be small enough such that 
\be\label{Im_lowbound}
\im M_x(S,Z) \ge c/2,\quad 1\le i \le N.
\ee 
Let $M\equiv M(S,Z)$ denote the diagonal matrix with entries $M_{xy}: = M_x\delta_{xy}$.  With \eqref{Im_lowbound}, we can get the following lemma, which was proved as Lemma 2.7 in \cite{PartII}.

\begin{lemma} \label{inversebound}
Suppose $z \in \mathbb C_+$ satisfies $|\re z|\le 2-\kappa$ and $|z|\le \kappa^{-1}$ for some (small) constant $\kappa>0$. Suppose \eqref{heiz} and \eqref{bony} hold for some small enough constant $c_0>0$. Then there exist constants $c_1, C_1>0$ such that 
\be\label{gbzz2}
\left \| (1-M^2S)^{-1} \right \|_{l^\infty\to l^\infty}< C_1,
\ee
and
\be\label{tianYz}
\left| \left[(1-M^2S)^{-1}\right]_{xy} -\delta_{xy}\right|\le
 \begin{cases}
 C_1W^{-d},\quad & {\rm  if} \quad |x-y|\le (\log N)^2W\\ 
 N^{-c_1\log N},\quad & {\rm  if} \quad |x-y|> (\log N)^2W
 \end{cases}\; .
 \ee
\end{lemma}

\subsection{The main results.}\
In this subsection, we state the main results of this paper, including the new averaging fluctuation estimate and some applications of it in random band matrices. We first give the results on the delocalization of bulk eigenvectors of random band matrices as in Assumption \ref{jyyuan2}.

\subsubsection{Weak delocalization of random band matrices}\label{sec mainadd}
For simplicity of presentation, we will use the following notion of stochastic domination, which was first introduced in \cite{EKY_Average} and subsequently used in many works on random matrix theory, such as \cite{isotropic,delocal,Semicircle,Anisotropic}. It simplifies the presentation of the results and their proofs by systematizing statements of the form ``$\xi$ is bounded by $\zeta$ with high probability up to a small power of $N$".

\begin{definition}[Stochastic domination]\label{stoch_domination}
(i) Let
\[\xi=\left(\xi^{(N)}(u):N\in\mathbb N, u\in U^{(N)}\right),\hskip 10pt \zeta=\left(\zeta^{(N)}(u):N\in\mathbb N, u\in U^{(N)}\right)\]
be two families of nonnegative random variables, where $U^{(N)}$ is a possibly $N$-dependent parameter set. We say $\xi$ is stochastically dominated by $\zeta$, uniformly in $u$, if for any fixed (small) $\epsilon>0$ and (large) $D>0$, 
\[\sup_{u\in U^{(N)}}\mathbb P\left[\xi^{(N)}(u)>N^\epsilon\zeta^{(N)}(u)\right]\le N^{-D}\]
for large enough $N\ge N_0(\epsilon, D)$, and we will use the notation $\xi\prec\zeta$. Throughout this paper, the stochastic domination will always be uniform in all parameters that are not explicitly fixed (such as matrix indices, and $z$ that takes values in some compact set). 
If for some complex family $\xi$ we have $|\xi|\prec\zeta$, then we will also write $\xi \prec \zeta$ or $\xi=\OO_\prec(\zeta)$. 

(ii) As a convention, for two {\it deterministic} nonnegative quantities $\xi$ and $\zeta$, we shall use $\xi\prec\zeta$ if and only if $\xi\le N^\tau \zeta$ for any constant $\tau>0$. 


(iii) We say an event $\Xi$ holds with high probability (w.h.p.) if for any constant $D>0$, $\mathbb P(\Xi)\ge 1- N^{-D}$ for large enough $N$. More generally, given an event $\Xi$, we say $\Omega_N$ holds $w.h.p.$ in $\Xi$ if for any fixed $D>0$,
 $$\P( \Xi\setminus \Omega_N)\le N^{-D}$$
for sufficiently large $N$.
\end{definition}


We denote the eigenvectors of $H$ by $\{\mathbf u_\alpha\}_{\alpha \in \Z_N^d}$, with entries $\mathbf u_\alpha(x)$, $x\in \Z_N^d$. For $l\in \mathbb N$, we define the characteristic function $P_{x,l}$ projecting onto the complement of the $l$-neighborhood of $x$,
$$P_{x,l}(y):=\mathbf 1(|y-x|\ge l).$$
Define the random subset of eigenvector indices through 
$$\mathcal A_{\e,\kappa, l}:=\left\{ \alpha: \lambda_\alpha \in I_\kappa, \sum_x |u_\alpha(x)|\|P_{x,l}\bu_{\al}\|\le \e \right\}, \quad I_\kappa:=(-2+\kappa, 2-\kappa). $$
Our {first main result} is the following delocalization of bulk eigenvectors for random band matrices in dimensions $d\ge 2$, which was referred to as ``complete delocalization" in \cite{delocal}.

\begin{theorem}[Complete delocalization of bulk eigenvectors]\label{comp_delocal}
Suppose the Assumption \ref{jyyuan2} holds and $d\ge 2$. Suppose 
\be\label{NW}
N\prec W^{1+\frac{d}{2}}. 
\ee
Fix any constants $\kappa>0$ and $c>0$. For any $l\le N^{1 - c}$, we have 
$$\frac{|\mathcal A_{\e,\kappa, l}|}{N^d} \le C\sqrt{\e} + \OO_\prec(N^{-2c})$$
for any $\e>0$.
\end{theorem}

\begin{remark}\label{rem_comp_delocal}
For any fixed $\gamma,K>0$, we define another random subset of eigenvector indices 
$$\mathcal B_{K,l} :=\left\{\alpha: \lambda_\alpha \in I_\kappa, \exists\, x_0\in \Z_N^d \text{ s.t. } \sum_x |\mathbf u_\alpha(x)|^2 \exp\left[\left(\frac{|x-x_0|}{l}\right)^\gamma\right] \le K \right\}. $$
Notice that the set $\mathcal B_{K,l}$ contains all indices associated with eigenvectors that are exponentially localized in balls of radius $\OO(l)$. In fact, by \cite[Corollary 3.4]{ErdKno2011}, Theorem \ref{comp_delocal} implies that 
$$\lim_{N\to \infty}\mathbb E\frac{|\mathcal B_{K,l}|}{N} = 0,$$
i.e. the fraction of eigenvectors localized sub-exponentially on scale $l$ vanishes with high probability for large $N$. This explains the name ``complete delocalization".
\end{remark}

\begin{remark}
Using resolvents of $H$, {an analogous result} was proved in \cite{delocal} under the condition $N\ll W^{1+\frac{d}{4}}$, which turns out to be wrong: it should be $N\ll W^{1+\frac{d}{d+2}}$ instead. This condition was improved to $N\ll W^{1+\frac{d}{d+1}}$ later in \cite{HeMa2018}. In fact, by studying the evolution operator $e^{-\ii Ht}$, the complete delocalization was proved under the condition $N\ll W^{1+\frac{d}{6}}$ in \cite{ErdKno2013,ErdKno2011}. Our result improves all these results. 
\end{remark}

\subsubsection{Strong delocalization and universality of 1d random band matrices}
For 1d random band matrices whose entries are close to a Gaussian in the four moment matching sense, the following version of strong delocalization of bulk eigenvectors was proved in \cite{BaoErd2015} under the assumption $W\gg N^{6/7}$: 
\be\label{strong delocal}
\max_{\al: \lambda_\al \in I_\kappa} \|\bu_\al\|_\infty \prec N^{-1/2}.
\ee
Our delocalization result as given by Theorem \ref{comp_delocal} is certainly a weaker version of that result in some averaged sense. Based on the new fluctuation averaging estimate of this paper, i.e. Theorem \ref{YEniu} below, the same strong delocalization and bulk universality was proved in \cite{PartI}  for 1d random band matrices under a weaker assumption $W\gg N^{3/4}$. {We remark that using the estimate \eqref{rough1} in \cite{EKY_Average}, \cite{PartI} can only give the strong delocalization under the assumption $W\gg N^{4/5}$. If, instead of using a fluctuation averaging estimate, we use the Fourier space analysis in \cite{HeMa2018}, then \cite{PartI} may give the strong delocalization under the assumption $W\gg N^{7/9}$ (although there are a lot of details to verify because \cite{HeMa2018} dealt with regular resolvents).} 

\begin{theorem}[Main result of the series \cite{PartI}, \cite{PartII} and this paper]\label{mainthmstrong}
Suppose the Assumption \ref{jyyuan2} holds and $d=1$. Suppose $W\ge N^{3/4+\e}$ for some constant $\e>0$. Then for any constant $\kappa>0$, the estimate \eqref{strong delocal} holds. Moreover, the bulk eigenvalue statistics converge to those of the GOE (real case) or GUE (complex case).  
\end{theorem}

 This result was proved as Theorem 1.2 and Theorem 1.4 in \cite{PartI} with the generalized resolvent estimate, Theorem 4.5, as a key input. The generalized resolvent estimate is proved by combining the arguments in \cite{PartII} with the fluctuation averaging estimate, Theorem \ref{YEniu}, below. More results were also proved in \cite{PartI}, including the local semicircle law of bulk eigenvalues and quantum unique ergodicity of bulk eigenvectors. The role of this paper in the whole series have been discussed in details in Section \ref{sec relation}. 
 
Next, we state the main  fluctuation averaging estimate of this paper, based on which we shall give a simple and self-contained proof of Theorem \ref{comp_delocal} in Section \ref{sec mainaddpf}. 

\subsubsection{Averaging fluctuations} 
Throughout the following discussion, we will abbreviate $G\equiv G(H,Z)$. Recall the $T$ variables defined in (\ref{def: T}). 
We add and subtract $\sum_{\al}S_{x\al}|M_\al|^2T_{\al y} $ 
so that      
$$
T_{xy} 
=\sum_{\al}S_{x\al}|M_\al|^2T_{\al y} +\sum_{\al}S_{x\al}\left( |G_{\al y}|^2-|M_\al|^2T_{\al y}\right),
$$
which immediately gives that
\begin{equation}\label{Tequation}
T_{xy}=\sum_{\al}\left[\left(1-S|M|^2\right)^{-1}S\right]_{x\al}\left(  |G_{\al y}|^2-|M_\al|^2T_{\al y} \right).
\end{equation}
Isolating the diagonal terms,  we can write the $T$-equation as 
\be\label{T0}
T_{xy}=T_{xy}^{\,0}+\sum_{\al: \al \ne y }\left[(1-S|M|^2)^{-1}S\right]_{x\al}\left(  |G_{\al y}|^2-|M_\al|^2T_{\al y} \right)
,
\ee
where
$$ T_{xy}^{\,0}:=\left[(1-S|M|^2)^{-1}S\right]_{xy}\left(  |G_{yy}|^2-|M_y|^2T_{yy} \right). $$

The {second main result} of this paper is the following fluctuation averaging estimate on the sum in \eqref{Tequation}. 
We introduce the notation  
 \begin{align}\label{tri} 
   \tnorm{G}^2(H,Z) := \max_y \sum_{ x\in \Z_N^d} \left(|G_{xy}|^2+|G_{yx}|^2\right).
    \end{align}


   
  


 \begin{theorem}[Averaging fluctuations] \label{YEniu} 
Fix any $z \in \mathbb C_+$ satisfying $|\re z|\le 2-\kappa$ and $|z|\le \kappa^{-1}$ for some constant $\kappa>0$. Suppose that Assumption \ref{jyyuan} holds.  Suppose that \eqref{heiz} holds for some sufficiently small constant $c_0>0$ (which implies that \eqref{bony}, \eqref{gbzz2} and \eqref{tianYz} hold). Assume that  
 \be\label{xiazhou}
 \min_x \left(\im z_x \right)\ge N^{-C_2} 
 \ee
 for some constant $C_2>0$. Let $\Phi$ and  $\Gamma$ be deterministic parameters satisfying 
 \be\label{GM1.5}
W^{-d/2} \le \Phi \le N^{-\delta}, \quad \Gamma\ge 1,
 \ee
for some constant $\delta>0$. If the following estimates hold,
 \begin{align}\label{GM1}
\|G - M\|_{\max} \prec \Phi, \quad  \tnorm{G}^2 \prec \Gamma^2 ,
 \end{align}
then for any deterministic sequence $\mathbf b=(b_x)_{x\in \Z_N^d}$ with $\|\mathbf b\|_\infty=\OO(1)$, we have 
\be\label{GM2}
  \max_y \Big|\sum_{x}b_{x}\left(  |G_{xy}|^2-|M_x|^2T_{xy} \right) \Big|\prec \Gamma^2\Phi^2 + 1.
\ee
\end{theorem}  

\begin{remark}
The above statements should be understood as follows. If \eqref{GM1} holds with probability $1-N^{-D}$ up to a factor $N^{\e}$, then there exists a constant $C>0$ independent of $\e$ and $D$ such that \eqref{GM2} holds with probability $1-N^{-D/C}$ up to a factor $N^{C\e}$.
In particular, Theorem \ref{YEniu} can only be applied for $\OO(1)$ many times. 
\end{remark}

The following notations have been used in the introduction. 

\begin{definition}[$P_k$ and $Q_x$]  \label{Ek}
We define $\E_x$ as the partial expectation with respect to the $x$-th row and column of $H$, i.e.,
$$\E_x(\cdot) := \E(\cdot|H^{[x]}),$$
where $H^{[x]}$ denotes the $(N^d-1)\times(N^d-1)$ minor of $H$ obtained by removing the $x$-th row and column (see Definition \ref{minors} for the general definition). For simplicity, we shall also use the notations 
$$P_x :=\E_x , \quad Q_x := 1-\E_x.$$
In the proof, we will follow the convention that $P_x(A)B\equiv [P_x (A)]B$ and $P_xAB\equiv P_x (AB)$, and similarly for $Q_x$.
 \end{definition}

\begin{proof}[Proof of Theorem \ref{YEniu}] 
We can use \eqref{GM1} to control the diagonal term by $\left||G_{yy}|^2-|M_y|^2T_{yy}\right| = \OO(1)$ with probability $1-\OO(N^{-D})$. 
Then it remains to control the off-diagonal terms. Fix any $y\in \Z_N^d$, and call it $\star$ which stands for a special index throughout the proof. We can write the off-diagonal terms as in \eqref{divide}. Then for the two terms on the right-hand side, we have the following two key lemmas. Note that by considering the real and imaginary parts separately, it suffices to assume that $b_x$'s are real.

\begin{lemma}\label{Ppart}
Suppose the assumptions of Theorem \ref{YEniu} hold, and $b_x$ are real deterministic coefficients satisfying $\max_x |b_x|=\OO( 1)$.
Then for any fixed (large) $p\in 2\mathbb N$ and (small) $\tau>0$, we have 
\be\label{eqn-Ppart}
 \mathbb E\Big| \sum_{x:x\ne \star} b_x\left(\mathbb E_x |G_{x \star}|^2-|M_x|^2T_{x \star}\right) \Big|^p
\le   \left[N^{\tau}\left(\Gamma^2\Phi^2 + 1\right)\right]^p 
\ee
for large enough $N$. 
\end{lemma}

\begin{lemma}\label{Qpart}
Suppose the assumptions of Theorem \ref{YEniu} hold, and $b_x$ are real deterministic coefficients satisfying $\max_x |b_x|=\OO( 1)$.
Then for any fixed (large) $p\in 2\N$ and (small) $\tau>0$, we have  
\be\label{eqn-Qpart}
\E\Big|\sum_{x:x\ne \star} b_x  Q_x |G_{x \star}|^2\Big|^{ p}   \le   \left[N^{\tau } \left(\Gamma^2\Phi^2 + 1 \right)\right]^{ p}
\ee
for large enough $N$.
\end{lemma}

With Lemma \ref{Ppart} and Lemma \ref{Qpart}, using Markov's inequality we can prove \eqref{GM2}.
\end{proof}

\subsection{Proof of Theorem \ref{comp_delocal}.}\label{sec mainaddpf}\ 
In the proof, we shall use tacitly the following basic properties of stochastic domination $\prec$. 

\begin{lemma}[Lemma 3.2 in \cite{isotropic}]\label{lem_stodomin}
Let $\xi$ and $\zeta$ be two families of nonnegative random variables. Let $C>0$ be any constant. 

(i) Suppose that $\xi (u,v)\prec \zeta(u,v)$ uniformly in $u\in U$ and $v\in V$. If $|V|\le N^C$, then $\sum_{v\in V} \xi(u,v) \prec \sum_{v\in V} \zeta(u,v)$ uniformly in $u$.

(ii) If $\xi_1 (u)\prec \zeta_1(u)$ and $\xi_2 (u)\prec \zeta_2(u)$ uniformly in $u\in U$, then $\xi_1(u)\xi_2(u) \prec \zeta_1(u)\zeta_2(u)$ uniformly in $u$.

(iii) Suppose that $\Psi(u)\ge N^{-C}$ is deterministic and $\xi(u)$ satisfies $\mathbb E\xi(u)^2 \le N^C$ for all $u$. Then if $\xi(u)\prec \Psi(u)$ uniformly in $u$, we have $\mathbb E\xi(u) \prec \Psi(u)$ uniformly in $u$.
\end{lemma}

Theorem \ref{comp_delocal} is a corollary of the following theorem, which gives estimates on the resolvent entries that are much finer than the one in Theorem \ref{semicircle}.

\begin{theorem}\label{main_thm}
Suppose the assumptions of Theorem \ref{comp_delocal} hold. 
Fix any constants $\kappa>0$ and $c>0$. Then for any fixed $\delta>0$, we have
\be\label{strong_semicircle}
\max_{x,y}|G_{xy}(z)-\delta_{xy}m(z)|  \prec W^{-d/2}
\ee
uniformly in $z\in \{z = E+\ii \eta \in \mathbf D(\kappa, \delta/2):\eta \ge W^{2+\delta}/{N^2}\}$. Moreover, for any $l\le N^{1-c}$,  
we have
\be\label{priori_weak}
\frac{\eta}{\im m}\sum_{x: |x-y| \le l} |G_{xy}(z)|^2 \prec N^{-2c}W^\delta,
\ee
for all $z=E+\ii \eta$ with $E\in (-2+\kappa, 2-\kappa)$ and $\eta={W^{2+\delta}}/{N^2}.$
\end{theorem}

Recall the Ward's identity \eqref{Ward} for the resolvent entries,
by \eqref{strong_semicircle}, we then have
$$ \frac{\eta}{\im m}\sum_{x} |G_{xy}|^2 = \frac{\im G_{yy}}{\im m} = 1+ \OO_\prec ( W^{-d/2})$$
for any $z\in \mathbf D(\kappa,\delta/2)$. Hence, for all $y$, $\left(\sqrt{\frac{\eta}{\im m}} G_{xy}\right)_{x\in \Z_N^d}$ is approximately a unit vector. The estimate \eqref{priori_weak} then means that this column vector of resolvent entries cannot be localized on any scale $l\ll N$. 

\begin{proof}[Proof of Theorem \ref{comp_delocal}]
Given \eqref{strong_semicircle} and \eqref{priori_weak} with $\delta>0$ being an arbitrarily small constant, the proof is exactly the same as the one for \cite[Proposition 7.1]{delocal}.
\end{proof}

It remains to prove Theorem \ref{main_thm}. We first record the following local law of $G(z)$ proved in \cite{Bulk_generalized,Semicircle}. It will serve as an a priori estimate for the proof of Theorem \ref{comp_delocal}. 
 
 \begin{theorem}[Local law]\label{semicircle}
Suppose the Assumption \ref{jyyuan2} holds. For any constants $\kappa, \delta >0$, we define the spectral domain
\be\label{spectral_D}
\mathbf D\equiv \mathbf D(\kappa,\delta):=\{z=E+\ii\eta: |E|\le 2-\kappa, \eta \ge W^{-d+\delta} \}.
\ee
Then the following local law holds uniformly in $z\in \mathbf D(\kappa,\delta)$:
\be\label{jxw}
\max_{x,y}|G_{xy}(z) - \delta_{xy}m(z)| \prec ( W^d\eta)^{-1/2}. 
 \ee 
\end{theorem}

The following lemma shows that the size of $\left\|G - M\right\|_{\max}^2$ is controlled by $\|T\|_{\max}$ with high probability.

  
\begin{lemma}[Lemma 2.1 of \cite{PartII}]\label{bneng} 
Suppose the assumptions of Theorem \ref{YEniu} hold. Suppose there is a probability set $\Omega$ such that
\be\label{cllo}
 \mathbf 1_{\Omega} \|G - M \|_{\max}\le N^{-\delta}, \quad  \mathbf 1_{\Omega}\|T\|_{\max}\le \Phi^2 , 
\ee 
for some constant $\delta>0$ and some deterministic parameter $W^{-d/2} \le \Phi \le N^{-\delta}$. 
Then for any fixed (small) $\tau>0$ and (large) $D>0$, 
\be\label{34}
\P\left({\bf 1}_{  \Omega} \|G - M\|_{\max}\ge N^\tau \Phi \right)\le N^{-D}.
\ee 
\end{lemma}

To bound the $T$-variables, we use the $T$-equation as in \eqref{T0}:
\be\label{T00}
T_{xy}=T_{xy}^{\,0}+ \sum_{w}^{(y)} \Theta_{xw} \left( |G_{w y}|^2-|m|^2T_{wy}\right) ,
\ee
where
\be\label{zlende}
\Theta_{xw}:=\left[(1-|m|^2S)^{-1}S\right]_{xw}=\OO_\prec\left( \frac1{N^d\eta}+\frac{1}{W^2 \langle x-w\rangle^{d-2}}\right).
\ee
For the second estimate, we prove it in Appendix \ref{appd}. 
Now we prove Theorem \ref{main_thm} using Theorem \ref{YEniu}. 

\begin{proof}[Proof of Theorem \ref{main_thm}]
By Theorem \ref{semicircle} and Ward's identity \eqref{Ward}, it is easy to see that \eqref{GM1} holds with
\be\label{priori1}
\Phi=(W^d\eta)^{-1/2}\; , \qquad \Gamma^2=\eta^{-1} \; ,
\ee
for any $z\in \mathbf D(\kappa, \delta/2)$. For the $T_{xy}^0$ in \eqref{T00}, we can bound it as
\be\label{Txy0}
T_{xy}^0=(|m|^2+\OO_\prec(\Phi))\Theta_{xy}. 
\ee
Then using \eqref{zlende}, \eqref{GM2} and \eqref{priori1}, we can bound \eqref{T00} as
$$T_{xy}\prec \left( \frac{1}{N^d\eta}+ \frac{1}{W^d}\right) \left( 1+ \frac{ \Phi^2 }{\eta}\right) .$$
For $ z = E+\ii \eta \in \mathbf D(\kappa, \delta/2)$ with $\eta \ge W^{2+\delta}/{N^2}$, the above estimate gives
$$T_{xy}\prec  W^{-d}  + W^{-\delta} \Phi^2$$
under the conditions $d\ge 2$ and \eqref{NW}. Together with Lemma \ref{bneng}, it implies the following self-improving estimate:
\be\label{Phi2}
\|G-m\|_{\max}\prec \Phi  \Rightarrow \|G-m\|_{\max}\prec W^{-d/2} + W^{-\delta/2}\Phi .
\ee
After $\OO(\delta^{-1})$ many iterations of \eqref{Phi2}, we can conclude \eqref{strong_semicircle}.

Then we prove \eqref{priori_weak}. We have
\begin{align} 
\sum_{x : |x-y|\le l} |G_{xy}|^2 = \sum_{x : |x-y|\le l} \sum_{w}s_{wx} |G_{xy}|^2  \le  \sum_{w : |w-y|\le l+C_s W} \sum_{x}s_{wx} |G_{xy}|^2 =  \sum_{x : |x-y|\le l+C_s W} T_{xy} \nonumber\\
\prec \max_{w}\Big(\sum_{x : |x-w|\le l+C_s W} \Theta_{xw}\Big)\left( 1 + \sum_{w:w\ne y} \wt b_w \left( |G_{w y}|^2-|m|^2T_{wy}\right) \right)  \label{fluc_aver0}
\end{align}
for some real deterministic coefficients $\wt b_w$ satisfying $\max_w |\wt b_w|=\OO(1)$. In the above derivations, we used \eqref{fxy} in the first step, \eqref{bandcw1} in the second step, the definition of $T$ variables \eqref{def: T} in the third step, and the $T$-equation \eqref{T00} in the last step. 
We can bound the sum in \eqref{fluc_aver0} with \eqref{GM2} and \eqref{priori1}. Also with \eqref{zlende}, it is easy to prove that
$$\max_{y}\Big(\sum_{x : |x-y|\le l+C_s W} \Theta_{xy}\Big) \prec \frac{l^d}{N^d \eta} + \frac{l^2}{W^2}.$$
  Thus for $z=E+\ii \eta$ with $E\in (-2+\kappa, 2-\kappa)$ and $\eta={W^{2+\delta}}/{N^2}$, we have
\begin{align*} 
\sum_{x : |x-y|\le l} |G_{xy}|^2 \prec \frac{l^2}{W^2}\left( 1 + \frac{\Phi^2}{\eta} \right) \prec  \frac{l^2}{W^2} + \frac{l^2}{W^{2+d}\eta},
\end{align*}
where we used \eqref{strong_semicircle} in the second step. Then using \eqref{NW}, we obtain that for $\eta={W^{2+\delta}}/{N^2}$,
$$\frac{\eta}{\im m}\sum_{x: |x-y| \le l} |G_{xy}(z)|^2 \prec \eta \frac{l^2}{W^2} + \frac{l^2}{W^{2+d}} \prec N^{-2c}W^{\delta}.$$
This proves \eqref{priori_weak}.
\end{proof}

\subsection{Basic ideas for the proof.} \label{sec_idea} \
In this subsection, we discuss the basic ideas for the proof of Lemma \ref{Ppart} and Lemma \ref{Qpart}. 
In the rest of this section, we focus on \eqref{eqn-Qpart}, while the proof for \eqref{eqn-Ppart} is actually easier.


We expand the left-hand side of \eqref{eqn-Qpart} as a sum of the products of $2p$ resolvent entries. In fact, keeping track of the correlations among all the resolvent entries in each large product is rather involved. For this purpose, a convenient graphical tool was developed in \cite{EKY_Average} to organize the calculation, where the indices are the vertices of the graphs and the resolvent entries are represented by the edges between vertices. Moreover, in this paper all the graphs are rooted graphs, with the root representing the $\star$ index. In this paper, we shall extend the arguments in \cite{EKY_Average} and develop a graphical representation with more structural details. Also as in \cite{EKY_Average}, estimating the high moments requires an unwrapping of the hierarchical correlation structure among several resolvent entries, which will be performed using resolvent expansions in Lemma \ref{lem_exp1} and Lemma \ref{lemmai}, such as  
\be\label{resol_examp}
G _{xy}=G_{xy}^{(\al)}+\frac{G_{x\al}G_{\al y}}{G_{\al\al}}, \quad  \al\notin \{x,y\}, \quad  \text{or} \quad G _{xy} = - G_{xx}\sum_{\al}H_{x\al}G^{(x)}_{\al y}, \quad x\ne y.
\ee
Here for any $a\in \Z_N^d$, $G^{(a)}$ denotes the resolvent of the $(N^d-1)\times (N^d-1)$ minor of $H$ obtained by removing the $a$-th row and column (see Definition \ref{minors} below). The resolvent expansions are represented by graph expansions, i.e. expanding a graph into a linear combination of several new graphs. For example, applying the first expansion in \eqref{resol_examp} to the $G_{xy}$ edge in a graph gives two new graphs, where in one of them the $G_{xy}$ edge is replaced by the two edges $G_{x\al}$ and $G_{\al y}$. For the second expansion in \eqref{resol_examp}, we will create a new vertex $\alpha$ in the graph, which is in the $W$-neighborhood of $x$.




 Comparing \eqref{rough1} with \eqref{rough3} or \eqref{eqn-Qpart}, one can notice that we essentially replace the $N^d\Phi^2$ factor with the $\Gamma^2$ factor in \eqref{rough3}. The origin of these two factors is as following. In the high moments calculation, terms like
\be\label{termsimple0}
\sum_{x} c_{x} G_{x y_1}G_{x y_2} \quad  \text{or} \quad \sum_{x} c_{x} G_{x y_1}\overline{G_{x y_2} } , \quad c_x=\OO(1),
\ee
will appear in the expressions. The authors in \cite{EKY_Average} bounded them by $N^d \Phi^2$, which is not good enough when we consider band matrices (although it is sharp for mean-field random matrices with $W=N$).
Instead, we shall use the better estimate 
\be\label{termsimple}
 \sum_{x} \left|G_{x y_1}G_{x y_2}\right| \prec  \Gamma^2
\ee
by the second estimate in \eqref{GM1} and Cauchy-Schwarz inequality. This is the very origin of the $\Gamma^2$ factor in \eqref{eqn-Qpart}. 
In the rest of this subsection, we discuss the main difficulties and the new ideas to resolve them. In particular, the graphical tool plays an essential role in our approach.  


\subsubsection{The nested property}\label{sec_nested}

In order to apply the bound $\Gamma^2$ to the expressions as in \eqref{termsimple}, the order of the summation is important. For example, using \eqref{termsimple} we can bound the following sum as 
\be\label{good_nest}
\begin{split}
 \sum_{x_1,x_2,x_3} \left|G_{x_1 \star} G_{x_1 x_2} G_{x_1 x_3} G_{x_2\star}^2 G_{x_2 x_3}\right| = \sum_{x_2}|G_{x_2\star}|^2  \sum_{x_1}\left|G_{x_1 \star} G_{x_1 x_2}   \right| \sum_{x_3}\left|G_{x_1 x_3} G_{x_2 x_3}\right| \\
 \prec \Gamma^2\sum_{x_2}|G_{x_2\star}|^2  \sum_{x_1}\left|G_{x_1 \star} G_{x_1x_2}   \right|  \prec  \Gamma^2\sum_{x_2}|G_{x_2\star}|^2  \prec   \Gamma^6.
\end{split}
\ee
However, in some cases, we may not be able to find such a summation order to get enough number of $\Gamma$ factors. For example, the following sum is also an average of the product of 6 resolvent entries, but we can only get
\be\label{bad_nest}
\begin{split}
\sum_{x_1,x_2,x_3} \left|G_{x_1 \star} G_{x_1x_2} G_{x_1 x_3} G_{x_2\star}G_{x_2x_3} G_{x_3\star}\right|= \sum_{x_3} \left|G_{x_3\star}\right| \sum_{x_2} \left|G_{x_2\star}G_{x_2x_3} \right|\sum_{x_1}\left|G_{x_1 \star} G_{x_1x_2} G_{x_1 x_3} \right| \\
 \prec   {\Gamma^2 \Phi} \sum_{x_3} \left|G_{x_3\star}\right| \sum_{x_2} \left|G_{x_2\star}G_{x_2x_3} \right|  \prec  {\Gamma^4 \Phi} \sum_{x_3} \left|G_{x_3\star}\right|  \prec   \Gamma^5 ( N^{d/2}\Phi ),
\end{split}
\ee
using \eqref{termsimple} and \eqref{GM1}, where one $\Gamma$ factor is replaced by a $N^{d/2}\Phi$ factor. (Note that we get the same bound if we sum over $x_2$ or $x_3$ first.) 
This example shows that in general, we are not guaranteed to get enough number of $\Gamma$ factors in the high moment estimate if the indices of some expression do not satisfy the following {\it well-nested property}. 
Given an average of certain product of resolvent entries over free indices $x_1, \ldots, x_p$, we shall say that these indices are {\it well-nested} if there exists a partial order $\star \preceq x_{i_1} \preceq \cdots \preceq {x_{i_p}} $ such that for each $1\le k \le p$, there exist at least two resolvent entries that have pairs of indices $(x_{i_k},x_{\alpha_k})$ and $(x_{i_k},x_{\beta_k})$ with $ x_{\alpha_k},x_{\beta_k} \preceq x_{i_k}$. (Here ``$\preceq$" means a partial order, not the stochastic domination.) Note that if the indices are well-nested, then one can sum according to the order $x_{i_p}\to \cdots \ldots \to x_{i_1}$ to get a $\Gamma^{2p}$ factor. 
In our proof, we always start with expressions with well-nested indices. However, after several resolvent expansions, it will be written as a linear combination of much more complicated averages of monomials of resolvent entries. It is often very hard to check that the indices in the new expressions are also well-nested. This is one of the main difficulties in our proof.

To resolve the above difficulty, we try to explore some property that guarantees well-nested summation indices and, at the same time, is robust under the resolvent expansions. In terms of the graphical language, the well-nested property of indices is translated into a structural property of the graphs, which we shall call the {\bf ordered nested property}. 
Suppose we want to estimate the $p$-th moment in \eqref{eqn-Qpart}. After some (necessary) resolvent expansions, we will have graphs containing vertices $\{x_1, \ldots, x_p, \star\}$. Roughly speaking, a graph $\mathcal G$ has {\it ordered nested property} if its vertices $\{x_1, \ldots, x_p, \star\}$ can be partially ordered in a way 
\be\label{order}
\star\preceq x_{i_1} \preceq x_{i_2} \preceq \cdots \preceq x_{i_p} 
\ee 
such that each of the vertex $x_{i_k}$, $1\le k\le p$, has at least two edges connecting to the {\it preceding atoms} (here we say $a$ precedes $b$ if $a\prec b$). For example, the left graph in Fig.\,\ref{fig1pdf} corresponding to \eqref{good_nest} has ordered nested property, while the right graph in Fig.\,\ref{fig1pdf} corresponding to \eqref{bad_nest} does not.

\begin{figure}[htb]
\centering
\includegraphics[width=11cm]{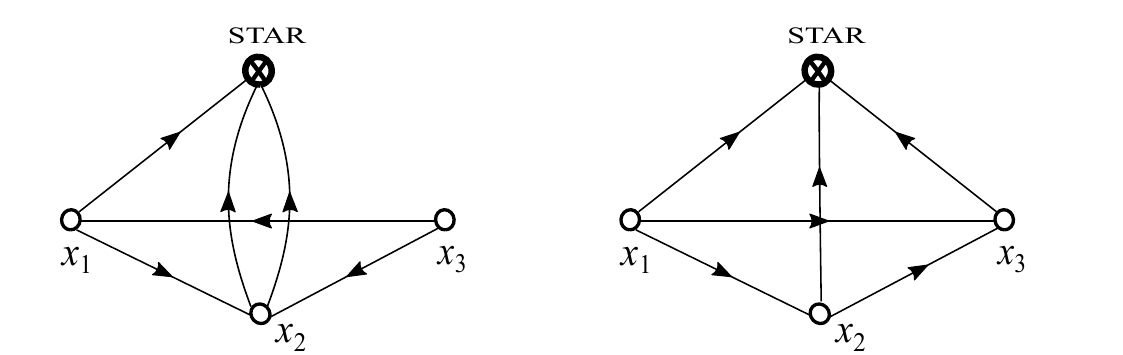}
\caption{The left graph represents \eqref{good_nest} and satisfies the ordered nested property with the order $\star \preceq x_2\preceq x_1 \preceq x_3$. The right graph represents \eqref{bad_nest} and does not satisfy the ordered nested property.}
\label{fig1pdf}
\end{figure}

Suppose a graph satisfies the ordered nested property with \eqref{order}, then one can sum over the vertices according to the order $\sum_{x_{i_1}} \sum_{x_{i_2}}\cdots \sum_{x_{i_p}}$. 
If the graph contains $2p+s$ edges, then $2p$ of them will be used in the above sum to give a $\Gamma^{2p}$ factor while the rest of the $s$ edges will be bounded by $\Phi^s$. 
However, the ordered nested property is hard to track under graph expansions, especially because the order of the vertices will change completely after each expansion. Fortunately, we find that the ordered nested property is implied by a stronger but more trackable structural property of graphs, which we shall call the {\bf independently path-connected (IPC) nested property}. A graph $\mathcal G$ with vertices $\{x_1, \ldots, x_p, \star\}$ is said to satisfy the {\it IPC nested property} (or has the {\it IPC nested structure}) if for each vertex, there are at least 2 separated paths connecting it to $\star$, and the edges used in these $2p$ paths are all distinct. One can show with pigeonhole principle that a graph with IPC nested structure always satisfies the ordered nested property. For example, the graphs in Fig.\,\ref{fig1pdf} do not satisfy the IPC nested property.  
On the other hand, the graphs in Fig.\,\ref{fig2pdf} have IPC nested structures and one can see that the vertices can be ordered as $\star\preceq x_2\preceq x_1$.

\begin{figure}[htb]
\centering
\includegraphics[width=11cm]{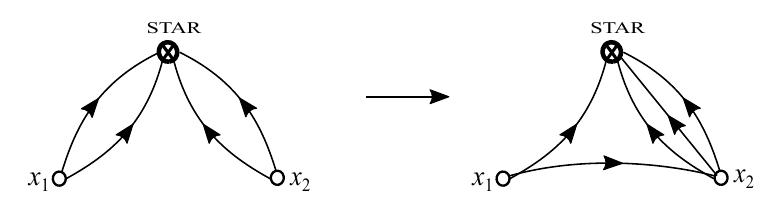}
\caption{The left graph represents $|G_{x_1  \star}|^2  |G_{x_2 \star}|^2 $. We apply the first resolvent expansion in \eqref{resol_examp} to $G_{x_1\star}$ and draw one of the new graphs on the right, where the $G_{x_1 \star}$ edge is replaced by two edges $G_{x_1x_2}G_{x_2\star}$, which still constitute a path from $x_1$ to $\star$. Here we omitted the $G_{x_2x_2}^{-1}$ factor in the second graph.}
\label{fig2pdf}
\end{figure}

In the proof, we always start with graphs with IPC nested structures. The main reason we introduce this stronger concept is that compared with the ordered nested property, it is much easier to check that the IPC nested property is preserved under resolvent expansions. 
Here the IPC nested property is preserved in the sense that if the original graph has IPC nested structure, then all the new graphs appeared in the resolvent expansions also have IPC nested structures. This in fact follows from a simple observation that, in resolvent expansions, we always replace an edge between vertices, say, $\alpha$ and $\beta$ with a path between the same two vertices $\alpha$ and $\beta$. In particular, the path connectivity from any vertex to the $\star$ vertex is unchanged. Hence we are almost guaranteed to have the IPC nested property (which implies the ordered nested property) at each step of our proof. However, we need to be very careful during the proof since the graph operations other than resolvent expansions may break the IPC nested structure, and this brings a lot of technical difficulties to our proof as we will see in Section \ref{sec_colors}.

\subsubsection {Two-level structures}  \label{sec_2level}
In estimating the $p$-th moment in \eqref{eqn-Qpart}, the initial graph will contain $p$ free indices, say $\{x_1, \ldots, x_p\}$. However, in some resolvent expansions, we will add new vertices to the new graphs, such as the new vertex $\alpha$ in the second expansion in \eqref{resol_examp}. Moreover, these indices lie within  $W$-neighborhoods around the free indices. Thus in general, we shall bound averages of products of the form
$$\prod_{i=1}^p \left( G_{\alpha^{(i)}_1\beta^{(i)}_1} \cdots G_{\alpha^{(i)}_{k_i}\beta^{(i)}_{k_i}}\right), \quad \max_{1\le k \le k_i}\Big|\alpha^{(i)}_k - x_i\Big| =\OO(W), \quad 1\le i \le p,$$
up to the choice of the charges of the resolvent entries. (Here the charge of a resolvent entry indicates whether it is a $G$ factor or a $\overline G$ factor.)
Unfortunately, the introduction of new indices breaks the connected paths from the free vertices to the $\star$ vertex. Hence we lose the IPC nested property of the free vertices $\{x_1, \ldots, x_p\}$, which, as we discussed above, helps us to get enough number of $\Gamma$ factors.

To handle this problem, we introduce the random variables $(\Psi_{xy})_{x,y \in \Z_N^d}$, see Definition \ref{lzzay}. They are roughly defined as the local $L^2$-averages of the $G$ entries with indices within $W$-neighborhoods of $(x,y)$:
$$|\Psi  _{xy}|^2:= {W^{-2d}}  \sum_{\max\{|x'-x|, |y'-y|\}\le N^\tau W}\left(|G_{x'y'}|^2+|G_{ y'x'}|^2\right),$$
for small constant $\tau>0$. 
It is easy to see that under \eqref{GM1},
\be\label{Psi1}
|\Psi  _{xy}| \prec N^\tau\Phi, \quad \sum_{x} |\Psi_{xy}|^2 \prec N^{2\tau}\Gamma^2.
\ee
The importance of the $\Psi$ variables is that they provide {\it locally uniform} bounds on the off-diagonal $G$ entries, i.e., for any free vertices $x_i$ and $x_j $,
\be\label{molecule1}
\max_{\max\{|\alpha - x_i|, |\beta - x_j |\}\le (\log N)^CW} \mathbf 1(\alpha \ne \beta)|G_{\alpha\beta}| \prec \Psi_{x_i x_j}.
\ee
This follows from a standard large deviation estimate; see the proof for (\ref{chsz2}). It then motivates us to organize the graphs according to certain subclasses of vertices. More specifically, we shall call the indices {\it{atoms}}, where the $\star$ index is called the $\star$ atom and the free indices $\{x_1, \ldots, x_p\}$ are called free atoms. We then group each free atom $x_i$ and the atoms within its $W$-neighborhood into a subclass called {\it{molecule}}, denoted by $[x_i]$. ({More precisely, an atom $\al$ belongs to the molecule $[x_i]$ only if $\al$ can only take values subject to the condition $|\al-x_i|\le N^\tau W$.} Note that even if an atom $\beta$ is {\it not} in the molecule $[x_i]$, {\it some} of its values can still lie in the $W$-neighborhood of $x_i$.) Here we are using the words ``atom" and ``molecule" in a figurative way. We now have a two-level structures for a particular graph, that is, the structure on the atomic level and the one on the molecular level (i.e., on the graph where each molecule is regarded as one vertex). We have the following simple observations: 
\begin{itemize}
\item although the graphs can keep expanding with new atoms added in, the graphs on the molecular level are always simple with the $\star$ atom and $p$ molecules $[x_i]$, $i=1, \ldots, p$;
\item by \eqref{molecule1}, for all the off-diagonal edges with one end in molecule $[x_i]$ and one end in molecule $[x_j]$, they can be bounded by the same $\Psi_{x_i x_j}$ variable;
\item the path connectivity from any molecule to the $\star$ vertex on the molecular level is preserved under resolvent expansions (since in each expansion, we replace some edge between atoms, say, $\alpha$ and $\beta$, with a new path between two atoms in the same molecules as $\alpha$ and $\beta$).
\end{itemize}
These facts together with \eqref{Psi1} make the molecular graphs and the $\Psi$ variables particularly suitable for defining the IPC nested property. That is, for a general graph, we say it satisfies the {IPC nested property} if the molecular graph with vertices $[x_i]$, $i=1,\ldots, p$, has this property. For this reason, we shall say that the IPC nested structure is an {\it inter-molecule structure}. For example, the molecular graph in Fig.\,\ref{fig3pdf} satisfies the IPC nested property. 
Now following the arguments in Section \ref{sec_nested}, as long as we keep the {IPC nested structure} of the molecular graphs, we can bound the inter-molecule edges by $\Psi$ variables, sum over the free indices according to the nested order, and apply the second bound in \eqref{Psi1} to get the desired factor $\Gamma^{2p}$ in the $p$-th moment estimate.

\begin{figure}[htb]
\centering
\includegraphics[width=15cm]{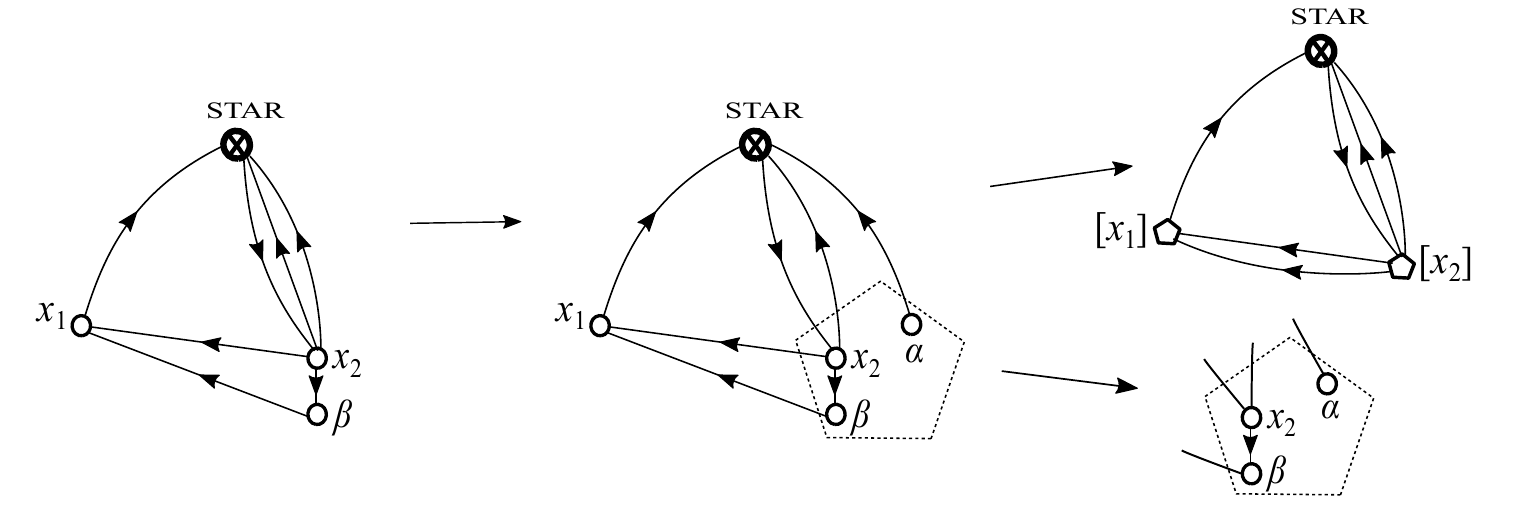}
\caption{Given a graph with two free atoms $x_1, x_2$ and an atom $\beta$ in the $W$-neighborhood of $x_2$, we perform the second resolvent expansion in \eqref{resol_examp} to the edge $G_{x_2\star}$ and get the middle graph, where we create a new atom $\alpha$ in the $W$-neighborhood of $x_2$.  We group $x_2$, $\alpha$ and $\beta$ into a single molecule $[x_2]:=\{x_2,\alpha,\beta\}$, i.e. the part inside the pentagon. The middle graph has a two-level structure, where we draw the molecular graph with molecules $[x_1]$ and $[x_2]$ on the top, and the structure inside the molecule $[x_2]$ (i.e. the inner-molecule structure) on the bottom. 
Again we have omitted some details in the graphs, such as the $G_{xx}$ and $H_{x\alpha}$ factors.}
\label{fig3pdf}
\end{figure}

Given the above definition, it is easy to check that the {IPC nested property} on the molecular graphs are preserved under resolvent expansions. Moreover, the above view of point of ``two-level structure" will also facilitate our following proof. In fact, besides the $\Gamma^2$ factors from the {IPC nested structure}, we still need to extract enough number of $\Phi$ factors. Roughly speaking, we will adopt the idea in \cite{EKY_Average}, which has led to the two extra $\Phi$ factors in \eqref{rough1} besides the factor $N^d\Phi^2$. The approach in \cite{EKY_Average} allows one to divide the graph into smaller subgraphs and bound each part separately. This is possible because only the total number of off-diagonal edges (i.e.\;the $\Phi$ factors) in the graph matters. But the same approach cannot be applied to our proof, because we need to maintain the IPC nested structure of the graph as a whole. As a result, some manipulations of the graphs in \cite{EKY_Average} that can destroy the IPC nested structure are not allowed. Instead, we shall organize our proof according to the two-level structure: the {\it inter-molecule} structure, and the {\it inner-molecule} structures (i.e.\;the subgraphs inside the molecules). In the proof, the inter-molecule structure are only allowed to be changed through resolvent expansions, since we need to keep the IPC nested property. We will show that the inter-molecule structures of the graphs only provide a $\Gamma^2\Phi$ factor in \eqref{rough3}. On the other hand, the rest of the $\Phi$ factor will come from graph operations which may change the inner-molecule structures but preserve the IPC nested structures. {This will be discussed in detail in next section.}


\subsubsection{The role of $Q_x$'s}\label{sec_colors}

In this subsection, we discuss the basics idea to obtain the $\Phi^2$ factor. {The mechanism for this improvement was first discovered in \cite{EKY_Average}, and it has played an essential role in the study of random band matrices in \cite{delocal}.} So far in the discussion, we have ignored the $Q_x$'s in \eqref{eqn-Qpart}. In fact, to bound the left-hand side of \eqref{eqn-Qpart}, we need to estimate averages of the following form
 \be\label{color_aver}
\mathbb E\sum_{\mathbf x} c_{\,\mathbf x} \cal G(\mathbf x), \quad   \mathbf x :=(x_1, \ldots, x_p), \quad c_{\,\mathbf x} =\OO(1), \quad \cal G(\mathbf x) :=\prod_{i=1}^{p}Q_{x_i} \left(\cal G_{x_i} \right),  
 \ee
where $\cal G_{ x_i} $ denotes the part of the expression obtained from the resolvent expansions of $|G _{{x_i}\star}|^2$. We will use colors to represent the $Q_x$'s in graphs, i.e.\;we associate to all the components in $\mathcal G_{x_i}$ a color called ``$Q_{x_i}$". 
To avoid ambiguity in the graphical expressions, we require that every component of the graph has a unique color, in the sense that every component belongs to at most one $Q_x$ group. 
In Fig.\,\ref{fig4pdf}, we give an example of a colorful graph.

\begin{figure}[htb]
\centering
\includegraphics[width=11cm]{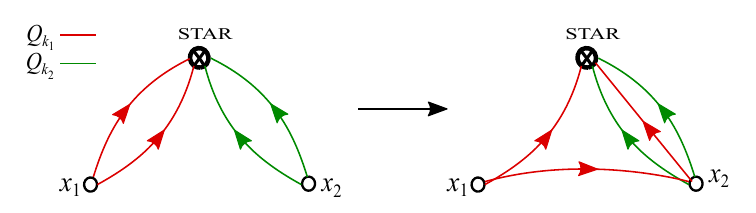}
\caption{We color the graphs in Fig.\,\ref{fig2pdf} with colors $Q_{x_1}$ (red) and $Q_{x_2}$ (green). The left graph now represents $Q_{x_1}(|G_{x_1  \star}|^2)  Q_{x_2}(|G_{x_2 \star}|^2) $, and the right graph represents $Q_{x_1}  (G_{x_1x_2}G_{x_2\star}G_{x_2x_2}^{-1}\overline{G_{x_1\star}^{(x_2)}})Q_{x_2}\left(  |G_{x_2\star}|^2\right) $, where we did not draw the $G_{x_2x_2}^{-1}$ factor. }
\label{fig4pdf}
\end{figure}

The idea of using averaging over $Q_x(\cdot)$ terms to get an extra $\Phi$ factor is central in \cite{EYY_B} and subsequently used in the proofs of fluctuation averaging results of many other works, e.g. \cite{EKYY1,Semicircle,EYY_rigid,PY}. In these papers, the authors studied the specific quantity $\sum_x b_x Q_x(G_{xx}^{-1})$, but we can apply the same idea to $\sum_{x\ne \star} b_x Q_x |G_{x  \star}|^2$.  
Roughly speaking, we can write the expectation of the product in $\cal G(\mathbf x)$ as
$$\mathbb E\left[  \cal A_i Q_{x_i} \left(\cal G_{x_i} \right) \right] = \mathbb E\left[ \left(\cal A_i - \cal A_i^{(x_i)} \right) Q_{x_i} \left(\cal G_{x_i} \right)  \right], $$
where $\cal A_i $ is the expression outside $Q_{x_i}$, $\cal A_i^{(x_i)}$ is any expression that is independent of the $x_i$-th row and column of $H$, and we have used $\mathbb E_{x_i}Q_{x_i}(\cdot)=0$ for the equality. It turns out that if $\cal A_i$ does not contain the $x_i$ index, then it is weakly correlated with the $x_i$-th row and column of $H$, and we can chose $\cal A_i^{(x_i)}$ such that the typical size of $(\cal A_i - \cal A_i^{(x_i)})$ is smaller than $\cal A_i$ by a $\Phi$ factor. 
If $\cal A_i$ contains the $x_i$ atom, then it already contains sufficiently many off-diagonal edges, i.e. $\Phi$ factors, as we need. We can perform the above operations to all the free indices $x_i$, $1\le i \le p$, and obtain an extra $\Phi^p$ factor. As an example, for $p=2$ and $x_1\ne x_2$, we can use the first resolvent expansion in \eqref{resol_examp} to write
\begin{align}\label{fluc example}
&\mathbb E \left(Q_{x_1} |G_{x_1\star}|^2\right)\left( Q_{x_2} |G_{x_2\star}|^2\right) \\
&= \mathbb EQ_{x_1} \left[\left(G_{x_1\star}^{(x_2)}+\frac{G_{x_1x_2}G_{x_2\star}}{G_{x_2x_2}}\right)\overline{\left(G_{x_1\star}^{(x_2)}+\frac{G_{x_1x_2}G_{x_2\star}}{G_{x_2x_2}}\right)}\right]\left( Q_{x_2} |G_{x_2\star}|^2\right) .
\end{align}
Thus for $\Gamma_2= Q_{x_1} |G_{x_1\star}|^2$, we can choose $\Gamma_2^{(x_2)}= Q_{x_1} |G^{(x_2)}_{x_1\star}|^2$ such that $(\cal A_2-\cal A_2^{(x_2)})$ contains 
at least one more off-diagonal edge of order $\Phi$ (see the right graph of Fig.\,\ref{fig4pdf}). In the actual proof, instead of using the free indices, we will use the concept of free molecules, but the main ideas are the same. 

The origin of the second $\Phi$ factor is more subtle, and was first identified in \cite{EKY_Average}. Roughly speaking, it comes from averages of the following form in \eqref{color_aver}: 
\be\label{charged_intro}
\sum_{\alpha } b_\alpha G_{\alpha \beta_1}G_{\alpha\beta_2},\quad b_\alpha=\OO(W^{-d})\mathbf 1\left( |\alpha - x_i|\lesssim W\right),
\ee
where $\beta_{1,2}$ are atoms outside the molecule $[x_i]$. 
A key observation of \cite{EKY_Average} is that $G_{\alpha\beta_1}G_{\alpha\beta_2}$ satisfies the self-consistent equation 
\be\label{charged_intro1}
G_{\alpha\beta_1}G_{\alpha\beta_2} = \sum_\gamma (1-m^2S)^{-1}_{\alpha \gamma}\left[Q_\gamma\left(G_{\gamma\beta_1} G_{\gamma\beta_2} \right)+\cal E_\gamma \right],
\ee
where $\cal E_\gamma$ denotes the error term for each $\gamma$, and it is smaller than the main term by a $\Phi$ factor. For the main terms, we get an average of the form
\be\label{charged_intro2}\sum_\gamma c'_\gamma Q_\gamma\left(G_{\gamma\beta_1} G_{\gamma\beta_2} \right), \quad c'_\gamma=\OO(W^{-d})\mathbf 1\left( |\gamma - x_i|\lesssim W\right),
\ee
which leads to another $\Phi$ factor by the argument in the previous paragraph. {One main difficulty in applying the above argument to our setting is that, different from \cite{EKY_Average}, we need to maintain the IPC nested property defined in Section \ref{sec_nested} throughout all the operations on the graphs. Roughly speaking, we will see that the above argument works due to the following reasons:}
\begin{itemize}
\item[(1)] the entries $(1-m^2S)^{-1}_{\alpha\gamma}$ are negligible for $|\alpha-\gamma|\ge (\log N)^2 W$ (see \eqref{tianYz}), so we can include $\gamma$ into the molecule $[x_i]$ such that the IPC nested structure of the graph is unchanged after replacing \eqref{charged_intro} with \eqref{charged_intro2};

\item[(2)] replacing \eqref{charged_intro} with the $\mathcal E_\gamma$ part also preserves the IPC nested structure; 

\item[(3)] 
each free molecule $[x_i]$ contains at least one atom $\alpha$ that is connected with two edges of the form \eqref{charged_intro}. 
\end{itemize}
Here (1) and (2) ensure the IPC nested structure of the new graphs, and (3) shows that we can get enough $\Phi$ factors from the free molecules. However, we still have the following technical issues, which make the above argument to be the trickiest part of our proof. 
\begin{itemize}

\item[(i)] We always start with a colorful graph. However, for the above arguments to work, the two edges $G_{\alpha\beta_1}G_{\alpha\beta_2}$ need to be colorless. Thus we first need to remove all the colors (i.e.\;the $Q_x$'s) from the graphs, i.e. write a colorful graph into a linear combination of colorless graphs. 

\item[(ii)] The atom $\alpha$ connected with the two edges $G_{\alpha\beta_1}G_{\alpha\beta_2}$ may be also connected with other edges. Thus we need to perform some operations to get {a new graph which} contains a (possibly different) atom $\alpha'$  that is connected with only two edges $G_{\alpha'\beta_1}G_{\alpha'\beta_2}$ and is in the same molecule as $\alpha$. We shall call such an atom a {\it simple charged atom}. 

\item[(iii)] The simple charged atoms in different molecule may share edges. 
Hence we have to handle them one by one, not as a whole. Moreover, each time we apply the previous argument from \eqref{charged_intro} to \eqref{charged_intro2}, we need to repeat the processes in (i) and (ii) again. 
\end{itemize}
{Due to these issues, the operations on the graphs have to be performed one by one in a carefully chosen order.} It is worth mentioning that the operations in (i) and (ii), although can be very complicated, are easy to check to preserve the IPC nested structures of the graphs. 

Finally, we remark that the above arguments for \eqref{charged_intro} cannot be applied to terms of the form $G_{\alpha\beta_1}\overline{G}_{\alpha\beta_2}$, since the $(1-m^2S)^{-1}$ in \eqref{charged_intro1} is well-behaved due to the nonzero imaginary parts of $m$, while $(1-|m|^2S)^{-1}$ in the case of $G_{\alpha\beta_1}\overline{G_{\alpha\beta_2}}$ is not since $|m| = 1- \OO(\eta)$.

\subsubsection{Summary of the proof}

Following the above discussions, our main proof for Lemma \ref{Ppart} and Lemma \ref{Qpart} consists of the following four steps.
\begin{itemize}
\item[{\bf Step 0:}] Develop a graphical tool which extends the previous ones used in e.g. \cite{EKY_Average,Semicircle}. This is the content of Section \ref{subsec: graph} and Section \ref{def_II}. 
 
\item[{\bf Step 1:}] Starting with the graphs in the high moment calculation, we perform graph expansions, identify the IPC nested structures and obtain the first $\Phi$ factor. This is the content of Sections \ref{sec_simple}-\ref{sec_simple2} and Section \ref{sec_step1}. This step, although contains the main new ideas of this paper as discussed in Sections \ref{sec_nested} and \ref{sec_2level}, is actually the relatively easier step of our proof. 

\item[{\bf Step 2:}] Remove the colors as discussed in the above item (i). This is the content of Section \ref{40+}.

\item[{\bf Step 3:}] Create simple charged atoms as discussed in the above item (ii). This is the content of Section \ref{step3}.

\item[{\bf Step 4:}] Deal with simple charged atoms using \eqref{charged_intro1}. This is the content of Section \ref{charged}.
\end{itemize}






 \section{Basic tools}  \label{sec_tool}
 
The rest of this paper is devoted to proving Lemma \ref{Ppart} and Lemma \ref{Qpart}. In this section, we collect some tools and definitions that will be used in the proof.
 

\begin{definition}[Minors] \label{minors}
For any $L\times L$ matrix $A$ and $\mathbb T \subset \{1, \dots, L\}$, $L\in \N$, we define the minor of the first kind $A^{[\mathbb T]}$ as the $(L-|\mathbb T|)\times (L-|\mathbb T|)$ matrix with
\begin{equation*}
(A^{[\mathbb T]})_{ij} \;\deq\; A_{ij}, \quad i,j  \notin \mathbb T.
\end{equation*}
For any $L\times L$ invertible matrix $B$, we define the minor of the second kind $B^{(\mathbb T)}$ as the $(L-|\mathbb T|)\times (L-|\mathbb T|)$ matrix with
\begin{equation*}
(B^{(\mathbb T)})_{ij}=\left( (B^{-1})^{[\mathbb T]}\right)^{-1}_{ij}, \quad i,j  \notin \mathbb T,
\end{equation*} 
whenever $(B^{-1})^{[\mathbb T]}$ is invertible. Note that we keep the names of indices when defining the minors. By definition, for any sets $\mathbb U,\mathbb T\subset \{1,\dots, L\}$, we have 
\be\label{ABTU}
(A^{[\mathbb T]})^{[\mathbb U]}=A^{[\mathbb T\cup \mathbb U]}, \quad
 (B^{(\mathbb T)})^{(\mathbb U)}=B^{(\mathbb T\cup \mathbb U)} .
\ee
For convenience, we shall also adopt the convention that for $i\in \mathbb T$ or $j\in \mathbb T$,
$$(A^{[\mathbb T]})_{ij}=0, \quad (B^{(\mathbb T)})_{ij}=0.$$
We will abbreviate $(\{a\})\equiv (a)$, $[\{a\}]\equiv [a]$, $(\{a,b\})\equiv (ab)$, $[\{a,b\}]\equiv [ab]$ and $\sum_x^{(\mathbb T)}:= \sum_{x \col x \notin \mathbb T}$. 
\end{definition}

\begin{remark}
In previous works, e.g. \cite{EKYY2,Bulk_generalized}, we have used the notation $(\cdot)$ for both the minor of the first kind and the minor of the second kind. Here we try to distinguish between $(\cdot)$ and $[\cdot]$ in order to be more rigorous.
\end{remark}

The following identities are easy consequences of the Schur complement formula. The reader can refer to, for example, Lemma 4.2 of \cite{Bulk_generalized} and Lemma 6.10 of \cite{EKYY2} for the proof.

\begin{lemma}[Resolvent identities]\label{resolvent_id}
For any $L\times L$ invertible matrix $B$ and $1\le i,j,k\le L$, we have
\begin{equation} \label{Gij Gijk}
B_{ij} \;=\; B_{ij}^{(k)} + \frac{B_{ik} B_{kj}}{B_{kk}},
\end{equation}
\begin{equation} \label{Gij Gijk2}
\frac{1}{B _{ii}}=\frac{1}{B_{ii}^{(k)}} - \frac{B_{ik}B_{ki}}{B _{ii} B_{kk}B^{(k)}_{ii}},
\end{equation}
and
\begin{equation} \label{sq root formula2}
\frac 1{B_{ii}} =  (B^{-1})_{ii}-\sum_{k,l}^{(i)}(B^{-1})_{ik}  B^{(i)}_{kl} (B^{-1})_{li} .
\end{equation}
Moreover, for $i \neq j$ we have
\begin{equation} \label{sq root formula}
B_{ij} = - B_{ii} \sum_{k}^{(i)} (B^{-1})_{ik} B_{kj}^{(i)} = - B_{jj} \sum_k^{(j)}  B^{(j)} _{ik} (B^{-1})_{kj}\,.
\end{equation}
{The above equalities are understood to hold whenever the expressions in them make sense.}
\end{lemma}

Next we introduce the $\Psi$ random variables, which are important control parameters for our proof.  

 
  \begin{definition}[Definition of $\Psi_{xw}$]\label{lzzay}
For any small constant $\tau>0$, we define positive random variables $\Psi_{xy}$ as 
$$|\Psi  _{xy}|^2\equiv |\Psi  _{xy}( \tau)|^2:= s_{xy} + \sum_{ |x-x'| \le N^{\tau}W}\sum_{ |y-y'|\le N^{\tau}W}\frac{1}{W^{2d}}\left(|G_{x'y'}|^2+|G_{ y'x'}|^2\right), \quad x,y\in \Z_N^d.$$
Similarly for any $\mathbb T \subset \{1, \dots, N\}$, we can define $\Psi^{(\mathbb T)}$ by replacing the $G$ entries with $G^{(\mathbb T)}$ entries in the above definition. For simplicity, we will often do not write out $\tau$ explicitly when using the $\Psi$ variables. 
\end{definition}

Note that $|\Psi_{xy}|$ is a local $L^2$-average of the $G$ entries with indices within an $N^\tau W$-neighborhood of $(x,y)$. 
The importance of the $\Psi$ variables is that they provide local uniform bounds on the $G$ entries, see (\ref{chsz2}) below. 
 
%
 
 Since we do not want to keep track of the number of $N^\tau$ factors in our proof, we introduce the following notations. 
For any non-negative variable $A$, we use $B= \OO_\tau(A)$ or $|B| \le N^{\OO(\tau)}A$ to mean that $ |B| \le N^{C\tau} A $ for some constant $C>0$ {independent of $\tau$}. We use $B\prec_\tau A$, $B\prec \OO_\tau(A)$ or $B= \OO_{\tau,\prec} (A)$ to mean that $ |B| \prec N^{C\tau} A $ for some constant $C>0$ {independent of $\tau$}. {In particular, $C\tau$ will be a small constant as long as $\tau$ is sufficiently small.} Moreover, we denote
$$\OO_\tau\left(f\left(\{\Psi\}_{x,y\in \Z_N^d}\right)\right):= \OO\left(N^{C\tau}  f\left(\{\Psi (\tau)\}_{x,y\in \Z_N^d}\right)  \right).$$
 where $f$ is a non-negative function of $\Psi$ variables. 
 
We will use the following lemma tacitly in the proof. It can be proved easily using the definition of high probability events. 
\begin{lemma}[Lemma B.1 of \cite{Semicircle}]\label{lem_partial}
Given a nonnegative random variable $X$ and a deterministic control parameter $\varphi$ such that $X\le \varphi$ with high probability. Suppose $\varphi\ge N^{-C}$ and $X\le N^C$ almost surely for some constant $C>0$. 
Then we have for any fixed $n\in \mathbb N$, 
\be\label{partial_P}
\mathbb E X^n =\OO(\varphi^n), \quad \text{ and } \quad \max_x \mathbb E_x X \prec  \varphi .
\ee
\end{lemma}
Note that by (\ref{xiazhou}), we have the deterministic bound 
\be\label{xiangmmz}
\|G\| \le \frac{1}{\min_i (\im z_x)} \le N^{C_2}.
\ee
This provides a deterministic bound on $X$ required by Lemma \ref{lem_partial} when $X$ is a polynomial of $G$ entries. 
 
The following lemma gives a large deviation bound that will be used in the proof of Lemma \ref{xiyan}. 
\begin{lemma}[Theorem B.1 of \cite{delocal}]\label{large_deviation}
Let $(X_i)_{i=1}^N$ be an independent families of random variables and $(b_{i})_{i=1}^N$ be deterministic complex numbers. Suppose all entries $X_i$ satisfy 
$$\mathbb E X_i=0, \quad \mathbb E| X_i |^2 = 1, \quad \left(\mathbb E|X_{i}|^p\right)^{1/p} \le \mu_p,$$
for all $p$ with some constants $\mu_p$. Then we have
\begin{equation}
\Big|\sum\limits_i {b_i X_i } \Big| \prec \Big( {\sum\limits_i {\left| {b_i} \right|^2 } } \Big)^{1/2} . 
\end{equation}
\end{lemma}

 We now collect some important properties of $\Psi$ variables in the next lemma. For simplicity, we introduce the following notations: consider a path $x=w_0 \to w_1 \to \cdots \to w_k \to w_{k+1}=y$ with each edge assigned a weight $\Psi_{w_i w_{i+1}}$, 
 we shall denote
\be\label{psipath}
\Psi_{(x, w_1 , w_2 , \cdots, w_k, y)}(\tau):=\Psi_{xw_1}(\tau)\Psi_{w_1 w_2}(\tau) \ldots \Psi_{w_{k-1} w_k}(\tau) \Psi_{w_k y}(\tau).
\ee
In particular, by convention we have $\Psi_{(x,  y)}(\tau)=\Psi_{xy}(\tau)$.

\begin{lemma}\label{xiyan}
Fix any sufficiently small constant $\tau>0$ and any subset $\mathbb T\subset \Z_N$ with $|\mathbb T|=\OO(1)$. Suppose \eqref{GM1} holds. Then we have the following statements.  
\begin{itemize}
\item 
We have for any $x,y\in \Z_N^d$,
\be\label{chsz1}
\Psi_{xy}(\tau)=\Psi_{yx}(\tau), \quad s_{xy}^{1/2} \le \Psi_{xy}(\tau) \prec N^{d\tau}\Phi .
\ee

\item We have for any $y\in \Z_N^d$,
\be\label{chsz4}
  \sum_{x} |\Psi_{xy}(\tau)|^2 \prec N^{2d\tau} \Gamma^2.
\ee

\item 
For any $\wt \tau\ge \tau+(\log N)^{-1/2} $, if for some constant $C>0$, 
\be\label{qren0}
\max\{|x-x'|, |y-y'|\}\le CN^\tau W ,
 \ee 
 then we have
\be\label{chsz3}
\Psi _{xy}(\tau)\le \Psi _{x'y'}(\wt \tau). 
\ee

\item 
  If for some constant $C>0$,
\be\label{qren}
\max\{|x-x'|, |y-y'|\}\le (\log N)^C W ,
 \ee
then we have 
\be\label{chsz2}
{\bf 1}(x\ne y)\left| G_{xy}  \right|\prec \Psi_{x'y'}\left(\tau\right) 
\ee
 If $x,y \notin \mathbb T$ and \eqref{qren} holds, then
\be\label{chsz2.1}
\left|G ^{(\mathbb T)}_{xy}- G_{xy}\right|\prec \sum_{k=1}^{|\mathbb T|}\sum_{(w_1, w_2 , \ldots, w_k) \in \mathcal P_k (\mathbb T)} \Psi_{(x', w_1 , w_2 , \cdots, w_k, y')}(\tau), 
\ee
where $\mathcal P_k (\mathbb T)$ is the collection of all the $k$ ordered indices in $\mathbb T$ and we recall \eqref{psipath}. The estimates \eqref{chsz2} and \eqref{chsz2.1} also hold if we replace the $\Psi$ variables with the $\Psi^{(x)}$ and $\Psi^{(xy)}$ variables.
\item 
For any $\wt \tau\ge \tau+(\log N)^{-1/2} $, we have
\be\label{chsz2.5}
 \Psi^{(\mathbb T)}_{xy}(\tau)\prec N^{d\tau}\Psi_{xy}(\wt \tau) + N^{d\tau} \sum_{k=1}^{|\mathbb T|}\sum_{(w_1, w_2 , \ldots, w_k) \in \mathcal P_k (\mathbb T)} \Psi_{(x, w_1 , w_2 , \cdots, w_k, y)}(\wt\tau)\ .
\ee 
In particular, it implies that 
\be\label{chsz4.0}
 \Psi^{(\mathbb T)}_{xy}(\tau)\prec\OO_\tau(\Phi),\quad  \sum_{x} |\Psi_{xy}^{(\mathbb T)}(\tau)|^2 \prec \OO_\tau (\Gamma^2).
\ee
\end{itemize}
\end{lemma}

From \eqref{chsz2}, one can see that the $\Psi$ variables serve as local uniform bounds on the $G$ (and $G^{(\mathbb T)}$) entries. Moreover, (\ref{chsz4}) shows that the sum of $|\Psi_{xy}|^2$ over $x$ or $y$ gives the factor $\Gamma^2$ (instead of $N^d \Phi^2$), which is one of the key components of the proof for Lemma \ref{Ppart} and Lemma \ref{Qpart}. 
 
 
\begin{proof}[Proof of Lemma \ref{xiyan}]   
Using Definition \ref{lzzay} and \eqref{GM1}, one can easily prove \eqref{chsz1}, \eqref{chsz4} and \eqref{chsz3}. Now we prove \eqref{chsz2}. We first consider the case $\mathbb T=\emptyset$. Since $\{h_{xw}\}$ entries are independent of the $G^{(x)}$ entries, then with \eqref{sq root formula} and the large deviation estimate in Lemma \ref{large_deviation}, we get that 
\be\label{miny2}
|G_{xy}| \prec |G_{xx}|  \Big(\sum_{w}^{(x)} s _{xw}|G^{(x)}_{wy}|^2\Big)^{1/2} \prec \Big(\sum_{w}^{(x)} s _{xw}|G^{(x)}_{wy}|^2\Big)^{1/2},\quad x\ne y ,
 \ee
where we used $|G_{xx}|\sim 1$ with high probability by \eqref{GM1} in the second step. 
Then with \eqref{Gij Gijk} and \eqref{GM1}, we obtain that for $x\ne y$,
 $$\sum_{w}^{(x)} s _{xw}|G^{(x)}_{wy}|^2\le 2 \sum_{w}^{(x)}s _{xw}|G _{wy}|^2+2\sum_{w}^{(x)}s _{xw}\frac{|G_{wx}G_{xy}|^2}{|G_{xx}|^2} = 2 \sum_{w}^{(x)}s _{xw}|G _{wy}|^2+\OO_\prec \left(\Phi^2|G_{xy}|^2\right).
 $$
Plugging this bound into \eqref{miny2}, we obtain that
\be\label{minywer}
|G_{xy}| \prec \Big(\sum_{w}^{(x)}s _{xw}|G _{wy}|^2\Big)^{1/2} + \OO_\prec \left(\Phi|G_{xy}|\right) \Rightarrow |G_{xy}| \prec \Big(\sum_{w}^{(x)}s _{xw}|G _{wy}|^2\Big)^{1/2}.
 \ee
With the same method, we can also prove that
\be\label{miny30}
|G_{xy}| \prec \Big(\sum_{v}^{(y)} |G^{(y)}_{xv}|^2 s _{vy}\Big)^{1/2},\quad x\ne y ,
\ee
and
\be\label{miny3}
|G_{xy}| \prec \Big(\sum_{v}^{(y)} |G_{xv}|^2 s _{vy}\Big)^{1/2},\quad x\ne y .
\ee
Now applying this bound \eqref{miny3} to $G_{wy}$'s in \eqref{minywer}, we obtain that
\be\label{minywerwss}
|G_{xy}| \prec \Big(\sum_{w,v}s _{xw}s_{vy}|G _{wv}|^2 + s_{xy}\Big)^{1/2} \lesssim \Psi_{x'y'}(\tau),\quad x\ne y ,
 \ee
where the $s_{xy}$ comes from the diagonal term with $w=y$ in \eqref{minywer}, and we used the Definition \ref{lzzay} and \eqref{qren} in the second step. Note that applying \eqref{miny3} to the $G^{(x)}_{wy}$ entry in \eqref{miny2}, we get that
$$|G_{xy}| \prec \Big(\sum_{w}^{(x)} s _{xw}s_{vy} |G^{(x)}_{wv}|^2 + s_{xy}\Big)^{1/2}\lesssim \Psi_{x'y'}^{(x)}(\tau), \quad x\ne y .$$
Similarly, applying \eqref{miny30} to to the $G^{(x)}_{wy}$ entry in \eqref{miny2}, we get that 
$$|G_{xy}| \prec \Big(\sum_{w}^{(x)} s _{xw}s_{vy} |G^{(xy)}_{wv}|^2 + s_{xy}\Big)^{1/2}\lesssim \Psi_{x'y'}^{(xy)}(\tau),\quad x\ne y .$$
Thus we have proved \eqref{chsz2}.

The estimate \eqref{chsz2.1} can be proved with mathematical induction in the indices of $\mathbb T$. By \eqref{Gij Gijk}, for $w\notin \{x,y\}$ we have
$$|G_{xy}^{(w)}-G_{xy}| = \left|\frac{G_{xw}G_{wy}}{G_{ww}}\right| \prec \Psi_{x'w}\Psi_{wy'},$$
where in the second step we used \eqref{chsz2} and $|G_{xx}|\asymp 1$ with high probability due to \eqref{GM1} . Now suppose for some set $\mathbb T$ with $|\mathbb T|=\OO(1)$ and $w\notin \mathbb T$, the estimate \eqref{chsz2.1} holds. Then we have
\begin{align}
&|G_{xy}^{(\mathbb T\cup \{w\})}-G_{xy}|  \prec  |G_{xy}^{(\mathbb T\cup \{w\})}-G_{xy}^{(\mathbb T)}| +\sum_{k=1}^{|\mathbb T|}\sum_{(w_1,  \ldots, w_k) \in \mathcal P_k (\mathbb T)}\Psi_{(x', w_1 , w_2 , \cdots, w_k, y')} \nonumber\\
&= \left|\frac{G_{xw}^{(\mathbb T)}G_{wy}^{(\mathbb T)}}{G_{ww}^{(\mathbb T)}}\right| + \sum_{k=1}^{|\mathbb T|}\sum_{(w_1, \ldots, w_k) \in \mathcal P_k (\mathbb T)} \Psi_{(x', w_1 , w_2 , \cdots, w_k, y')} \nonumber\\
& \prec \Big(\left|G_{xw}\right| +\sum_{k=1}^{|\mathbb T|} \sum_{(w_1,  \ldots, w_k) \in \mathcal P_k (\mathbb T)} \Psi_{(x', w_1 , \cdots, w_k, w)} \Big) \Big( \left|G_{wy}\right| + \sum_{l=1}^{|\mathbb T|} \sum_{(w'_1, \ldots, w'_l) \in \mathcal P_l (\mathbb T)} \Psi_{(w,w'_1,\cdots, w'_l, y')} \Big) \nonumber\\
&+\sum_{k=1}^{|\mathbb T|}\sum_{(w_1,  \ldots, w_k) \in \mathcal P_k (\mathbb T)} \Psi_{(x', w_1 , w_2 , \cdots, w_k, y')}\nonumber\\
& \prec  \sum_{k=1}^{|\mathbb T|+1}\sum_{(w_1,  \ldots, w_k) \in \mathcal P_k (\mathbb T\cup \{w\})} \Psi_{(x', w_1 , w_2 , \cdots, w_k, y')} \ . \label{shorten}
\end{align}
Here in the third step we used \eqref{chsz2}, the induction hypothesis and that $G_{ww}^{(\mathbb T)} = m + \OO_\prec(\Phi)  \sim 1$ with high probability. In the last step, for a path of the form $x'\to w_1 \to \cdots \to w_k \to w \to w'_1 \to \cdots \to w'_l \to y'$, we can find the smallest $1\le i \le k$ and the largest $1\le  j \le l$ such that $w_i = w'_j$, and then we can bound the weights in between as $\Psi_{(w_i, w_{i+1},\cdots,w_{j-1}',w_j')}\prec 1$ using \eqref{chsz1} as long as $\tau$ is sufficiently small. In other words, we erase all the loops in the path and get a shorter path from $x'$ to $y'$ without any loop. This explains the expression in \eqref{shorten}. Now by induction, we prove \eqref{chsz2.1} .

Finally, we prove \eqref{chsz2.5}. With Definition \ref{lzzay}, we can write
\begin{align}\label{mid_step}
|\Psi  _{xy}^{(\mathbb T)}(\tau)|^2= |\Psi  _{xy}( \tau)|^2+ \sum_{ |x-x'| \le N^{\tau}W}\sum_{ |y-y'|\le N^{\tau}W}\frac{1}{W^{2d}}\left(|G^{(\mathbb T)}_{x'y'}|^2-|G_{x'y'}|^2 +|G^{(\mathbb T)}_{ y'x'}|^2-|G_{y'x'}|^2\right).
\end{align}
Now using \eqref{chsz2} and \eqref{chsz3}, it is easy to show that
\begin{align*}
& |G^{(\mathbb T)}_{x'y'}|^2-|G_{x'y'}|^2+|G^{(\mathbb T)}_{ y'x'}|^2-|G_{y'x'}|^2 \\
& \prec \Psi_{x'y'}(\tau) \sum_{k=1}^{|\mathbb T|}\sum_{(w_1,  \ldots, w_k) \in \mathcal P_k (\mathbb T)} \Psi_{(x', w_1, \cdots, w_k, y')}(\tau) + \Big(\sum_{k=1}^{|\mathbb T|}\sum_{(w_1,  \ldots, w_k) \in \mathcal P_k (\mathbb T)} \Psi_{(x', w_1, \cdots, w_k, y')}(\tau)  \Big)^2  \\
&\prec \left(\Psi_{xy}(\wt \tau)\right)^2 +\Big(\sum_{k=1}^{|\mathbb T|}\sum_{(w_1,  \ldots, w_k) \in \mathcal P_k (\mathbb T)}\Psi_{(x, w_1, \cdots, w_k, y)}(\wt\tau)  \Big)^2 .
\end{align*}
Together with \eqref{mid_step}, we obtain \eqref{chsz2.5}. Then \eqref{chsz4.0} follows easily from \eqref{chsz2.5} using \eqref{chsz1} and \eqref{chsz4}.
\end{proof}

\section{Proof of Lemma \ref{Ppart}}\label{sec_Ppart}
 
In this section, we prove Lemma \ref{Ppart}. Our goal is to reduce Lemma \ref{Ppart} into another fluctuation averaging lemma---Lemma \ref{Q1}, whose proof will be postponed until Section \ref{sec_simple2}.

We first prove the following lemma on diagonal resolvent entries.
 
 \begin{lemma}\label{qinggan}
We define the $\cal Z$ variables as
 $$ \cal Z_x:= Q_x\left(\sum^{(x)}_{w,v}H_{xw}H_{xv}G^{(x)}_{wv}\right)  - H_{xx}. $$
Under the assumptions of Theorem \ref{YEniu} and \eqref{GM1}, we have that
\be\label{cming}
G_{xx} = M_x+M_x^2 {\cal Z}_x +\OO_\prec \left(\Phi^2\right ) \quad \text{ and }\quad \cal Z_x=\OO_\prec (\Phi). 
 \ee 
 \end{lemma}
 
 \begin{proof}
Note that by \eqref{sq root formula2}, we have $\cal Z_x = - Q_x (G_{xx}^{-1}-M_x^{-1})$. Then by Lemma \ref{lem_partial} and \eqref{xiangmmz}, we have
\be\label{Zi}
\cal Z_x \prec \Phi . 
\ee
 Now applying \eqref{sq root formula2} and \eqref{Gij Gijk}, 
 we get that   
 $$ \frac 1{G_{xx}} =  -z_x -\sum_{y} s_{xy}G _{yy} - \cal Z_x + \OO_\prec(\Phi^2) ,\quad x\in \bZ^d_N. $$
With the definition of $M_x$ in \eqref{falvww}, we then get 
 $$G^{-1}_{xx}-M_x^{-1}=-\sum_{y} s_{ xy }\left(G _{yy}-M_y\right) - \cal Z_x+\OO_\prec(\Phi^2) .$$ 
By (\ref{GM1}), we have $G^{-1}_{xx}-M_x^{-1}=(M_x)^{-2}(M_x-G_{xx})+\OO_{\prec}(\Phi^2)$. Then we obtain that  
$$G_{xx} - M_x = M_x^2 \left(\sum_{y}s_{xy}(G_{yy}-M_{y})+\cal Z_x \right)+\OO_\prec \left(\Phi^2 \right) ,$$
which implies 
 $$G_{xx} - M_x=\sum_{y}\left[(1-M^2S)^{-1}\right]_{xy} M_{y}^2\cal Z_y+\OO_\prec \left(\|(1-M^2S)^{-1}\|_{l^\infty\to l^\infty} \Phi^2 \right)  . $$
By \eqref{gbzz2} and \eqref{tianYz}, we see that with some deterministic coefficients
$$c_y=\OO(W^{-d}) \cdot {\bf 1}\left(|x-y|\le (\log N)^2W\right),$$
we can write 
$$G_{xx} - M_x=M_x^2 \cal Z_x \; + \sum_{y }  c_y \cal Z_y + \OO_\prec \left(\Phi^2 \right) .$$ 
For the second term on the right-hand side, we can apply the fluctuation averaging results in \cite[Theorem 4.6]{Semicircle} to get 
$$
\sum_{y}  c_y \cal Z_y=\OO_\prec(\Phi^2) . 
$$
This completes the proof of Lemma \ref{qinggan}. 
  \end{proof} 
 

 
 Now we start proving Lemma \ref{Ppart}. Our goal for the rest of this section is to reduce Lemma \ref{Ppart} into Lemma \ref{Q1}, whose proof is postponed until Section \ref{sec_simple2}.
 Fix any $x\ne \star$. Recall \eqref{sq root formula}, we can write $ G_{x\star}$ as
\be\label{noV1}  
G_{x\star }=-G_{xx}\sum_{w}^{(x)}H_{xw}G^{(x)}_{w\star }  .
\ee
With the assumption \eqref{GM1} and \eqref{cming}, we know that 
\be\label{taiji}
G_{xx}\sim 1, \quad   \mathbf 1(w\ne x)G_{xw}\prec \Phi , \quad w.h.p.
\ee
Then using \eqref{chsz2}, \eqref{chsz2.5} and (\ref{noV1}), we get that for any fixed $0<\tau' <\tau$, 
\be\label{taiji22}
\sum_{ w}^{(x)}H_{xw}G^{(x)}_{w\star } \prec \Psi_{x\star}^{(x)}(\tau') \prec \OO_\tau\left(\Psi_{x\star}(\tau)\right), \quad {\text{and}} \quad \sum_{ w}^{(x)}H_{xw}G^{(x)}_{w\star } \prec \Psi_{x\star}(\tau) .
\ee
 On the other hand, we have the trivial bound \eqref{xiangmmz} on the ``bad event" with small probability.
 Note that $\Psi_{x\star}^{(x)}$ is independent of the $H$ entries in $x$-th row and column.
Then plugging \eqref{cming}, \eqref{taiji} and \eqref{taiji22} into \eqref{noV1} and using \eqref{partial_P}, we get that 
\begin{align}
       &\mathbb E_x  |G_{x\star}|^2  =|M_x|^2\sum_w^{(x)} s_{xw}|G^{(x)}_{w\star}|^2
       + 2|M_x|^2\re \Big(M_x \E_x  \Big(\cal Z_x\sum_{w,w'}^{(x)}H_{xw}H_{xw'}G^{(x)}_{w\star}\overline{G^{(x)}_{w'\star}}\Big)\Big) +
       \OO_\prec \left(\Phi^2(\Psi^{(x)}_{x\star}(\tau'))^2\right) \nonumber\\
       &=|M_x|^2\sum_w^{(x)} s_{xw}|G^{(x)}_{w\star}|^2
       + 2|M_x|^2\re \Big(M_x \E_x  \Big(\cal Z_x\sum_{w,w'}^{(x)}H_{xw}H_{xw'}G^{(x)}_{w\star}\overline{G^{(x)}_{w'\star}}\Big)\Big) +
       \OO_{\tau,\prec} \left(\Phi^2 \Psi^2_{x\star}(\tau)\right). \label{YYa}
\end{align}
  Next we apply \eqref{Gij Gijk} to $G^{(x)}_{w\star }$ in the first term on the right-hand side of \eqref{YYa}, i.e., 
  \be\label{Gkstar}
  G^{(x)}_{w\star }=G_{w\star }-\frac{G_{wx}G_{x\star }}{G_{xx}}.
  \ee
Since $|x-w|+|x-w'|=\OO(W)$ in \eqref{YYa}, using \eqref{chsz2}-\eqref{chsz2.5} we get that
\be \label{xiangta}
|G_{w\star}| + |G_{w'\star}|+ |G_{x\star}| \prec \Psi_{x\star}^{(x)} (\tau') \prec \OO_\tau\left(\Psi_{x\star}(\tau)\right). 
\ee   
Then with \eqref{taiji}, we obtain that with high probability, 
   \begin{equation}\label{haozy mn}
    \begin{split}
        \mathbb E_x  |G_{x\star}|^2    \; = \;& |M_x|^2\sum_w s_{xw}|G _{w\star}|^2 
 - 2\re \left(M_x \sum_w^{(x)} s_{xw} G _{w\star}\overline{ G_{wx}} \overline{ G_{x\star}}\right)
       \\
         \; + \; &   2|M_x|^2\re \left(M_x\E_x  \left(\cal Z_x\sum_{w,w'}^{(x)}H_{xw}H_{xw'}G^{(x)}_{w\star}\overline{G^{(x)}_{w'\star}}\right)\right) +\OO_{\tau,\prec}  \left(\Phi^2\Psi_{x\star}^2(\tau)\right)\ .
         \end{split}
      \end{equation}
 Here for the term in the second line, using the definition of $\cal Z_x$ we have that
\be\label{4z18}
\begin{split}
   \E_x \cal Z_x\sum^{(x)}_{w,w'}H_{xw}H_{xw'}G^{(x)}_{w\star}\overline{G^{(x)}_{w'\star}}
  =  \; &\E_x \sum^{(x)}_{w_1,w_2,w_3,w_4}H_{xw_1}H_{xw_2}G^{(x)}_{w_1\star}\overline{G^{(x)}_{w_2\star}} \left(H_{xw_3}H_{xw_4}-\delta_{w_3w_4}s_{xw_3}\right)G^{(x)}_{w_3w_4} \\
 - \; & \E_x H_{xx}\sum^{(x)}_{w_1,w_2}H_{xw_1}H_{xw_2}G^{(x)}_{w_1\star}\overline{G^{(x)}_{w_2\star}} \ .
 \end{split}
\ee
Recall that for any $w \in \mathbb Z_N^d$, $H_{x w }$ is independent of $G^{(x)}$ and $\E H_{x w}=0$. Then we see that the second line of \eqref{4z18} vanishes. For the first line, it is easy to calculate
$$\E_x  H_{xw_1}H_{xw_2} \left(H_{xw_3}H_{xw_4}-\delta_{w_3w_4}s_{xw_3}\right),$$ 
which is non-zero only when each $w$ index appears at least twice and all of the indices are in the $W$-neighborhood of $x$. Together with \eqref{chsz2} and \eqref{chsz2.1}, 
we obtain that 
\begin{align*}  
\eqref{4z18} =  \; & \;\OO_\prec (W^{-d} \Psi_{x\star}^2)+2\sum_{w\ne w'}^{(x)}s_{xw}s_{xw'}G^{(x)}_{ww'}G^{(x)}_{w\star}\overline{G^{(x)}_{w'\star}} \ .
\end{align*}
Again using \eqref{Gij Gijk}, we can write each $G^{(x)}$ entry as a combination of the $G$ entry with an error term as in \eqref{Gkstar}. Together with the bounds in \eqref{xiangta}, we obtain that 
\begin{align} 
\eqref{4z18} =   \; & \;\OO_\prec( \Phi^2\Psi_{x\star}^2)+2\sum_{w\ne w'}^{(x)}s_{xw}s_{xw'}G _{ww'}G _{w\star}\overline{G _{w'\star}} \ .
\end{align}
Plugging it into \eqref{haozy mn} and then using \eqref{def: T} and \eqref{chsz4}, we obtain that {\it{w.h.p.}},
\begin{align*}  
   \sum_{x\ne \star} b_x\left(\mathbb E_x |G_{x\star}|^2-|M_x|^2T_{x\star}\right)   = - 2 \sum_{x\ne \star}  b_x \re\left(M_x \sum^{(w)}_x  s_{xw} G _{w\star} \overline {G_{wx} G_{x\star}} \right) \\
     +4\sum_{x\ne \star}  b_x|M_x|^2 \re\left(M_x\sum^{(x)}_{w\ne w'} s_{xw}s_{xw'}G _{ww'}G _{w\star}\overline{G_{w'\star}}\right)
     +\OO_{\tau,\prec}(\Gamma^2\Phi^2).
\end{align*}
The contribution of the terms with $w$ or $w'$ equal to $\star$ can be easily bounded by $\OO_\prec(W^{-d}\Gamma^2)$. Then we can write 
\begin{align*}
\sum_{x\ne \star} b_x\left(\mathbb E_x |G_{x\star}|^2-|m|^2T_{x\star}\right)   
=& \re \left(\sum_{x,w: x\ne w}^{(\star)} c_{xw}  G _{w\star}G_{wx} \overline{G _{x\star} }\right)
+\OO_{\tau,\prec}\left(\Gamma^2{\Phi^2}\right) 
\end{align*}
for some deterministic coefficients $c_{xw}$ satisfying
$$
c_{xw}=\OO( W^{-d}){\bf 1}_{|x-w| = \OO((\log N)^2W)}.
$$
(Here we have used $\OO((\log N)^2W)$ instead of $\OO(W)$ due to the Lemma \ref{478} below.) Therefore, to prove Lemma \ref{Ppart}, it suffices to prove that
 \be\label{gush1}
\left|  \sum_{x,w: x\ne w}^{(\star)} c_{xw}  \left(G _{w\star}G_{wx}\right) \overline{G_{x\star}}\right| =\OO_{\tau,\prec} (1+\Gamma^2\Phi^2) \ .
 \ee
In fact, the $G _{w\star}G_{wx}$ term can be written as a linear combination of $Q_w$ terms as in the following lemma.  

   \begin{lemma}\label{478}  
Under the assumptions of Lemma \ref{Ppart}, 
for $w,w'\in \Z_N^d\setminus \{x\}$\nc, we have 
\be\label{saziy}   G_{xw}G_{xw'} =Q_x(G_{xw}G_{xw'}) +\sum_{y: y\ne w,w'}d_{xy}Q_{y}\left(G_{yw}G_{yw'}\right)+\OO_\prec (\Psi_{xw}\Psi_{xw'}\Phi +W^{-d}\Psi_{ww'} +W^{-d}\delta_{ww'} ),
 \ee
for some deterministic coefficients $d_{xy}=\OO(W^{-d}){\bf 1}_{|x-y|\le (\log N)^2W}.$
 \end{lemma}

\begin{proof}
By \eqref{sq root formula}, since $x$, $w$ and $w'$ are all different, we can write 
\begin{align}\label{step1}
\E_x G_{xw}G_{xw'}=&\E_x G^2_{xx}\sum_{\al,\al'} H_{x\alpha}H_{x\alpha'}G^{(x)}_{\alpha w}G^{(x)}_{\alpha' w'} = M_x^2 \sum_{\alpha}s_{x\alpha}G^{(x)}_{\alpha w}G^{(x)}_{\alpha w'}+\OO_\prec(\Psi_{xw}\Psi_{xw'}\Phi) \ ,\end{align}
where we used (\ref{taiji}) and (\ref{taiji22}) in the second step.
Furthermore, with \eqref{Gij Gijk}, \eqref{chsz3} and \eqref{chsz2}, we get
$$\E_x G_{xw}G_{xw'}=M_x^2 \sum_{\al}s_{x\alpha}G _{\alpha w}G _{\alpha w'}+\OO_\prec(\Psi_{xw}\Psi_{xw'}\Phi).
$$
Hence, for all $x\ne w,w'$, we have
$$  G_{xw}G_{xw'}= M_x^2 \sum_\al s_{x\alpha}G _{\alpha w}G _{\alpha w'}+Q_x \left(G_{xw}G_{xw'}\right)+\OO_\prec(\Psi_{xw}\Psi_{xw'}\Phi) . $$
For $x=w$, we can write {\it{w.h.p.}},
\begin{align*}
G_{ww}G_{ww'} &=  m^2 \sum_\al s_{w\al}G _{\alpha w}G _{\alpha w'}+ \left( G_{ww}G_{ww'}-m^2 \sum_\al s_{w\al}G _{\alpha w}G _{\alpha w'}\right) 
\\
&= m^2 \sum_\al s_{xw}G _{\alpha w}G _{\alpha w'}+ \OO_\prec(\Psi_{ww'} + \delta_{ww'}),
 \end{align*}
and we have a similar expression for the $x=w'$ case.  
Therefore, we get a vector equation for $(G_{xw}G_{xw'}: x\in \bZ^d_N)$, which gives that 
\begin{align*}
G_{xw}G_{xw'} &=\sum_{y:y\ne w, w'}\left[(1-M^2 S)^{-1}\right]_{xy} \left[ Q_y \left(G_{yw}G_{yw'}\right)+\OO_\prec(\Psi_{yw}\Psi_{yw'}\Phi)\right]\\
&+\OO_\prec\left(|(1-M^2 S)^{-1}_{xw}|  \Psi_{ww'}+|(1-M^2 S)^{-1}_{xw'}| \Psi_{ww'} + |(1-M^2 S)^{-1} _{xw}| \delta_{ww'}\right) \ .
\end{align*}
Using \eqref{tianYz} and \eqref{chsz3}, we conclude \eqref{saziy}. 
 \end{proof}

 By Lemma \ref{478}, we know that for some deterministic coefficients $\wt c_{xy}=\OO(W^{-d}){\bf 1}_{|x-y|\le 2 (\log N)^2W}$, 
   \begin{align*}
 \sum_{x,w: x\ne w}^{(\star)} c_{xw}  \left(G _{w\star}G_{wx}\right) \overline{G_{x\star}}
  & = \sum_{x,w: x\ne w}^{(\star)} \wt c_{xw} Q_w \left(G _{w\star}G_{wx}\right) \overline{G_{x\star}} + \sum_{x,w: x\ne w}^{(\star)} c_{xw} \OO_{\prec}\left( \Psi_{w\star}\Phi^2\Psi_{x\star} +W^{-d}\Psi^2_{x\star} \right)
       \\
  & = \sum_{x,w: x\ne w}^{(\star)} \wt c_{xw} Q_w \left(G _{w\star}G_{wx}\right)  \overline{G_{x\star}}+
       \OO_{\tau,\prec}\left(\Gamma^2\Phi^2\right) ,  
       \end{align*}     
where we used \eqref{chsz4} and $ \Phi^2\ge W^{-d}$ in the second step. 
Furthermore, with the resolvent expansion \eqref{Gij Gijk},
we get that for distinct $x,w$ and $\star$, 
\be\label{jsjahyzv}
Q_w\left( G _{w\star}G_{wx}\right)  \overline{G_{x\star}}
= Q_w \left(G _{w\star}G_{wx} \overline{G_{x\star}}\right) 
-  Q_w \left(G _{w\star}G_{wx} \frac{\overline{G_{xw}G_{w\star}} }{\overline {G_{ww}}}\right) 
+ Q_w \left( G _{w\star}G_{wx}\right)\frac{\overline{G_{xw}G_{w\star}} }{\overline {G_{ww}}}.
\ee
With (\ref{chsz2}), we get that for any fixed $0<\tau'<\tau$, 
$$G _{w\star}G_{wx}\prec \Psi^{(w)}_{w\star}(\tau')\Phi , \quad \quad\frac{\overline{G_{xw}G_{w\star}} }{\overline G_{ww}}\prec \Psi^{(w)}_{w\star}(\tau')\Phi. $$
Since these bounds are independent of the $H$ entries in $w$-th row and column, using $|Q_w (\cdot)| \le |\cdot| + |\mathbb E_w(\cdot)|$ and \eqref{partial_P}, we obtain that
$$Q_w\left( G _{w\star}G_{wx}\right) \overline{G_{x\star}}= Q_w \left(G _{w\star}G_{wx} \overline{G_{x\star}}\right) + \OO_\prec\left( (\Psi^{(w)}_{w\star}(\tau'))^2\Phi^2  \right) = Q_w \left(G _{w\star}G_{wx} \overline{G_{x\star}}\right)+ \OO_{\tau,\prec}\left( \Psi _{w\star}^2(\tau)\Phi^2 \right) ,$$
where we used \eqref{chsz2.5} in the second step. The last term then gives $\OO_{\tau,\prec}( \Gamma^2\Phi^2)$ when summing over $\wt c_{xw}$ by \eqref{chsz4}. Hence to prove Lemma \ref{Ppart}, it suffices to prove the following lemma.


\begin{lemma}\label{Q1}
Suppose the assumptions of Lemma \ref{Ppart} hold, and $c_{x\al}$ are deterministic coefficients satisfying 
$$c_{x\al }=\OO( W^{-d}){\bf 1}_{|x-\al|=\OO((\log N)^2 W) } .$$
Then for any fixed (large) $p\in 2\mathbb N$ and (small) $\tau>0$, we have 
 \be\label{gush20}
 \mathbb E\Big|  \sum_{x,\alpha:x\ne \al}^{(\star)} c_{x\al} Q_x \left(G _{x\star}G_{x\al}\overline {G _{\al \star}}\right)\Big|^p \le \left(N^\tau  \Gamma^2\Phi^2  \right) ^p.
 \ee
\end{lemma}
 
\section{Graphical tools - Part I} \label{sec_graph1}

 Now to finish the proof of Theorem \ref{YEniu}, it suffices to prove the Lemma \ref{Qpart} and Lemma \ref{Q1}. 
To help the reader to follow the main idea of the proof, we start with the following easier lemma. Note that compared with (\ref{eqn-Qpart}), (\ref{gush2}) has one less $\Phi$ factor on the right-hand side. 

 \begin{lemma}\label{Q2}
 Suppose the assumptions of Theorem \ref{YEniu} hold, and $c_{x}$ are real deterministic coefficients such that $\max_x|c_x|=\OO(1)$. Then for any fixed $p\in 2\N$ and $\tau>0$, we have  
\be\label{gush2}
 \mathbb E\Big|\sum_{x }^{(\star)} c_{x} Q_x \left(G _{x\star} \overline G _{x\star}\right)\Big|^p \le \left( N^\tau  \Gamma^2\Phi \right)^p .
\ee
\end{lemma}

In the proof of Lemma \ref{Qpart}, Lemma \ref{Q1} and Lemma \ref{Q2}, we use graphical tools, which will be introduced starting from this section.  
For example, the left-hand side of \eqref{gush2} can be written as 
\be\label{Qsimz}\E  \sum_{x_1,\cdots, x_p}^{(\star)}  \left(\prod_{i=1}^p c_{x_i} Q_{x_i} \left(G _{{x_i}\star} \overline G _{{x_i}\star}\right)  \right).
\ee
Then we will expand this expression with the resolvent expansions in Lemma \ref{resolvent_id}, and the graphical tools will help us to bound the long expressions as discussed in the introduction. 


In the rest of this paper, we suppose that the assumptions of Theorem \ref{YEniu} hold. In particular, we always assume \eqref{xiazhou}-\eqref{GM1}, and we will not repeat them again.
  
\subsection{Definition of  Graph  - Part 1.}  \label{subsec: graph} \
In this subsection, we introduce some basic components of the graphical tools needed to prove Lemma \ref{Q2} and Lemma \ref{Q1}.

\begin{definition} [Colorless graph] \label{def_colorlessg}

We consider graphs that contain the following elements.

\begin{itemize}
\item {\bf The star atom $\otimes$\;:} In each graph, there exists at most one star atom, which represents the $\star$ index. 

\item{\bf Regular atoms $\circ$\, :} Any vertex that is not the star atom is called a regular atom (or simply atom). 

 
\item{\bf  Labelled solid edges:} A solid edge that connects atoms $\al$ and $\beta$ represents a $G_{\al\beta}$ factor. 
Each solid edge has the following labels (see the example in \eqref{fig5}): 
\begin{itemize}
\item  a direction, which indicates whether it is $G_{\al\beta}$ or $G_{\beta\al}$;
\item  a charge, which indicates whether it is a $G$ factor or a $G^{*}(\equiv \overline G)$ factor;
\item  an independent set $\mathbb T$ for the $G^{(\mathbb T)}$ entry.
\end{itemize}
We sometimes ignore the direction, charge and independent set, and denote the edge by $Edge(\al, \beta).$ If we want to emphasize the independent set, then we will write  ${Edge}^{(\mathbb T)}(\al, \beta)$.

 \item {\bf Weights $\Delta$:}  A weight at atom $x $ represents a $G_{xx}$ factor or a $\left(G_{xx}\right)^{-1}$ factor. It is drawn as a solid $\Delta$ in the graph. We will introduce other types of weights later in Section \ref{def_II}. 
 Each weight has the following labels (see the example in \eqref{fig5}):
 \begin{itemize}
\item a flavor, which indicates whether it is a $G$ factor or a $G^{-1}$ factor, and we will use the notations $f_1$ (i.e. {\it flavor} 1) for $G$ factors and $f_2$ (i.e. {\it flavor} 2) for $G^{-1}$ factors in the graph;
\item  a charge, which indicates whether it is a $G$ factor or a $G^{*}(\equiv \overline G)$ factor;
 \item an independent set $\mathbb T$ for the $G^{(\mathbb T)}$ or $\left(G^{(\mathbb T)}\right)^{-1}$ entry.
\end{itemize}
We sometimes ignore the flavor, charge and independent set, and denote the weight by ${\rm{Weight}}(x)$. If we want to emphasize the independent set, then we will write  ${\rm{Weight}}^{(\mathbb T)}(x)$.
\end{itemize}

\end{definition} 

In the definition, we used the word ``atom" to illustrate various concepts in a more figurative way. We also remark that a weight is represented by a bubble diagram in the usual graphical language. The following \eqref{fig5} gives a simple example of a colorless graph:
\be\label{fig5}
 \parbox[c]{0.30\linewidth}{\includegraphics[width=\linewidth]{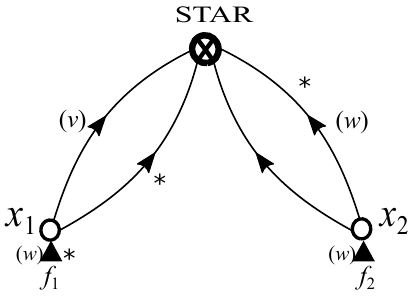}} \qquad 
 = G^{(v)}_{x_1\star}\overline G_{x_1\star}\overline G^{(w)}_{x_1x_1} G_{x_2\star}\overline G^{(w)}_{x_2\star}  (G^{(w)}_{x_2x_2} )^{-1}.
 \ee
Here the equality holds in the following sense.

\vspace{5pt}
\noindent{\bf Values of graphs:} 
For a graph $\mathcal G$, we define its value as the product of all the factors represented by its elements. We will almost always identify a graph with its value in the following proof. 
 
\vspace{5pt}

To represent the $Q_x$'s in the graphs, we introduce the concept of ``colors". There are $2N^d+1$ kinds of colors $\{P_x: x\in \bZ_N^d\}\cup\{Q_x: x\in \bZ_N^d\} \cup \{P_{\emptyset}\}$. Note that $P_x$ and $Q_x$ are related through $Q_x = 1- P_x$, but we treat them as different colors. Also by convention $P_{\emptyset}$ is an identity operator. 

\begin{definition} [Colorful graph]
A colorful graph is a graph with some edges and weights colored with $P_x$'s or $Q_x$'s. Moreover, each edge or weight can have at most one color, and we will regard the ``colors" as another type of labels of the edges and weights. For the edges and weights with the same color, we group them together as a product and apply $P_x$ or $Q_x$ on them.
\end{definition} 

As an example, the following graph has two colors $Q_{x_1}$ and $Q_{x_2}$: 
 \be\label{fig6}
  \parbox[c]{0.36\linewidth}{\includegraphics[width=\linewidth]{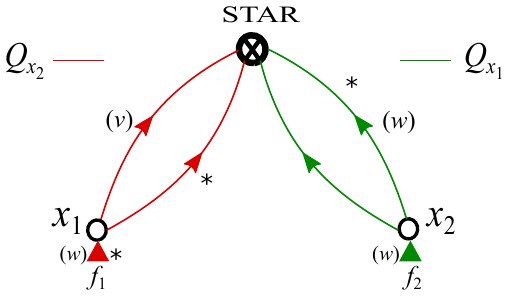}} = Q_{x_2}\left(G^{(v)}_{x_1\star}\overline G_{x_1\star}\overline G^{(w)}_{x_1x_1} \right) Q_{x_1}\left(G_{x_2\star}\overline G^{(w)}_{x_2\star}  (G^{(w)}_{x_2x_2})^{-1}\right).
 \ee

With the above graphical notations, we can express the resolvent expansions in \eqref{Gij Gijk} and \eqref{Gij Gijk2} as graph expansions as in the following lemma. Its proof is obvious.

\begin{lemma} \label{lem_exp1}
We have $G _{xy}=G_{xy}^{(w)}+\frac{G_{xw}G_{wy}}{G_{ww}}$, i.e.,   
  \be\label{Gp1}
   \parbox[c]{0.50\linewidth}{\includegraphics[width=\linewidth]{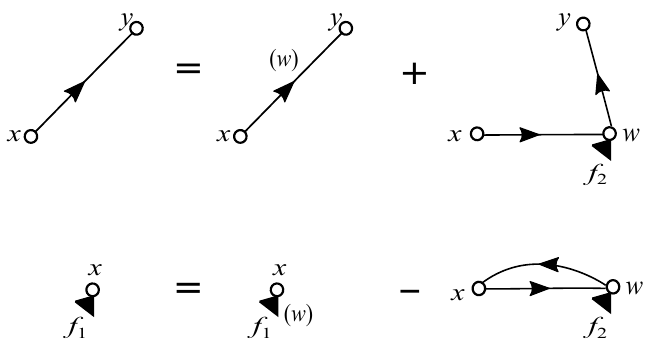}}
   \ee
We have $ (G _{xx})^{-1}=(G_{xx}^{(w)})^{-1} - \frac{G_{xw}G_{wx}}{G _{xx}G_{ww}G ^{(w)}_{xx}}$, i.e.,  
  \be\label{Gp2}
   \parbox[c]{0.45\linewidth}{\includegraphics[width=\linewidth]{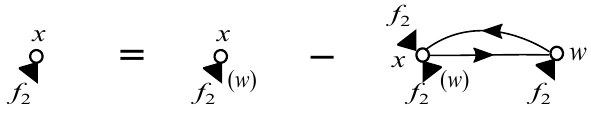}}
\ee  
 Here  the equality in each graph means the equality of the values of the graphs, not the equality in the graphical sense. These expansions preserve colors in the sense that after an resolvent expansion, each new component has the same color as its ancestor (i.e. the component from which it is expanded). 
 \end{lemma}

   
\begin{remark} 
In \eqref{Gp1} we have ``$+$" sign, while in \eqref{Gp2} we have ``$-$" sign. The $\pm$ signs are very hard to track, and actually they will not affect our proof. In the proof, we will try to be precise with the signs when we draw some specific graphs. However,  when we write or draw a general linear combination of graphs, we will always use the $+$ sign.  
\end{remark}


 
 
 \noindent{\bf Dashed \nc  edges:} In a graph, we use a dashed line connecting atoms $\al$ and $\beta$ to represent the factor $\delta_{\al\beta}$.  
For example, we have
 \be\label{fig3}
 \parbox[c]{0.26\linewidth}{\includegraphics[width=\linewidth]{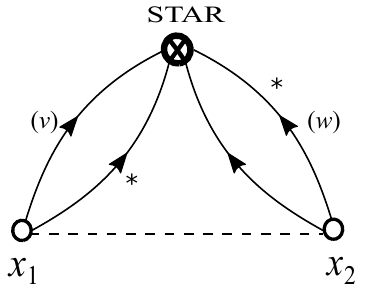}} \qquad =   \begin{cases}
   0 ,\quad  & x_1\ne x_2 \\
  G^{(v)}_{x_1\star}\overline G_{x_1\star}G_{x_1\star}\overline G^{(w)}_{x_1\star} , \quad & x_1=x_2 
  \end{cases}.
\ee
On the other hand, we use a dashed line with a cross ($\times$) to represent the $(1-\delta_{\al\beta})$ factor. The dashed lines and $\times$-dashed lines are useful in organizing the summation of indices represented by the regular atoms. For example, we can represent
$\sum_{x_1,x_2}G^{(v)}_{x_1\star}\overline G_{x_1\star} G_{x_2\star}\overline G^{(w)}_{x_2\star}$ 
by the graphs
\be\label{fig4}
 \parbox[c]{0.72\linewidth}{\includegraphics[width=\linewidth]{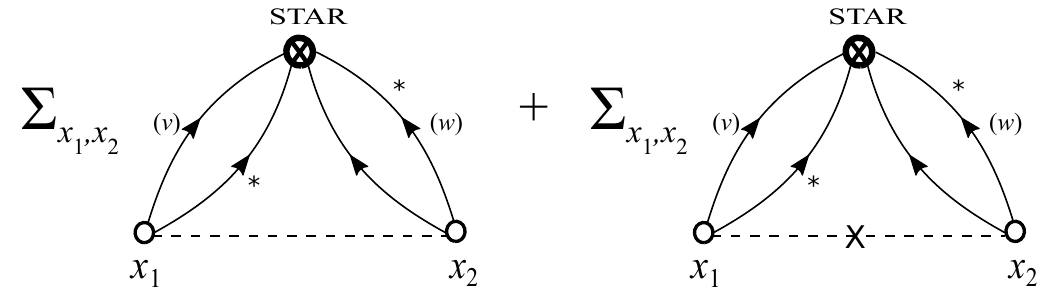}} 
 \ee
 For simplicity, 
 in the proof (not in the graph) we will also use the notation $x --- y$ (or $x -\times-y $)
to mean that {\it there is a dashed line connecting atoms $x$ and $y$} (or {\it there is a 
 $\times$-dashed line connecting atoms $x$ and $y$}). 

\vspace{5pt}

\noindent {\bf Dashed-line partition:}  Given a set of atoms $\{x_1, x_2, \cdots, x_n\}$. 
Let $\cal E_D$ be a collection of some dashed and $\times$-dashed edges between these atoms. We say $\cal E_D$ is a dashed-line partition of the atoms $\{x_1, x_2, \cdots, x_n\}$ if and only if it satisfies the following properties:
\begin{itemize}
\item {\bf completeness:} for any $i\ne j$, there is either a dashed edge or a $\times$-dashed edge in $\cal E_D$ between atoms $x_i$ and $x_j$;
\item {\bf self-consistency:} if $x_i$ and $x_j$ are connected by a dashed edge in $\cal E_D$, and $x_j$ and $x_l$ are also connected by a dashed edge in $\cal E_D$, then $x_i$ and $x_l$ must be connected by a dashed edge in $\cal E_D$. 
\end{itemize}
We say $\cal E_D$ is a dashed-line partition of a graph $\cal G$ if it is a dashed-line partition of all the atoms of graph $\cal G$.

\vspace{5pt}

For example, in the case $n=3$, we show the five possible dashed-line partitions in \eqref{G2sa}. The partitions with two dashed lines and one $\times$-dashed line are not {\it self-consistent}. 
\be\label{G2sa}
  \parbox[c]{0.60\linewidth}{\includegraphics[width=\linewidth]{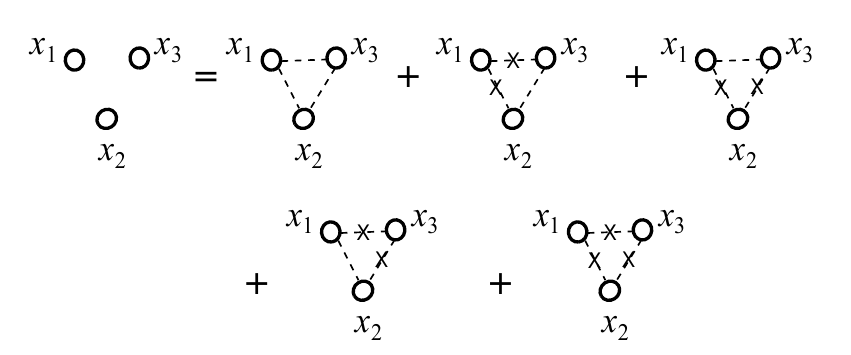}}
   \ee

\noindent {\bf Off-diagonal edges:} Let $\cal E_D$ be a dashed-line partition of a graph $\cal G$. If a solid edge connects atoms that are not equal under $\cal E_D$, then we shall call it an {\it{off-diagonal edge}}.

\vspace{5pt}

\noindent{\bf Fully expanded (fully independent):} Consider a subset of atoms $\{x_1, x_2, \cdots, x_n\}$ in graph $\cal G$. Let $\cal E_D$ be a dashed-line partition of $\cal G$. Then the restriction of $\cal E_D$ to $\{x_1, x_2, \cdots, x_n\}$ is a dashed-partition of these atoms. 

\begin{itemize}
\item We say a solid edge $Edge(\al, \beta)$ is fully expanded (fully independent) with respect to ($\{x_i\}_{i=1}^n , \cal E_D$) if its independent set union the end atoms $\{\alpha,\beta\}$ contains the set $\{x_1, x_2, \cdots, x_n\}$ after identification by $\cal E_D$. 
\item We say a weight on atom $\alpha$ is fully expanded (fully independent) with respect to ($\{x_i\}_{i=1}^n , \cal E_D$) if its independent set union the atom $\{\alpha\}$ contains the set $\{x_1, x_2, \cdots, x_n\}$ after identification by $\cal E_D$.
\end{itemize}

As defined above, if a solid edge $G_{\al\beta}^{(\mathbb T)}$ is fully expanded with respect to ($\{x_i\}_{i=1}^n , \cal E_D$), then $\mathbb T$ contains all the $x_i$ atoms which are non-equivalent to $\al$ or $\beta$ under $\mathcal E_D$. Similar property holds for weights. 
 
\vspace{5pt}

\noindent{\bf Independent of an atom:} Given a dashed-line partition $\mathcal E_D$ of atoms $\{x_1, x_2, \cdots, x_n\}$, we say that an edge (or a weight) is {\it{independent of atom $x_i$}} if the independent set of the edge (or the weight) contains an atom that is equivalent to $x_i$ under $\mathcal E_D$. In other words, an edge (or a weight) is said to be independent of an atom $x_i$ if it is independent of the $x_i$-th  row and column of $H$. Note that if a subgraph $\mathcal G$ is independent of atom $x_i$, then we have $P_{x_i}\mathcal G=\mathcal G$ and $Q_{x_i}\mathcal G=0$.

\vspace{5pt}

As discussed in Section \ref{sec_2level}, we next define the concept of molecules.

\vspace{3pt}


\begin{definition}[Molecules and Polymers]\label{def_poly}
{\rm{(i)}} {\bf Molecules:}  We partition the set of all the regular atoms into a union of disjoint sets $\cal M_j,$ $j=1 \ldots, p$. We shall call each 
 $\cal M_j$ a ``molecule" (even though the atoms in $\cal M_j$ may not be edge-connected). 
More precisely, the {\it{molecules}} are subsets of atoms that satisfy 
\be\label{yurenguodu}
\star \notin \cup_{j=1}^p\cal M_j,\quad \quad \{\star\} \cup \left(\cup_{j=1}^p\cal M_j\right)= \{\text{all atoms}\},\quad \text{ and }\ \ \cal M_i\cap \cal M_j=\emptyset\ \ \text{for } 1\le i < j\le p .
\ee

\vspace{5pt}

\noindent{\rm{(ii)}} {\bf Polymers:} Let $\cal G$ be a graph with $p$ molecules such that \eqref{yurenguodu} holds. We use the notations 
$$(1):  \cal M_i ---\cal M_j \quad \text{and} \quad (2):  \cal M_i--- \, \star$$
to mean that (1) there is a dashed line connecting an atom in $\cal M_i$ to an atom in $\cal M_j$,  and (2) there is a dashed line connecting an atom in $\cal M_i$ to the $\star$ atom. Then we define two subsets of molecules, $\cal Pol_1(\cal G)$ and $\cal Pol_2(\cal G)$, called ``polymers". A molecule $\cal M_i$ belongs to $\cal Pol_1(\cal G)$ if and only if there exists $\cal M_{i_1},\cdots, \cal M_{i_n}$ such that
$$ \cal M_i ---\cal M_{i_1} --- \cal M_{i_2} --- \cdots ---  \cal M_{i_n}--- \star \ .$$
Simply speaking, $\cal Pol_1(\cal G)$ consists of all the molecules that are connected to the star atom through a path of dashed lines. A molecule $\cal M_i$ belongs to $\cal Pol_2(\cal G)$ if and only if $\cal M_i \notin \cal Pol_1(\cal G)$ and there exists another $\cal M_j \notin \cal Pol_1(\cal G)$ such that $  \cal M_i---   \cal M_j.$
In other words, $\cal Pol_2(\cal G)$ consists of all the molecules that are not in $\cal Pol_1(\cal G)$ and have at least one dashed line-connected neighborhood. 

\vspace{5pt}

\noindent{\rm{(iii)}} {\bf Free molecules:} We say a molecule $\cal M_i$ is {\it free} if and only if  
$\cal M_i\notin \cal Pol_1\cup\cal Pol_2$. 
\end{definition}

For example, in Fig.\,\ref{polyexample}, we have
$$
\cal Pol_1=\{\cal M_4, \cal M_5\}, \quad \cal Pol_2=\{\cal M_3, \cal M_6, \cal M_7, \cal M_8\}.
$$
\begin{figure}[htb]
\centering
\includegraphics[width=12cm]{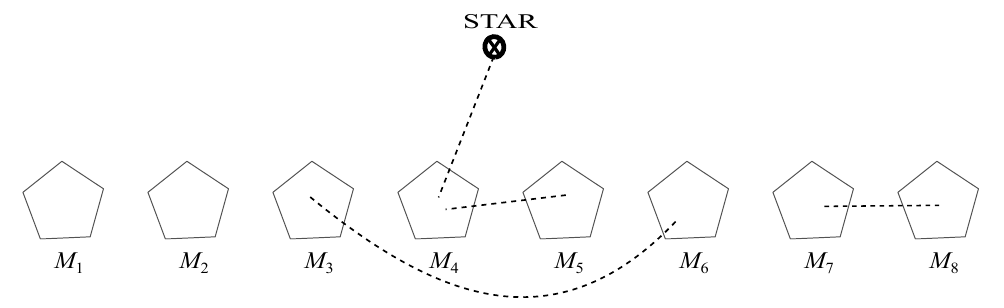}
\caption{We use pentagons to represent the molecules, and we only show the dashed lines between them. The pentagons do not really appear in the graph, and they are only drawn to help to understand the structures. }
\label{polyexample}
\end{figure}



\hspace{-0.27in} {\bf Degree:} Let  $\cal A$ denote any set of atoms in the graph. We define 
\be\label{degree}
{\deg}(\cal A) := \# \text{ of solid edges which connect atoms in } \cal A \text{ and } \cal A^c, 
\ee
i.e., the total number of solid edges which have one ending atom in $\cal A$ and the other one in $\cal A^c$. In particular, for any atom $x$, $\deg(x)$ denotes the number of solid edges attached to $x$, and for any molecule $\cal M_i$, $\deg (\cal M_i)$ denotes the number of solid edges that connect the atoms in $\cal M_i$ to the atoms outside the molecule. 
  
\vspace{5pt}
In the rest of this subsection, we introduce one of the most important graphical properties for the proof---the nested property of molecules.

\vspace{3pt}

\hspace{-0.27in} {\bf Path:}  Let $\cal G$ be a graph with molecules $\cal M_i$, $1\le i\le p$, such that \eqref{yurenguodu} holds. For some $1\le i\le p$, we say that there is {\bf a path from molecule $\cal M_i$ to $\star$}
if and only if there is a {\it{solid edge path}} connecting $\cal M_i$ to $\star$ in the molecular graph. In other words, the path is defined on the new graph where each molecule is viewed as a vertex. 
  
For example, let $\cal M_1=\{x_1\}$ and $\cal M_2=\{x_2,x_3\}$ in the following graph \eqref{pathexample}. Although there is no edge between atoms $x_2$ and $x_3$, there are still {\it{2 separated}} paths connecting $\cal M_1$ to $\star$, i.e., through $Edge(x_1,\star)$, and through $Edge(x_1,x_2)$ and $Edge(x_3,\star)$. Similarly, it is easy to see that there are {\it{3 separated}} paths connecting $\cal M_2$ to $\star$.
 \be\label{pathexample}
  \parbox[c]{0.40\linewidth}{\includegraphics[width=\linewidth]{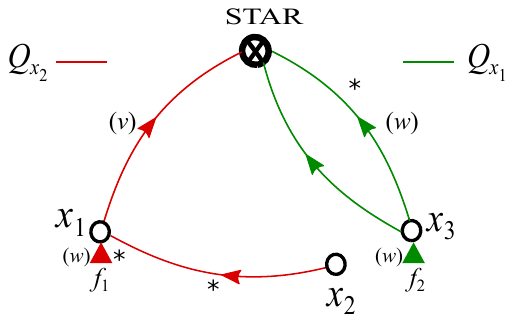}}
   \ee

 \begin{definition}[IPC Nested property] \label{def_nest}
 For a graph $\cal G$ with molecules $\cal M_i$, $1\le i\le p$, that satisfy \eqref{yurenguodu}, we say that it satisfies the {\bf independently path-connected (IPC) nested property} if
\begin{itemize}
\item for each molecule, there are at least 2 separated (solid edge) paths connecting it to $\star$;
\item the edges used in these $2p$ paths are all distinct. 
\end{itemize}
If a graph $\cal G$ satisfies the IPC nested property, then we say that it has an {\bf IPC nested structure}.
\end{definition}

For example, the graph in \eqref{pathexample} does not have an IPC nested structure, but the following one does. 
 \be
   \parbox[c]{0.40\linewidth}{\includegraphics[width=\linewidth]{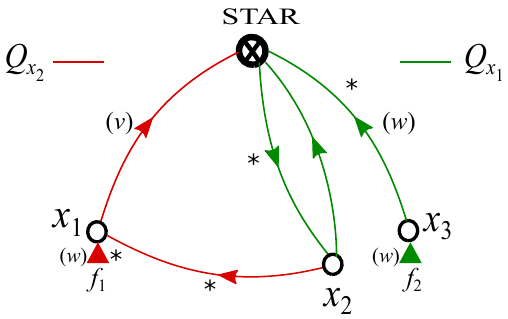}}
  \ee
 
 The IPC nested property implies the ordered nested property as discussed in the introduction. 

 \begin{lemma}[Ordered nested property]\label{zuomeng}
 Let $\cal G$ be a graph with a $\star$ atom and $p$ molecules that satisfy \eqref{yurenguodu}. Suppose $\cal G$ satisfies the IPC nested property, then it also satisfies the following {\bf ordered nested property}: 
 for any $t\in \mathbb N$, $1\le t\le p$, there exists $\pi=(\pi_1, \pi_2, \cdots,\pi_t) \in S_t$, the permutation group, such that 
\begin{equation}\label{suz}
\begin{split}
\text{\bf Ordered nested property}: \quad & \forall s\le t, \text{ there exist at least 2 solid edges connecting atoms in $\cal M_{\pi_s}$ }
 \\  
& \text{to  atoms in $ \{\star\} \cup\left(\cup_{s'<s}\cal M_{\pi_{s'}}\right)\cup \left(\cup_{t'>t}\cal M_{t'}\right)$.}
\end{split}
\end{equation}
\end{lemma}
\begin{remark}
Given the first $t$ molecules and $\pi \in S_t$, we can partially order them according to $ \cal M_{\pi_1} \preceq \cal M_{\pi_2} \preceq \cdots \preceq \cal M_{\pi_t}$. For the $\star$ atom and other molecules $\cal M_{t'}$, $t<t'\le p$, we define the partial order $\star \preceq \cal M_{t'} \preceq \cal M_{\pi_1}$ such that they are lower bounds of the subset $\{\cal M_1,\cdots, \cal M_t\}$. Then roughly speaking, the {\it{ordered nested property}} means that for any fixed $1\le t \le p$, there exists an order given by $\pi \in S_t$ such that each of the molecule $\cal M_s$, $1\le s\le t$, has at least 2 solid edges connecting to the preceding molecules. Note that $\cal M_s$, $1\le s\le t$, may or may not have solid edges connecting to molecules after it.
\end{remark}
 
\begin{proof} [Proof of Lemma \ref{zuomeng}]
A simple application of the pigeonhole principle shows that a graph with IPC nested property can always be rearranged to have the ordered nested property. Here we skip the details and leave it to the reader.  
\end{proof}

%

The Fig.\,\ref{nestexample} gives an example of the ordered nested property with $t=5$ and $\pi=(2,4,3,5,1)\in S_5$.  Note that the choice of $\pi$ is not unique for the ordered nested property. For example, we can also choose $\pi=(1, 4, 2, 3,5) $.
 \begin{figure}[htb]
\centering
\includegraphics[width=10cm]{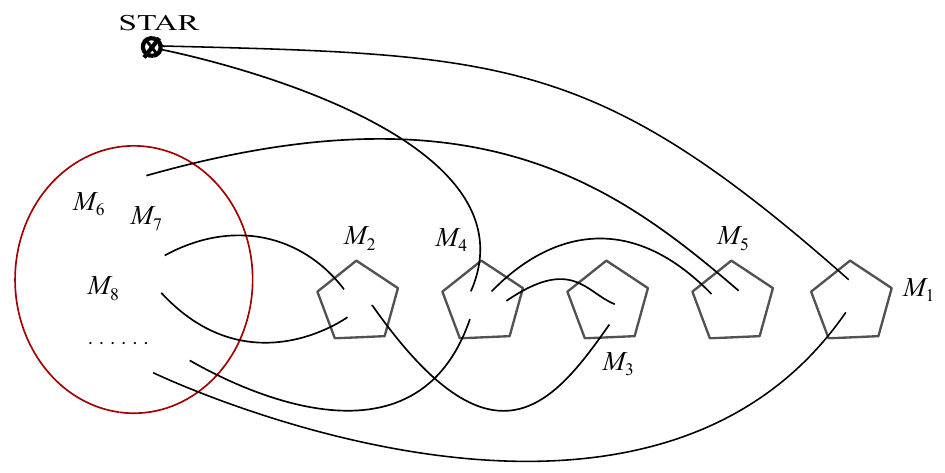}
\caption{The pentagons represent the molecules $\cal M_1,\cdots,\cal M_5$, the red circle represents the part $\left(\cup_{t'>t}\cal M_{t'}\right)$, and we only draw the solid edges used in the ordered nested property.}
\label{nestexample}
\end{figure} 
Given a graph with a large number of vertices, it is usually not easy to check whether the ordered nested property holds or not. 
To make things worse, after each expansion the order of vertices can be totally different, which makes the ordered nested property hard to track under graph expansions. On the other hand, the IPC nested property is often much easier to track. In particular, the following lemma shows that the IPC nested property is preserved under the resolvent expansions in \eqref{Gp1} and \eqref{Gp2}. Then Lemma \ref{zuomeng} guarantees that we have the desired ordered nested structure at each step of the proof.

 \begin{lemma}\label{Lumm} 
 Let $\cal G$ be a graph with $p$ molecules $\cal M_i$, $1\le i\le p$, that satisfy \eqref{yurenguodu}. Suppose $\mathcal G$ satisfies the IPC nested property. If we expand an edge or a weight in $\cal G$ using \eqref{Gp1}-\eqref{Gp2} and  
denote the resulting two graphs as
$$\cal G = \cal G_1+\cal G_2,$$
 then both $\cal G_1$ and $\cal G_2$ have IPC nested structures. 
  \end{lemma}
 \begin{proof} It follows trivially from the definition of the IPC nested property and the graph expansions in \eqref{Gp1} and \eqref{Gp2}. In fact, we observe that in the expansions \eqref{Gp1}-\eqref{Gp2}, we always replace an edge between two atoms with a path between the same two atoms. In particular, the path connectivity from any atom to the $\star$ atom is unchanged.
 \end{proof}


  \subsection{Proof of Lemma \ref{Q2}.} \label{sec_simple} \
In this subsection, we prove Lemma \ref{Q2} using the graphical tools introduced in last subsection. It suffices to prove the following bound for \eqref{Qsimz}:
\be\label{sfayz}
 \E \sum_{x_1,\cdots, x_p}^{(\star)}  \left(\prod_{i=1}^p c_{x_i} Q_{x_i} \left(G _{{x_i}\star} \overline G _{{x_i}\star}\right)  \right)
 \prec \OO_\tau \left(\Gamma^2\Phi\right)^{p}.
  \ee
Let $\cal G $ be the graph that represents 
\be\label{Gform}
\cal G = \prod_{i=1}^p Q_{x_i}   \left(G _{{x_i}\star} \overline G _{{x_i}\star}\right) .
\ee
For example, in the case $p=3$, we have
\be\label{G0-s}
\parbox[c]{0.36\linewidth}{\includegraphics[width=\linewidth]{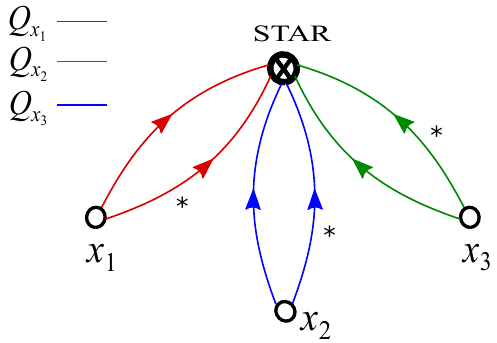}}
\ee
 Then $\cal G $ can be written as the sum of $ \cal E_D  \cdot \cal G$, where $\cal E_D$ ranges over all possible dashed-line partitions of the atoms $x_1, x_2, \cdots, x_p$. Since there are only $C_p$ different partitions, where $C_p>0$ is a constant depending only on $p$, we only need to prove that for any fixed $\cal E_D $,
\be\label{yikezz}
  \E \sum_{x_1,\cdots, x_p}^{(\star)} \left( \prod_{i=1}^p c_{x_i} \right)
\left( \cal E_D  \cdot \cal G \right)
 \prec\OO_\tau\left(\Gamma^2\Phi\right)^{p}.
  \ee
Now we expand the edges in $\cal G$ 
using the expansions \eqref{Gp1} and \eqref{Gp2} with respect to the $\left(\{x_i\}_{i=1}^p, \cal E_D\right)$ as following. 

\vspace{10pt}

\noindent{\bf Expansions with respect to ($\{x_i\}_{i=1}^p, \cal E_D)$:}  
For a graph $\cal G$, if all of its solid edges or weights are already fully expanded with respect to ($\{x_i\}_{i=1}^p, \cal E_D)$, then we stop. Otherwise, we can find a non-fully expanded solid edge $Edge(\al, \beta)$ or a non-fully expanded weight on atom $\al$. Then there exists a $x_i$, $1\le i\le p$, such that $x_i$ is not equal to any ending atom $\alpha,\beta$ of the solid edge (or the atom $\al$ of the weight) and is not equal to any atom in the independent set of the solid edge (or the weight), either. Then we expand this edge $Edge(\al, \beta)$ (or the weight on atom $\al$) using \eqref{Gp1} or \eqref{Gp2} with $x_i$ playing the role of the atom $w$.   
%
%
%
For instance, for the solid edge representing $G_{x_1\star}$ in the graph, the independent set is $\emptyset$. Hence we only need to find a $x_i$ atom such that there is a $\times$-dashed line in $\cal E_D$ that connects $x_i$ and $x_1$. If there is no such $x_i$, then we leave it unchanged. Otherwise,  we expand $G_{x_1\star}$ with \eqref{Gp1} as
 $$ G_{x_1\star}=G^{(x_i)}_{x_1\star}+G_{x_1x_i}\left(G_{x_ix_i}\right)^{-1}G_{x_i\star}.$$
 
After an expansion, every old graph is either unchanged or can be written as a linear combination of two new graphs. Then for each new graph, if there exists a non-fully expanded solid edge or weight, we again expand it with respect to some $x_i$ using \eqref{Gp1} or \eqref{Gp2}. We keep performing the same process to the newly appeared graphs at each step, and call this process the {\bf expansions with respect to ($\{x_i\}_{i=1}^p, \cal E_D)$}.
The following is an example with $p=2$ and two steps of expansions:
\be\label{Expan1}
\parbox[c]{0.86\linewidth}{\includegraphics[width=\linewidth]{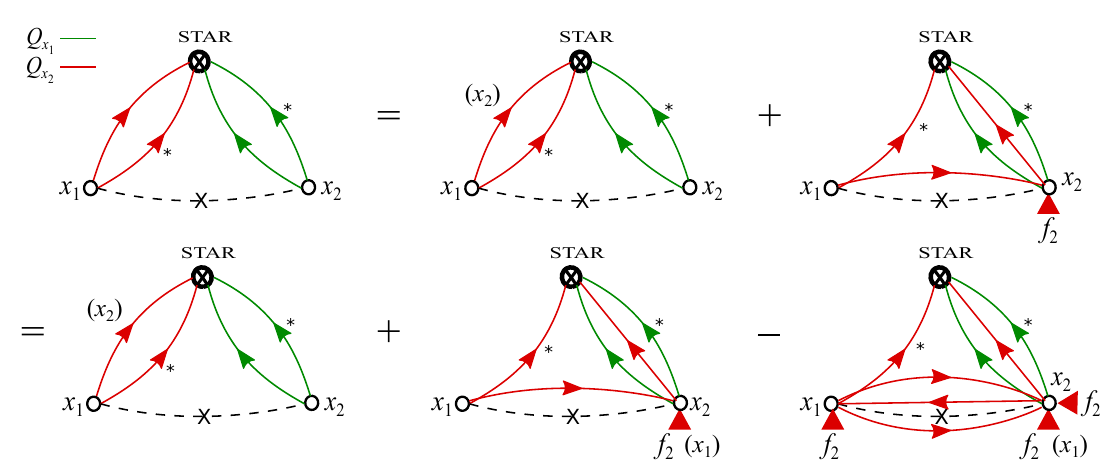}}
   \ee 
Here in the first step, we expanded $G_{x_1\star}$ with respect to $x_2$ using \eqref{Gp1}, and in the second step we expanded $(G_{x_2x_2})^{-1}$ with respect to $x_1$ using \eqref{Gp2}. (We also need to expand the first graph in the second row, but we did not draw it for simplicity.) Note that in the second row of \eqref{Expan1}, the leftmost red solid edge is fully expanded, and the red weight in the middle graph is also full expanded. 

During the process of expansions with respect to ($\{x_i\}_{i=1}^p, \cal E_D)$, it is easy to see that 
in every step of the expansion, each new graph satisfies one of the following conditions: 
\begin{itemize}
\item either everything in the new graph is the same as the old graph except that the size of the independent set of some solid edge/weight is increased by one, 

\item or some solid edge/weight in the old graph is replaced by some other (path of) solid edges and weights in the new graph, and the total number of solid edges are increased at least by one.
\end{itemize}
By definition, in the latter case, the newly appeared solid edges are all off-diagonal under $\cal E_D$. Hence every new solid edge provides a factor $ \Phi $, and graphs with sufficiently many off-diagonal edges will be small enough to be considered as error terms. 


Now we perform the expansions of the graphs $\cal E_D \cdot \cal G$ in \eqref{yikezz} with respect to ($\{x_i\}_{i=1}^p, \cal E_D)$. With the above observation, it is easy to show that for any fixed (large) $M\in \N$, there exists a constant $K_{M,p}\in \N$ such that after $K_{M,p}$ steps of expansions the following holds:
\be\label{Ejuzy}
  \cal E_D  \cdot \cal G = \sum_{\gamma } \cal E_D  \cdot \cal G _{  \gamma }+ \sum_{  \gamma' }\cal E_D  \cdot 
  \cal G _{   \gamma' }^{\text{\it error}},
\ee
where (1) each $\cal G _{  \gamma }$ is fully expanded with respect to ($\{x_i\}_{i=1}^p, \cal E_D)$, (2) each graph $\cal G _{  \gamma'}^{\it error}$ contains at least $M$ off-diagonal edges, and (3) the total number of terms on the right-hand side of (\ref{Ejuzy}) is bounded by some constant $C_{M,p}>0$. Note that the expansion in (\ref{Ejuzy}) is not unique, 
but any expansion with the above properties (1)-(3) will work for our proof. 
 

Since $M$ can be arbitrary large, 
to prove \eqref{yikezz} it suffices to show that for any fixed $\cal E_D$ 
and $\gamma $, 
\be\label{yikezz2}
  \E \sum_{x_1,\cdots ,x_p}^{(\star)} \left( \prod_{i=1}^p c_{x_i} \right)
\left( \cal E_D  \cdot \cal G_{ \gamma }\right)
 \prec \OO_\tau \left(\Gamma^2 \Phi \right)^{p}.
  \ee
For a graph $\cal E_D  \cdot \cal G_{ \gamma }$, we choose the molecules as 
$$  \cal M_i=\{x_i\}, \quad 1\le i\le p.$$ 
Since $x_1,\cdots, x_p \ne \star$, we have $\cal Pol_1(  \cal E_D  \cdot \cal G_{ \gamma } )=\emptyset$. Without loss of generality, we assume that 
\be\label{deftxxah}
 \cal Pol_2\left(  \cal E_D  \cdot \cal G_{ \gamma }\right)=\{\cal M_{t+1},\cal M_{t+2}, \cdots,  \cal M_{p}\}, 
 \ee 
for some $0\le t\le p$. 
Recall that $\cal G_{\gamma }$ comes from  $\cal G$, which has the form \eqref{Gform}. Then by the definition of our expansion process, 
one can easily see that $\cal G_\gamma$ has the following form: 
 \be\label{g0gam}
  \cal G_{ \gamma }=\prod_{i=1}^{p}Q_{x_i} \left(\cal G_{ \gamma ,i} \right),
  \ee
where $\cal G_{ \gamma,i} $ denotes the part of the graph coming from the expansions of $G _{{x_i}\star} \overline G _{{x_i}\star}$ inside the $Q_{x_i}$ in \eqref{Gform}. (Note that $\cal G_{ \gamma,i}$ is a colorless graph.) 
Now we pick some $i\ne j$ such that $x_i\ne x_j$, i.e., there is a $\times$-dashed line in $\cal E_D$ that connects $x_i$ and $x_j$. Since all the edges and weights in $\cal G_{ \gamma,i}$ are fully expanded, if the atom $x_j$ is not visited in $\cal G_{ \gamma,i} $ (i.e., if $x_j$ is not an ending atom of some edge in $\cal G_{ \gamma,i}$), then $x_j$ must be in the independent set of every edge and weight in $\cal G_{ \gamma,i}$. In other words, 
 $$ \text{atom $x_j$ is not visited in $\cal G_{ \gamma,i}$} \implies \text{$\cal G_{ \gamma,i}$ is independent of the atom $x_j$.} $$
 Now for any free molecule $\cal M_{j}$, $1\le j\le t$, we have $x_j \ne x_i$ for all $i\ne j$. Then we can extend the above statement to get
\begin{equation*}
\begin{split}
 &\text{atom $x_j$ is not visited in $\prod_{i\ne j} \cal G_{ \gamma,i}$}\\
 \implies &\text{ $\prod_{i\ne j} \cal G_{ \gamma,i}$ is independent of the atom $x_j$}\\
  \implies &\text{ $\prod_{i\ne j} Q_{x_i}(\cal G_{ \gamma,i})$ is independent of the atom $x_j$} \\
 \implies &\E\,  \cal G_{ \gamma}
 =\E\, P_{x_j} \cal G_{ \gamma}
 =\E \Big(\Big(\prod_{i\ne j} Q_{x_i}(\cal G_{ \gamma,i}) \Big)\cdot P_{x_j}\left(Q_{x_j}\cal G_{ \gamma , j}\right)\Big)=0,
 \end{split}
\end{equation*}
where  we used $ P_{x_j} Q_{x_j}=0$ in the last step. Therefore we only need to consider the
graphs in which
\be\label{gaszj}
 \text{for any } j: \ 1\le j\le t, \ \text{atom $x_j$ is visited in $\prod_{i\ne j} \cal G_{ \gamma , i} $}.    
\ee

Now it is instructive to count the number of off-diagonal edges. Recall that initially there are $2p$ off-diagonal edges in $\mathcal G$ (see \eqref{Gform}), and the newly appeared solid edges during the expansions are all off-diagonal.
Therefore, all the solid edges in $\prod_{i=1}^{p} \cal G_{ \gamma,i}$ are off-diagonal, and each of them provides a factor $ \Phi $. Note that if the free molecule $\mathcal M_{j}$ is visited $\prod_{i\ne j} \cal G_{ \gamma , i} $, then we must have used \eqref{Gp1} or \eqref{Gp2} in some step of expansion and picked the graph with at least one more off-diagonal edge. Thus \eqref{gaszj} implies that we must have at least $t$ more solid edges. This gives that 
\be\label{2p+t}
\text{\# of off-diagonal edges in } \prod_{i=1}^{p} \cal G_{ \gamma,i} \ge 2p+t . 
\ee
These extra $t$ off-diagonal edges are crucial to our proof.

Next we find a high probability bound on the graph $\cal G_\gamma$ in \eqref{g0gam}. Assume that, ignoring the directions and charges, $\cal G_\gamma$ consists of some weights and $m$ solid edges
 \be\label{solid_notes}
 {Edge}^{(\mathbb T_i)}(\al_{i}, \beta_{i}) \quad \text{with color $Q_{x_i}$ or $P_{x_i}$}, \quad 1\le i\le m ,
 \ee
 where some $x_i$ can be $\emptyset$, i.e. some edges are colorless.
Then with \eqref{chsz2}, \eqref{chsz2.5} and Lemma \ref{lem_partial}, for any fixed constants $0< \tau' <\tau$, we can bound $\mathcal G_\gamma$ as
\begin{equation}\label{def F1}
\begin{split}
\mathcal G_\gamma &\prec \prod_{i=1}^m \left(\Psi_{\al_i\beta_i}^{(\mathbb T_i x_i)}( \tau') + \delta_{\al_i\beta_i}\right) \\
&\prec \OO_\tau \left( \prod_{i=1}^m \left(\Psi_{\al_i\beta_i} (\tau)+ \sum_{k=1}^{|\mathbb T_i \cup \{x_i\}|} \sum_{(w_1 ,\cdots, w_k)\in \mathcal P_k(\mathbb T_i \cup \{x_i\})}\Psi_{(\al_i, w_1 ,\cdots, w_k, \beta_i)}(\tau) + \delta_{\al_i\beta_i}\right) \right).
\end{split}
\end{equation}
Inspired by \eqref{def F1}, given any dashed-line partition $\cal E_D$, we define the $\Psi$-graphs of $\mathcal G_\gamma$. 
\begin{definition}[$\Psi$-graphs]\label{Psi-graph}
Each $\Psi$-graph is a colorless graph consists of
\begin{itemize}
\item a star atom and regular atoms, 
\item solid edges without labels, where a solid edge connecting $\al$ and $\beta$ atoms represents a $\Psi_{\al\beta}$ factor if $\al-\times-\beta$ in $\mathcal E_D$ and a factor  $1$ otherwise, 
\item dashed edges,  
\item and no weights.
\end{itemize}
Moreover, each $\Psi$-graph is obtained from $\cal G_\gamma$ by 
\begin{itemize}
\item removing all the weights,

\item replacing each edge in \eqref{solid_notes} with a solid edge path 
\be \label{psi_paths}
(\al_i, w_1 ,\cdots, w_k, \beta_i) \quad \text{ with }\quad (w_1 ,\cdots, w_k)\in \mathcal P_k(\mathbb T_i \cup \{x_i\}),
\ee
 
\item removing all the labels including directions, charges, independent sets and colors,

\item and keeping all the dashed edges.
\end{itemize}
Again we define the value of each $\Psi$-graph as the product of all the factors represented by its elements, and we will always identify a $\Psi$ graph with its value in the following proof.
\end{definition}

Note that from any graph $\mathcal G_\gamma$, we can produce multiple $\Psi$-graphs by choosing different paths \eqref{psi_paths} for the edges in \eqref{solid_notes}. Since each independent set contains fewer than $p$ atoms and there are at most $C_{M,p}$ many solid edges for some constant $C_{M,p}>0$ (where $M$ is defined above \eqref{Ejuzy}), the number of different $\Psi$-graphs is bounded by some constant $\wt C_{M,p}>0$. Then we label these graphs as $\Psi(\cal G_\gamma, \mathcal E_D, {\bm\xi})$. Note that ${\bm \xi} \equiv {\bm\xi}(\cal G_\gamma,\mathcal E_D)$, and one can also regard it as the label for the paths in \eqref{psi_paths}.

Recall that $\mathcal G_\gamma$ has $p$ molecules $\cal M_i$, $1\le i\le p$, with $\cal M_1, \cdots, \cal M_t$ being free molecules. Then we define the molecules of $\Psi(G_\gamma, \mathcal E_D, {\bm\xi})$, called {\it $\Psi$-molecules}. 

\begin{definition}[$\Psi$-molecules]
For the non-free molecules of $\cal G_\gamma$ that are connected through dashed-lines, we combine them into one single molecule in $\Psi(\cal G_\gamma, \mathcal E_D, {\bm\xi})$ (note that we do not combine atoms). We shall call a $\Psi$-molecule {\it non-free} if it is connected to the $\star$ atom through a dashed line; otherwise we call it {\it free}. 
\end{definition}

We can also define $\Psi$-polymers as in Definition \ref{def_poly} for $\Psi$-molecules. However, for the molecules in $\cal Pol_2$ that are connected to each other through dashed-lines, we have combined them into one bigger $\Psi$-molecule, which is now free in $\Psi(\cal G_\gamma, \mathcal E_D, {\bm\xi})$. On the other hand, for the molecules in $\cal Pol_1$ that are connected through dashed-lines, we also combine them into one single $\Psi$-molecule. Hence a $\Psi$-molecule is either free or is connected with the $\star$ atom directly with dashed edges, which shows that there is really no need to introduce the concept of $\Psi$-polymers.

%
%
%
 
 In the proof of this subsection, all the $\Psi$-molecules are free since $\cal Pol_1= \emptyset$ in $\mathcal G_\gamma$. 
Now we keep all the free molecules ${\cal M}^{\Psi}_{s}=\cal M_s$, $1\le s \le t$, and denote the new $\Psi$-molecules by ${\cal M}^{\Psi}_{t+1}, \cdots, {\cal M}^{\Psi}_{t+r}$. 
%
%
%
From the above definitions and \eqref{def F1}, one can immediately obtain the following lemma.

\begin{lemma}\label{lem F1}
We have 
\be\label{def F10}
\mathcal E_D\cdot \cal G_\gamma \prec \OO_\tau\Big(\sum_{\bm \xi} \Psi(\cal G_\gamma, \mathcal E_D, {\bm\xi})\Big) .
\ee
Moreover, each $\Psi(\cal G_\gamma, \mathcal E_D, {\bm\xi})$ satisfies the IPC nested property with $\Psi$-molecules ${\cal M}^{\Psi}_1, \cdots , {\cal M}^{\Psi}_{t+r}$. 

\end{lemma}
\begin{proof} 
The bound \eqref{def F10} follows from \eqref{def F1}. Since the initial graph $\cal G$ in \eqref{Gform} has an IPC nested structure, then by Lemma \ref{Lumm}, $\cal G_\gamma$ also has an IPC nested structure. In the definition of the $\Psi$-graphs, we always replace an edge between two atoms as in \eqref{solid_notes} with a path between the same two atoms as in \eqref{psi_paths}. In particular, the path connectivity from any atom to the $\star$ atom is unchanged. Hence each $\Psi(\cal G_\gamma, \mathcal E_D, {\bm\xi})$ also satisfies the IPC nested property with the molecules $\cal M_1, \cdots, {\cal M}_{p}$. Finally, combining the non-free molecules into $\Psi$-molecules does not break the IPC nested structure. This finishes the proof. 
\end{proof}

Since the number of $\Psi$-graphs is bounded by $\wt C_{M,p}$, to conclude \eqref{yikezz2} it suffices to prove that 
\be\label{yikezz2EE2}
\sum^{(*)}_{x_1, x_2,\cdots, x_p } \left(\prod_{i=1}^{p}  c_{x_i} \right)  \E\Psi(\cal G_\gamma, \mathcal E_D, {\bm\xi})
\prec \OO_\tau\left(\Gamma^2\Phi\right)^{p}.
\ee 
By Lemma \ref{zuomeng}, $\Psi(\cal G_\gamma, \mathcal E_D, {\bm\xi})$ satisfies the ordered nested property, hence there exists $\pi\in S_{t+r}$ such that \eqref{suz} holds for $\Psi$-molecules.  Without loss of generality, we assume that $\pi=(1,2,\cdots,  t+r )\in S_{t+r}$. Then \eqref{suz} shows that for each $1\le s\le t+r$,  
\be\label{ljsadfl}
\begin{split}
& \text{there exist } \beta_s, \wt\beta_s\in   \{\star\} \cup \left(\cup_{s'<s} {\cal M}^\Psi_{ s' }\right) \text{such that there are two off-diagonal edges }\\
& \text{connecting the atoms in ${\cal M}^\Psi_s$ to $\beta_s$ and $\wt\beta_s$, respectively.}
\end{split}
\ee
We now estimate $\Psi(\cal G_\gamma, \mathcal E_D, {\bm\xi})$. There are at least $2p+t$ off-diagonal solid edges by \eqref{2p+t}, and the above ordered nested property used $2(t+r)$ of them. Each of the other $2p-t-2r$ solid edges is bounded by $\OO_{\tau,\prec}(\Phi)$. 
Note that we have only combined two molecules only if their $x$ atoms {\it always take the same value}.  
Hence we have that 
$$\Psi(\cal G_\gamma, \mathcal E_D, {\bm\xi})\prec \OO_{\tau}\Big(\Phi^{2p-t-2r}\cdot \prod_{1\le s\le t+r} \Psi_{\wt x_{s} y_s}\Psi_{\wt x_{s} \wt y_s}\Big),$$
where $\wt x_s$ is an atom in ${\cal M}^\Psi_s$, and $y_s$ ($\wt y_s$) belongs to the same molecule as $\beta_s$ ($\wt \beta_s$) and $y_s, \wt y_s \in \{\star,\wt x_1, \cdots, \wt x_{t+r}\}$. 

Plugging it into the left-hand side of \eqref{yikezz2EE2}, we obtain that
\begin{align}
 & \sum^{(*)}_{x_1,\cdots, x_p } \left(\prod_i^{p} c_{x_i } \right) \E\Psi(\cal G_\gamma, \mathcal E_D, {\bm\xi}) \nonumber \prec \OO_{ \tau} \left(\E \sum_{\wt x_1,\cdots, \wt x_{t+r}}^{(\star)} \Phi^{2p-t-2r}\cdot \prod_{1\le s\le t+r} \Psi_{\wt x_{s} y_s}\Psi_{\wt x_{s} \wt y_s} \right) \nonumber\\
 &\prec \OO_{ \tau} \left( \Phi^{2p-t-2r} \E \sum_{\wt x_1}  \Psi_{\wt x_{1} y_1} \Psi_{\wt x_{1} \wt y_1}\sum_{\wt x_2}  \Psi_{\wt x_{2} y_2} \Psi_{\wt x_{2} \wt y_2}  \cdots   \sum_{\wt x_{t+r}}\Psi_{\wt x_{t+r} y_{t+r}} \Psi_{\wt x_{t+r} \wt y_{t+r}}    \right) \nonumber \\
 & \prec \OO_{ \tau} \left( \Phi^{2p-t-2r} \Gamma^{2(t+r)}\right)  , \label{phir0}
\end{align}
 where in the second step we used \eqref{ljsadfl} such that one can sum over the $\wt x$'s according to the order $\wt x_{t+r}, \cdots, \wt x_1$, and in the third step we used \eqref{chsz4} to get a $\Gamma^2$ factor for each sum. Since $\Phi \ll 1$ and $\Gamma\ge 1$ by \eqref{GM1.5}, the factor $\Phi^{2p-t-2r} \Gamma^{2(t+r)}$ increases as $r$ increases. However, since we have only combined the molecules in $\cal Pol_2$, we must have $2r \le p-t$. Hence we can bound \eqref{phir0} by 
 $$\eqref{phir} \prec \OO_{ \tau} \left( \Phi^{2p-t-(p-t)} \Gamma^{2t+(p-t)}\right) \prec \OO_{ \tau} \left( \Gamma^{2p}\Phi^{p} \right), $$
where we used $t\le p$ in the second step. This proves \eqref{yikezz2EE2}, which finishes the proof of Lemma \ref{Q2}.

\subsection{Proof of Lemma \ref{Q1}. }\label{sec_simple2}\
The proof of Lemma \ref{Q1} is very similar to the one for Lemma \ref{Q2} in the previous subsection.  As in \eqref{sfayz}, 
we need to prove that 
  $$
  \sum_{x_1, x_2,\cdots, x_p, \al_1, \al_2\cdots \al_p}^{(\star)} \left(\prod_{i=1}^{p} (c_{x_i \al_i })^{\#_i}\right) \E\prod_{i=1}^pQ_{x_{i}} 
    \left(G _{ x_i \star} G _{  x_{i}\al_i}  \overline G _{\alpha_i \star }\right) ^{\#_i} \prec \OO_\tau\left(\Gamma^2\Phi^2 \right)^{p},
  $$
where
  $$
  A^{\#_i}:=\begin{cases}
A, \quad i\in 2\Z+1
\\
\overline A, \quad i\in 2\Z
  \end{cases} .
  $$ 
 Let $\cal G $ be the graph which represents 
$$\cal G =\prod_{i=1}^pQ_{x_{i}} 
    \left(G _{ x_i \star} G _{  x_{i}\al_i}  \overline G _{\alpha_i \star }\right) ^{\#_i} .$$
The following is an example of the graph with $p=3$. 
\be\label{G011-s}
\parbox[c]{0.42\linewidth}{\includegraphics[width=\linewidth]{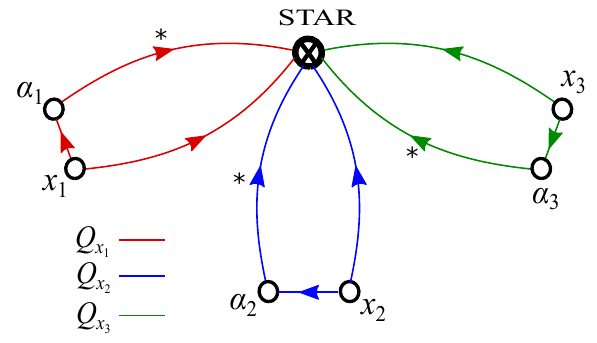}}
   \ee 
Let $\cal E_D$ be a dashed-line partition of atoms $x_i$ and $\al_i$, $1\le i\le p$.  Since in \eqref{gush20} we sum over $x\ne \al$, there is always a $\times$-dashed line between $x_i$ and $\al_i$ in $\cal E_D$.  Now as in \eqref{yikezz}, we only need to prove that for any fixed $\mathcal E_D$,
\be\label{yikezz11}
  \sum_{x_1, x_2,\cdots, x_p, \al_1, \al_2, \cdots, \al_p}^{(\star)} \left(\prod_i^{p} (c_{x_i \al_i })^{\#_i}\right)  \E 
\left( \cal E_D  \cdot \cal G\right)
 \prec\OO_\tau\left(\Gamma^2\Phi^2 \right)^{ p}.
  \ee
As in previous subsection, we expand the solid edges and weights with respect to ($\{x_i\}_{i=1}^p, \cal E_D)$. Note that during the expansions, 
we only add $x$ atoms into the independent set or make more visits to the $x$ atoms through off-diagonal solid edges. In particular, this process does not create weights at the $\alpha$ atoms. Moreover, for each atom $\al_i$, we always have $\deg (\al_i)=2$.

As in \eqref{Ejuzy}, after a constant number of steps of expansions 
we can write 
\be\label{Ejuzy2}
  \cal E_D  \cdot \cal G = \sum_{\gamma } \cal E_D  \cdot \cal G _{  \gamma }+ \sum_{\gamma' }\cal E_D  \cdot 
  \cal G _{ \gamma' }^{\text{\it error}},
\ee
where in each $\cal G _{\gamma}$, all edges and weights are fully expanded with respect to ($\{x_i\}_{i=1}^p, \cal E_D)$, and in each $\cal G _{  \gamma' }^{\text{\it error}}$, the total number of the off-diagonal edges is larger than some constant $M>0$. Since $M$ can be arbitrarily large, to prove \eqref{yikezz11}, we only need to show that for any fixed $\cal E_D$ and $\gamma$, 
\be\label{yikezz2EE}
    \sum^{(*)}_{x_1, x_2,\cdots, x_p, \al_1, \al_2\cdots \al_p} \left(\prod_{i=1}^{p} (c_{x_i \al_i })^{\#_i}\right)  \E
\left( \cal E_D  \cdot \cal G_{ \gamma}\right)
 \prec\OO_\tau\left(\Gamma^2\Phi^{2} \right)^{p}.
  \ee
  
For a graph $\cal E_D  \cdot \cal G_{ \gamma}$, we choose the molecules as
 $$\cal M_i=\{x_i, \al_i\},\quad 1\le i\le p.$$
Again we have $\cal Pol_1=\emptyset$.
Without loss of generality, we assume that 
\be\label{deftxxahEE}
\cal Pol_2(  \cal E_D  \cdot \cal G_{ \gamma} )=\{\cal M_{t+1},\cal M_{t+2}, \cdots,  \cal M_{p}\}
\ee
for some $0\le t\le p$. Now we repeat the argument from \eqref{deftxxah} to \eqref{phir0}, where the only difference is that \eqref{2p+t} is replaced by 
\be\label{3p+t}
\text{\# of off-diagonal edges in } \prod_{i=1}^{p} \cal G_{ \gamma,i} \ge 3p+t , 
\ee
because there are $3p$ off-diagonal edges in the original graph $\cal G$. Here $\cal G_{ \gamma,i} $ denotes the part of the graph coming from the expansions of $\left(G _{ x_i \star} G _{  x_{i}\al_i}  \overline G _{\alpha_i \star }\right) ^{\#_i}$  inside the $Q_{x_i}$.  Then we can define the $\Psi$-graphs $\Psi(\cal G_\gamma, \cal E_D, \bm \xi)$. To conclude \eqref{yikezz2EE}, it suffices to prove that for
\be\label{yikezz2EE3}
\sum^{(*)}_{x_1, x_2,\cdots, x_p, \al_1, \al_2\cdots \al_p} \left(\prod_{i=1}^{p} (c_{x_i \al_i })^{\#_i}\right)  \E\Psi(G_\gamma, \mathcal E_D, {\bm\xi})
\prec \OO_\tau\left(\Gamma^2\Phi^2\right)^{p}.
\ee
As in Lemma \ref{lem F1}, $\Psi(G_\gamma, \mathcal E_D, {\bm\xi})$ satisfies the ordered nested property, and, without loss of generality, we assume that  \eqref{ljsadfl} holds.
There are at least $3p+t$ off-diagonal solid edges by \eqref{3p+t}, and the above ordered nested property used $2(t+r)$ of them. Each of the other $3p-t-2r$ solid edges is bounded by $\OO_{\tau,\prec}(\Phi)$.  We know $c_{x_s \alpha_s}\ne 0$ only if $|x_s-\al_s|=\OO((\log N)^2 W)$. Moreover, we have combined two molecules only if each of them contains an atom such that these two atoms {\it always take the same value}.  Hence if we pick any atom $\wt x_s$ in ${\cal M}^\Psi_s$, then any other atom $\wt\alpha_s$ in ${\cal M}^\Psi_s$ satisfies $|\wt\al_s - \wt x_s|=\OO((\log N)^2 W)$. Then for $1\le s\le t+r$, with \eqref{chsz3} we have $\Psi_{\wt \alpha_s \beta_s}(\tau)  \prec \Psi_{\wt x_s y_s}(\wt \tau)$ and $  \Psi_{\wt \alpha_s \wt \beta_s}(\tau)  \prec \Psi_{\wt x_s \wt y_s}(\wt \tau) $ for $\wt \tau = \tau+(\log N)^{-1/2} $, 
where $y_s$ ($\wt y_s$) belongs to the same molecule as $\beta_s$ ($\wt \beta_s$) and $y_s, \wt y_s \in \{\star,\wt x_1, \cdots, \wt x_{t+r}\}$. Each of the other $3p-t-2r$ solid edges is bounded by $\OO_{\tau,\prec}(\Phi)$.
Thus we obtain that 
$$\Psi(\cal G_\gamma, \mathcal E_D, {\bm\xi})\prec \OO_{\wt\tau}\Big(\Phi^{3p-t-2r}\cdot \prod_{1\le s\le t+r} \Psi_{\wt x_{s} y_s}\Psi_{\wt x_{s} \wt y_s}\Big).$$
Plugging it into the left-hand side of \eqref{yikezz2EE}, we obtain that
\begin{align}
 & \sum^{(*)}_{x_1,\cdots, x_p, \al_1,\cdots,\al_p} \left(\prod_i^{p} (c_{x_i \al_i })^{\#_i}\right) \E\Psi(G_\gamma, \mathcal E_D, {\bm\xi}) \nonumber \prec \OO_{\wt\tau} \left(\E \sum_{\wt x_1,\cdots, \wt x_{t+r}}^{(\star)} \Phi^{3p-t-2r}\cdot \prod_{1\le s\le t+r} \Psi_{\wt x_{s} y_s}\Psi_{\wt x_{s} \wt y_s} \right) \nonumber\\
 &\prec \OO_{\wt\tau} \left( \Phi^{3p-t-2r} \E \sum_{\wt x_1}  \Psi_{\wt x_{1} y_1} \Psi_{\wt x_{2} \wt y_1}\sum_{\wt x_2}  \Psi_{\wt x_{2} y_2} \Psi_{\wt x_{2} \wt y_2}  \cdots   \sum_{\wt x_{t+r}}\Psi_{\wt x_{t+r} y_{t+r}} \Psi_{\wt x_{t+r} \wt y_{t+r}}    \right) \nonumber \\
 & \prec \OO_{\wt\tau} \left( \Phi^{3p-t-2r} \Gamma^{2(t+r)}\right) \prec \OO_{\tau} \left( \Phi^{3p-t-2r} \Gamma^{2(t+r)}\right), \label{phir}
\end{align}
 where in the second step we used \eqref{ljsadfl} such that one can sum over the $\wt x$'s according to the order $\wt x_{t+r}, \cdots, \wt x_1$, in the third step we used \eqref{chsz4} to get a $\Gamma^2$ factor for each sum, and in the last step we replaced $\wt \tau $ with $\tau$ by using $\wt\tau \le 2\tau$. Since $\Phi \ll 1$ and $\Gamma\ge 1$ by \eqref{GM1.5}, the factor $\Phi^{3p-t-2r} \Gamma^{2(t+r)}$ increases as $r$ increases. However, since we have only combined the molecules in $\cal Pol_2$, we must have $2r \le p-t$. Hence we can bound \eqref{phir} by 
 $$\eqref{phir} \prec \OO_{ \tau} \left( \Phi^{3p-t-(p-t)} \Gamma^{2t+(p-t)}\right) \prec \OO_{ \tau} \left( \Phi^{2p} \Gamma^{2p}\right), $$
where we used $t\le p$ in the second step. This proves \eqref{yikezz2EE3}, which concludes Lemma \ref{Q1}.

\section{Graphical tools - Part II} \label{sec_graph2}
In this section, we prove Lemma \ref{Qpart}. The proof is more involved than the ones for Lemma \ref{Q2} and Lemma \ref{Q1}, and we need to introduce some new types of components to our graphical tools. 

\subsection{Definition of Graph  - Part 2.} \label{def_II} \
\vspace{5pt}

\noindent{\bf Dotted edges}:
The dotted edge connecting atoms $\al$ and $\beta$ represents an $H_{\al\beta}$ factor. Since we only consider the real symmetric case, there is no need to label its direction and charge. (On the other hand, in the complex Hermitian case, we indicate either the direction or the charge of the dotted edge. This is one of the main differences from the real case.)
For example, we have
 \be\label{Dotlines}
\parbox[c]{0.18\linewidth}{\includegraphics[width=\linewidth]{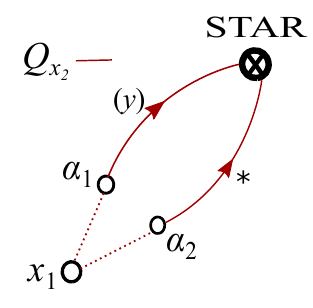}}\quad = Q_{x_2}\left(G^{(y)}_{\al_1\star}\overline{G_{\al_2\star}}H_{x_1\al_1}H_{x_1\al_2}\right).
 \ee
 
\vspace{5pt}

\noindent{\bf Weights and light weights}: 
We now introduce some new types (flavors) of weights in addition to $f_{1,2}$ introduced in Definition \ref{def_colorlessg}:
$$ f_3: \ \cal Y_{x}:=P_x(G^{-1}_{xx})=-z_x -\sum_\al s_{x\al}G^{(x)}_{\al\al} \ ,  \quad \quad f_4:    {\cal Y_{x }^{-1}}.$$
They are also drawn as solid $\Delta$ in graphs. It is important to observe that the $f_3$ or $f_4$ type of weights on atom $x$ are independent of the $x$-th row and column of $H$. By (\ref{GM1}) and \eqref{Zi}, it is easy to see that for the four types of weights we have
$$G_{xx}-M_x \prec \Phi, \quad G^{-1}_{xx}-M_x^{-1}\prec \Phi,\quad  \cal Y_x-M_x^{-1}\prec \Phi,\quad   {\cal Y^{-1}_x}-M_x \prec \Phi \ . $$
Correspondingly, we define the following four types of {\it light weights}: 
 $$f_1: G_{xx}-M_x,\quad f_2:G^{-1}_{xx}-M_x^{-1}\quad f_3: \cal Y_x-M_x^{-1}\quad f_4:{\cal Y^{-1}_x}-M_x.$$
They are drawn as hollow $\Delta$ in graphs. Furthermore, we define two more types of light weights: 
$$f_5: H_{xx},\quad f_6: W^{-d/2}.$$
Note that $H_{xx}\prec W^{-d/2}$ by (\ref{high_moment}).
For example, with the above definitions, we have
\be\label{wLW}
 \parbox[c]{0.18\linewidth}{\includegraphics[width=\linewidth]{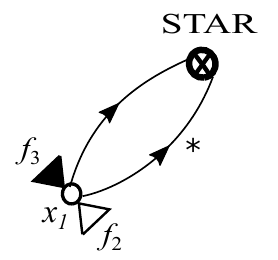}}\quad= G_{x_1\star}\overline{G_{x_1\star}} \cal Y_x (G^{-1}_{xx}-M_k^{-1}).
\ee
One can see that a regular weight provides a factor of order $\OO(1)$, while a light weight provides a factor of order $ \Phi $ like an off-diagonal edge. 

Now with the new graphical components introduced above, we give the graphical representations of more types of resolvent expansions.  
  
 \begin{lemma}\label{lemmai}
We have the following identities.
\begin{itemize}
  \item For $x\ne y$, we have $ G _{xy} = - G_{xx}\sum_{\al}H_{x\al}G^{(x)}_{\al y}$, i.e.,
 \be\label{Gp3}
\parbox[c]{0.40\linewidth}{\includegraphics[width=\linewidth]{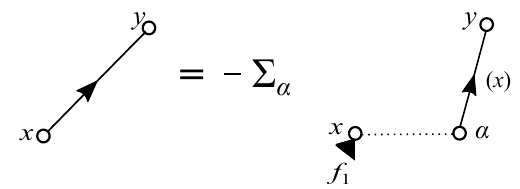}}
   \ee

 \item
 We have  
$$\left(G _{xx}\right)^{-1}=  \cal Y_{x} - \cal Z_x,\quad \cal Z_x:=-H_{xx}+Q_x\left(\sum_{\al,\beta}H_{x\al}H_{x\beta}G^{(x)}_{\al\beta}\right)
=-H_{xx}+\sum_{\al,\beta}H_{x\al}H_{x\beta}G^{(x)}_{\al\beta}-\sum_\al s_{x\al} G^{(x)}_{\al\al},$$
i.e.,
\be \label{Gp4}
\parbox[c]{0.82\linewidth}{\includegraphics[width=\linewidth]{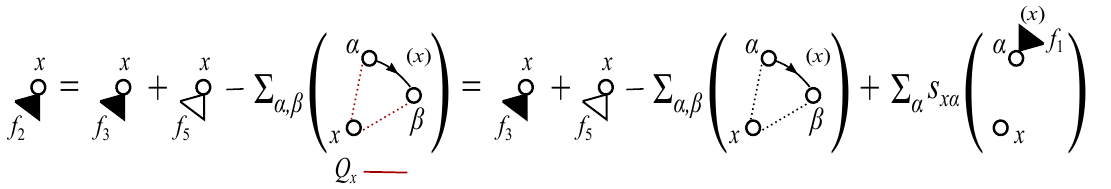}}
   \ee
Here we used the convention that we assign the value $1$ to an atom $x$ with no solid edge or weight attached to it.  

\item We have $G _{xx} =  \sum_{m=1}^\infty (\cal Y_{x})^{-m} (\cal Z_x)^{m-1},$ i.e.,  
\be\label{Gp5}
\parbox[c]{0.34\linewidth}{\includegraphics[width=\linewidth]{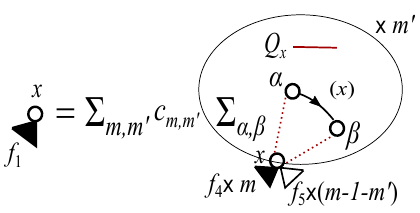}}
\quad c_{m,m'}:=(-1)^{m-1-m'}\begin{pmatrix}m-1\\m'\end{pmatrix},\ \ m\ge 1, \ \ 0\le m'\le m-1.
\ee
Here we omitted some details in the graph, i.e., for fixed $m\ge 1$ and $0\le m'\le m-1$, the graph should contain $m'$ copies of the part inside the big circle. The graph \eqref{Gp5} was indeed obtained by expanding the $(m-1)$-th power of $\cal Z_x$ 
using the binomial theorem. 
 
\item  For $x\ne y$, we have
$$
 \cal  Y_{x} =  \cal     Y^{(y)}_{x} - \sum_{\al}s_{x\al}G^{(x)}_{\al y}(G^{(x)}_{yy})^{-1}G^{(x)}_{y\al}
,\quad \quad \quad \quad 
\frac1{\cal  Y_{x}}  =\frac1{\cal     Y^{(y)}_{x} } + \frac1{\cal  Y_{x}}   \frac1{\cal     Y^{(y)}_{x} }\sum_{\al}s_{x\al}G^{(x)}_{\al y}(G^{(x)}_{yy})^{-1}G^{(x)}_{y\al}
$$
i.e.,  
\be\label{Gp6}
\parbox[c]{0.60\linewidth}{\includegraphics[width=\linewidth]{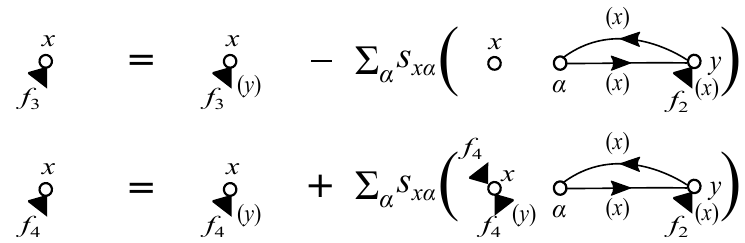}}
\ee

\end{itemize}
\end{lemma}

\begin{remark}
Note that the resolvent expansions in Lemma \ref{lem_exp1} are used to unravel the weak correlation between the edges (or the weights) and the atom, say $w$, that they are {\bf not} attached to. More precisely, each resolvent expansion in Lemma \ref{lem_exp1} expresses an edge (or a weight) into the sum of a term that is independent of the atom $w$ and an error term that is of higher order. 
On the other hand, the resolvent expansions in Lemma \ref{lemmai} will be used to unravel the dependence of the edges (or the weights) on the atom, say $x$, that they are attached to. In fact, the dependence is mainly through the dotted edges attached to $x$.
For an illustration of this principle, the reader can refer to e.g. the proof of Lemma \ref{lemEk} below.
\end{remark}

One can see that we create new atoms in the expansions in Lemma \ref{lemmai}, such as the atom $\al$ in \eqref{Gp3}. Some of the new atoms may be connected with the $\star$ atom with dashed edges, and hence create molecules in ${\cal P}ol_1$. Also it is important to put these new atoms into some molecules. If the expansions happen at the atom $x$, then the new molecules are all within an $\OO(W)$-neighborhood of $x$ and we put them into the molecule that contains $x$. In fact, in the proof we will always classify the new atoms in this way such that each molecule has diameter at most $\OO((\log N)^C W)$; see Definition \ref{def: SF} below. Moreover, under this classification, the IPC nested property still holds by the following lemma. 

 


\begin{lemma}\label{Lumm22}
Let $\cal G$ be a graph with $p$ molecules $\cal M_i$, $1\le i\le p$, that satisfy \eqref{yurenguodu}. Fix any (large) $D>0$. Let $x_0$ be an atom in molecule $\mathcal M_{x_{i_0}}$. Then for an edge or a weight attached to atom $x_0$, we can expand it using \eqref{Gp3}-\eqref{Gp6} and write $ \cal G$ as a linear combination of new graphs $\cal G_\gamma$:
 $$
 \cal G=\sum_\gamma \cal G_\gamma + \OO_\prec\left( N^{-D}\right) .
 $$
 We define the molecules in $\cal G_\gamma$ as
\be\label{new_atom}
 \cal M_{j}(\cal G_\gamma)= \begin{cases}{}
 \cal M_{j}(\cal G),  &  \text{ if } \ j\ne i_0
 \\\cal M_{j}(\cal G)\cup \{\text {new atoms}\},  & \text{ if } \ j=i_0 
 \end{cases},
\ee
that is, we include all the new atoms appearing from the expansions  \eqref{Gp3}-\eqref{Gp6} into the molecule containing $x_0$, and leave all other molecules unchanged. Then under this setting, each new graph $\cal G_\gamma $ has an IPC nested structure. 
\end{lemma}
\begin{proof}
First, notice that by (\ref{Zi}), we only need to keep a constant number of terms in \eqref{Gp5}. Then the lemma follows trivially from the definition of the IPC nested structure and the graphs in Lemma \ref{lemmai}. 
In fact, it is easy to see that under \eqref{new_atom}, the graphs in \eqref{Gp3}-\eqref{Gp6} only change the inner-molecule structures but do not affect the IPC nested property, which is an inter-molecule structural property by definition.
\end{proof}
 
\subsection{Proof of Lemma \ref{Qpart}: step 1.}\label{sec_step1} \
 Clearly, Lemma \ref{Qpart} is a stronger version of Lemma \ref{Q2}. In fact, the step 1 of the proof 
is just repeating the proof of Lemma \ref{Q2} until \eqref{2p+t}. 

\vspace{5pt}

 \noindent{\bf Step 1:} {\bf Expansion with respect to ($\{x_i\}_{i=1}^p , \cal E_D$)}. As in Section \ref{sec_simple}, the first step is to do the {\bf expansions with respect to ($\{x_i\}_{i=1}^p, \cal E_D)$}, and reduce the problem to proving that 
 \be\label{yikezz33}
  \E \sum_{x_1,\cdots, x_p}^{(\star)} \left( \prod_{i=1}^p b_{x_i} \right)
\left( \cal E_D  \cdot \cal G_{ \gamma}\right)
 \prec \OO_\tau\left(1+\Gamma^2\Phi^2\right)^{p},\quad  \cal G_{\gamma}=\prod_{i=1}^{p}Q_{x_i} \left(\cal G_{ \gamma,i} \right),
 \ee
where $\cal G_{ \gamma,i} $ denotes the part of the graph coming from the expansions of $G _{{x_i}\star} \overline G _{{x_i}\star}$. It is easy to see that the following conditions hold. 


\begin{enumerate}
 \item $\cal E_D $ is a dashed-line partition of the $\star$ atom and $x_i$ atoms, with $x_i-\times- \star$ for $1\le i\le p$.
 
\item In each $\cal G _{ \gamma , i}$, there are only solid edges and $f_2$ type of weights. All edges and weights are fully expanded with respect to ($\{x_i\}_{i=1}^p , \cal E_D$).

\item  We choose the molecules as $\cal M_i=\{x_i\}$. So far, we have $\cal Pol_1=\emptyset$. (In the later part of the proof, we will add more atoms to each molecule and $\cal Pol_1$ can be nonempty.) Moreover, we assume without loss of generality that the free molecules are $\cal M_{1},\cal M_{2}, \cdots,  \cal M_{t}$. 

\item For each $ 1\le i\le p$, the graph $\cal G _{\gamma, i}$ contains two separated paths connecting $\cal M_i$ to the $\star$ atom. 

\item 
For $1\le i\le t$, we have $\deg(x_i)\ge 4$ and $\deg(x_i)\in 2\Z$. This is due to the condition (\ref{gaszj}) and the fact that each expansion in \eqref{Gp1} and \eqref{Gp2} increases $\deg(x_i)$ by zero or two. 

\item 
If $\deg(x_i)=4$, the charges of the 4 solid edges must be ({\it 3 positive+1 negative}) or ({\it 1 positive+3 negative}). This follows immediately from the expansion process and the graphs in \eqref{Gp1}-\eqref{Gp2}. We did not emphasize this condition before, since it was not used in the previous proof. 
   
\item Without loss of generality, we can assume that for any $j\ne i$ such that $x_{j}=x_{i}$, the atom $x_{j}$ is {\bf not} visited in $ \cal G_{ \gamma , i}$. In fact, one can visit $x_i$ instead of $x_j$ in $\cal G_{\gamma, i}$.
\end{enumerate}


By \eqref{phir0}, we have that 
\be
\text{ LHS of \eqref{yikezz33}} \prec \OO_\tau \left(1+\Phi^{p}\Gamma^{p+t}\right).
\label{alkdjfyzjhs}
 \ee
 Suppose we can get $t$ more $\Phi$ factors. Then we have 
 \be
 \text{ LHS of \eqref{yikezz33}} \prec \OO_\tau \left(1+\Phi^{p+t}\Gamma^{p+t}\right) \prec \OO_\tau \left(1+\Phi^{2}\Gamma^{2}\right)^p\label{heur}
 \ee
 for any $0\le t\le p$. Thus the main goal of our proof is to show that each free molecule provides an extra factor $\Phi$, and hence $t$ more $\Phi$ factors in total.


We now explain briefly the basic strategy of our proof. Assume that
\be\label{num_pol}
|\cal Pol_1|=t_1,\quad |\cal Pol_2|=t_2, \quad |\{\text{free molecules}\}|=t,\quad t+t_1+t_2=p.
\ee
Then we have the following cases.
 \begin{itemize}
 \item For each molecule $\cal M_i$ in $\cal Pol_1$, we can obtain a factor $\OO(1)$ from $b_{x_i}$. 
 
\item Consider any molecule $\mathcal M_{x_{i_0}} \in \cal Pol_2$ that contains the atom $x_{i_0}$. By definition, there are $n\ge 1$ other molecules that are connected with $\mathcal M_i$ through dashed lines, say ${\mathcal M}_{i_1}, \cdots, {\mathcal M}_{i_n}$. Note that each $\mathcal G_{\gamma,i_k}$, $0\le k \le n$, is expanded from two off-diagonal edges $G_{x_{i_k}\star} \overline{G_{x_{i_k}\star}}$, and in the expansions $\mathcal G_{\gamma,i_k}$ always contains at least two separate paths of off-diagonal edges from the molecule $\cal M_{i_k}$ to $\star$. We call the two off-diagonal edges connected with $\mathcal M_{i_k}$ as ${\rm Edge}({\alpha_{i_k}, \beta_{i_k}})$ and ${\rm Edge}({\alpha'_{i_k},\beta_{i_k}})$ for $\alpha_{i_k},\alpha_{i_k}'\in \cal M_{i_k}$. Moreover, in the proof we will choose the molecules such that any two atoms in the same molecules have distance at most $\OO((\log N)^C W)$. In particular, all the atoms in $\mathcal M_{i_k}$, $0\le k \le n$, are within a $\OO((\log N)^C W)$-neighborhood of $x_{i_0}$. Then by \eqref{chsz2}, the above $(2n+2)$ off-diagonal edges are bounded by
\be\label{strat_phir}
 \sum_{x_{i_0}}  \prod_{0\le k \le n}|b_{x_{i_k}}| \Psi_{x_{i_0}\beta_{i_k}}\Psi_{x_{i_0}\beta'_{i_k}}. 
\ee
In the above sum, $2n$ of the $\Psi$ factors can be bounded by $\OO_{\tau,\prec}(\Phi^{2n})$. The rest of two $\Phi$ factors, say $\Psi_{x_{i_0}\beta_{i_0}}\Psi_{x_{i_0}\beta'_{i_0}}$, will be summed over $x_{i_0}$ and gives a factor $\sum_{x_{i_0}} \Psi_{x_{i_0}\beta_{i_0}}\Psi_{x_{i_0}\beta'_{i_0}} \prec \OO_\tau(\Gamma^2)$. 
Hence, each molecule $\cal M_{i_0}$ in $\cal Pol_2$ provides a factor of order 
$$ (\Gamma^2\Phi^{2n})^{\frac1{n+1}}\le (\Gamma^2\Phi^{2})^{1/2}\le \Gamma \Phi$$ 
for any $n\ge 1$. Here we emphasize that in order to be able to do the sum, we need the ordered nested property as in \eqref{ljsadfl}.
 

\item For each free molecule $\cal M_i\notin \cal Pol_1\cup\cal Pol_2 $, due to the ordered nested property, we obtain a factor $\Gamma^2$ when summing over $\sum_{x_i} $ as in \eqref{phir}. Furthermore, because of the condition (v) above, the existence of each free molecule increases the total number of off-diagonal edges in $\cal G_\gamma$ at least by 1. Every such edge provides a factor $ \Phi $. To conclude the proof, we still need  to extract one more $\Phi$ factor from each free molecule. If $\deg(x_i)>4$, then $\deg(x_i)\ge 6$ by condition (v), which increases the total number of off-diagonal edges at least by 2. This already gives the factor $\OO_{\prec}(\Phi^2)$ for each free molecule with degree larger than 4. 

 
\item Now we consider the free molecules $\cal M_i$ with $\deg(x_i)=4$. Recall the condition (vi) above. Without loss of generality, we assume  that there are {\it 3 positive} solid edges and {\it 1 negative} solid edge connected with the atom $x_i$ and they look like 
  $$  G_{x_i \beta_1} G_{x_i \beta_2} G_{x_i \beta_3} \overline{G}_{x_i \beta_4} ,$$
where $\beta_k \ne x_i$, $1\le k\le 4$. If we only consider this term, then it was proved in \cite{EKY_Average} that when summing over $\sum_{x_i} b_{x_i}$, 
\be\label{adsolyai}
  \sum_{x_i} b_{x_i}  G_{x_i \beta_1}G_{x_i \beta_2} G_{x_i \beta_3} \overline{G}_{x_i \beta_4}  \prec  \Phi^5, 
 \ee
i.e., each charged (non-neutral) atom provides an extra factor $\OO_\prec(\Phi)$. However, this is not an optimal bound for our goal, because two of the solid edges in $G_{x_i \beta_1} G_{x_i \beta_2} G_{x_i \beta_3} \overline{G}_{x_i \beta_4} $ have to be used in the ordered nested property. In other words, two of them should provide $\Gamma^2$ factors when taking the sums.  

\item Hence the main goal is to obtain an extra factor $\Phi$ as in \eqref{adsolyai} while keeping the IPC nested structure. 
\end{itemize}


\subsection{Proof of Lemma \ref{Qpart}: step 2.}\label{40+} \
So far, we have $ \cal G_{ \gamma}=\prod_{i=1}^{p}Q_{x_i} \left(\cal G_{ \gamma,i} \right)$. Our goal of this step is to write $\mathcal E_D\cdot \cal G_{ \gamma}$ as a linear combination of colorless graphs. 

\vspace{5pt}

\noindent{\bf Step 2: Removing the colors  $Q_{x_i}$.} We first study the single piece $Q_{x_i}\cal G_{ \gamma,i}$. Due to the condition (ii) below \eqref{yikezz33}, we know that in $\cal G_{ \gamma ,i}$ the edges and weights that are not attached to atom $x_i$ must be independent of the atom $x_i$ under $\mathcal E_D$. Thus we can rewrite 
\be\label{aljjayz}
\cal G_{ \gamma,i}= \cal G^{in}_{ \gamma ,i}\cal G^{out}_{ \gamma ,i},\quad \text{with} \quad P_{x_i}\cal G^{out}_{ \gamma ,i}=\cal G^{out}_{ \gamma ,i},
\ee
where $\cal G^{in}_{ \gamma ,i}$ consists of the solid edges and weights that connect to atom $x_i$, and $\cal G^{out}_{ \gamma ,i}$ consists of the rest of $\cal G_{ \gamma,i}$. Moreover there are only $f_2$ weights in $\cal G^{in}_{ \gamma ,i}$. If in $\cal G^{in}_{ \gamma ,i}$ the number of $f_2$ weights is $n$ and $\deg(x_i)=2s$, then locally it should look like the graph in (\ref{EKgr}) (with $x_i$ replaced by $x$). 
 Since $Q_{x_i} \left(\cal G_{ \gamma,i} \right)= \cal G^{out}_{ \gamma ,i} (\cal G^{in}_{ \gamma ,i}-P_{x_i} \cal G^{in}_{ \gamma ,i}),$ it suffices to write $P_{x_i}(\cal G^{in}_{ \gamma,i})$ as a linear combination of some colorless (local) graphs. This is the content of the following lemma.
   


   \begin{lemma}\label{lemEk}
   Let $\cal G_0$ be a colorless graph as following:
  \be\label{EKgr}
  \parbox[c]{0.25\linewidth}{\includegraphics[width=\linewidth]{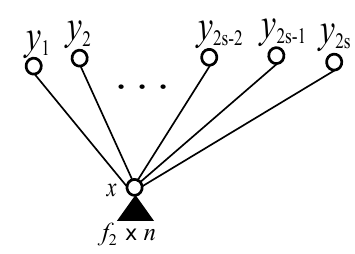}}
   \ee 
where we did not draw the labels for the edges and weights, and moreover, the labels for these $n$ weights can be different. We also assume that 
   $$
   x-\times-y_j,\quad 1\le j\le 2s,
   $$
i.e., there is a $\times$-dashed line between atom $x$ and each atom $y_j$ (although for simplicity we did not draw them in the above graph). Then performing the expansions in Lemma \ref{lemmai} with respect to atom $x$ and applying $P_{x}$, we get that for any fixed $D>0$, 
\be\label{Gzah}
   \E _x\cal G_0=\sum_\kappa \cal F_\kappa+\OO_\prec(N^{-D}) , \quad  \text{with} \quad \cal F_\kappa=\sum_{\vec \al} C^{\kappa}(\vec \al )\cdot \cal G_\kappa (\vec \al),
  \ee
where the total number of $\cal F_\kappa$ is of order $\OO(1)$, $\vec \al=(\al_1, \al_2\cdots, \al_{s' })$, $s'\equiv s'(\kappa)$, is the vector of newly added atoms, $C^{\kappa}(\vec \al )$ are complex-valued deterministic coefficients, and $\cal G_\kappa (\vec \al)$ are graphs which satisfy the following conditions.  

\begin{enumerate}
\item Each graph $\cal G_\kappa (\vec \al)$ looks like 
 \be\label{EKgr2r}
\parbox[c]{0.35\linewidth}{\includegraphics[width=\linewidth]{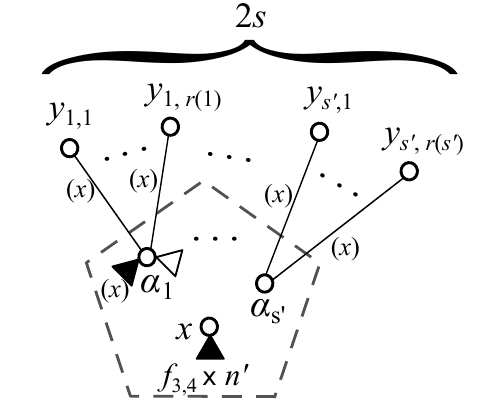}}
   \ee 
where we have some new atoms $\al_k$, $1\le k\le s'(\kappa)$. We emphasize again that the pentagon does not really appear in graph, and it is only used to help us to understand the structures.

\item There are no dotted lines in the graph. There are no solid edges connected with the atom $x$. There may be solid edges between $\al_k$ atoms. 
 
\item 
For each $1\le j\le 2s$, the solid edge $Edge(x, y_j)$ in \eqref{EKgr} was replaced with $Edge(\alpha_{t_j}, y_j)$ in \eqref{EKgr2r} for some $1\le t_j\le s'$. (Here $j\ne j'$ does not necessarily imply $t_j\ne t_{j'}$.) Furthermore, the $Edge(\alpha_{t_j}, y_j)$ keeps all the labels (direction and charge) of $Edge(x, y_j)$ except that the atom $x$ is added into the independent set of $Edge(\alpha_{t_j}, y_j)$.

 \item Except the edges $Edge(\alpha_{t_j}, y_j)$, $1\le j\le 2s$, there are no other solid edges and weights attached to $y_j$. 
 
 \item  In \eqref{EKgr2r}, 
 we have the following $\times$-dashed lines 
 $$
(1):  x-\times-y_j,\quad 1\le j\le 2s; \quad (2):  x-\times-\al_k,\quad 1\le k\le  s'; \quad  (3): \al_k-\times -\al_{k'},\quad 1\le k\ne k'\le s'.$$

\item atom $x$ only has $f_3$ and $f_4$ types of weights attached to it, while each atom $\alpha_k$ only has $f_1$ type of weights and $f_6$ type of light weights (i.e. $W^{-d/2}$ factors) attached to it. Moreover, the atom $x$ is in the independent set of each $f_1$ weight on $\al$ atoms. 

\item For each atom $\al_k$,  
 \begin{align*}
 &\deg(\al_k )\ne 2 \implies \text{there exists at least one $f_6$ type of light weight attached to it}. 
 \end{align*}
 \item We have
\be\label{swlkj}
 \left|C^{\kappa}(\vec \al )\right|=O  \left( W^{-ds'(\kappa)}  \right) {\bf 1}\left(\max_k |\al_k- x|=\OO(W)\right) .
 \ee
 \end{enumerate} 
\end{lemma}
 \begin{proof}
We start by expanding the edges and weights in $\cal G_0$ using \eqref{Gp3}, \eqref{Gp5} and the first identity in \eqref{Gp4}. With these expansions, we can write \eqref{EKgr} as a sum of graphs of the following form
\be\label{Gpin22}
   \wt {\cal G}_{\wt\kappa}:= \quad
   \parbox[c]{0.35\linewidth}{\includegraphics[width=\linewidth]{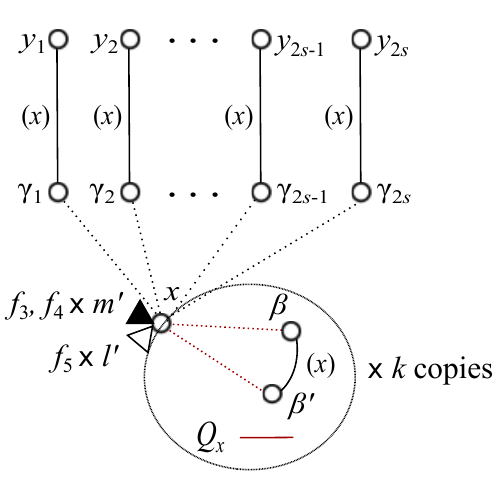}}
   \ee
 More precisely, in the above graph we have 
\begin{itemize}
\item $m'$  $f_3$ type and/or $f_4$ type of weights, which are independent of the atom $x$; 

\item $l'$ $f_5$ type of light weights (i.e., $H_{xx}$), which are independent of the rest of the graph;

\item $k $ copies of $Q_{ x}\left( H_{x\beta}H_{x\beta'}G^{(x)}_{\beta\beta'}\right)$, i.e., 
\begin{align*}
\prod_{j=1}^k Q_{ x} \left( H_{x\beta_{j}}H_{x\beta'_{j}}G^{( x)}_{\beta_{j}\beta'_{j}}\right)
= \prod_{j=1}^k \left( H_{x\beta_j}H_{x\beta'_j}G^{(x)}_{\beta_j\beta'_j}-{\delta_{\beta_j \beta_j'}}s_{x\beta_j}G^{(x)}_{\beta_j\beta _j}\right);
\end{align*}

\item $2s$ dotted lines which connect the atom $x$ with {new atoms} $\gamma_1$, $\gamma_2,$ $\cdots,$ $ \gamma_{2s}$;
 
\item $2s$ solid edges (which come from the $2s$ solid edges in \eqref{EKgr}) connecting atoms $\gamma_j$ with atoms $y_j$, and these edges are now independent of the atom $x$.
\end{itemize} 
With \eqref{Gpin22},  we can now write \eqref{EKgr} as
$$
\sum_{\tilde \kappa} \sum^{(x)}_{\gamma_1, \gamma_2,\cdots, \gamma_{2s}}\sum^{(x)}_{\beta_1, \beta_2,\cdots, \beta_{n}}\sum^{(x)}_{\beta'_1, \beta'_2,\cdots, \beta'_{n}} c_{\tilde \kappa} \wt {\cal G}_{\,\tilde \kappa},
$$
where $c_{\tilde \kappa}$ denotes deterministic coefficients which depend only on $m'$, $l'$ and $k$ (recall \eqref{Gp5}). It is easy to see that if either $k$ or $l'$ in $\wt {\cal G}_{\,\tilde \kappa}$ is
very large, then $\sum^{(x)}_{\bm \gamma }\sum^{(x)}_{\bm\beta}\sum^{(x)}_{\bm\beta'} c_{\tilde \kappa}\wt {\cal G}_{\,\tilde \kappa}$ will be small enough to be treated as error terms due to \eqref{Zi}. Thus we can focus on the graphs $\wt {\cal G}_{\,\tilde \kappa}$ whose $k$ and $l'$ are bounded by some large constant. Moreover, from \eqref{Gp3}-\eqref{Gp5}, it is easy to see that $m'$ is bounded by $n+k+l'+1$.

Now we can calculate $ \E_x  \wt {\cal G}_{\, \tilde \kappa}$, which is quite straightforward due to the following observations. (They are already contained in the previous discussions, but we repeat them here to make the proof clearer.)
\begin{itemize}
\item The solid edges and $f_3$, $f_4$ types of weights are independent of the atom $x$. 

\item The $f_5$ light weights are independent of all the other parts, and $\E _{x}H_{xx}^{n'} =\OO(W^{-n'd/2})$.

\item  The dotted edges can be written as
\be\label{sjkyrtr}
\mathbb E_x \left[\left(\prod_{j=1}^{2s}H_{ x\gamma_j}\right)\prod_{j=1}^n\left(H_{x\beta_j}H_{x\beta'_j}-s_{x\beta_j}\delta_{\beta_j\beta'_j}\right)\right].
\ee
\end{itemize}
Recall that $H_{xa}$ is independent of $H_{xb}$ if $a\ne b$. Thus in order for \eqref{sjkyrtr} to be nonzero, we need to pair the $\gamma$ and $\beta$ atoms. It can be accomplished using dashed lines as follows. We write 
$$\wt {\cal G}_{\, \tilde \kappa}=\sum_{\cal E} \cal E \cdot \wt {\cal G}_{\, \tilde \kappa},$$
where $\cal E$ denotes the dashed-line partitions of $\left(\{\gamma_j\}_{j=1}^{2s}, \{\beta_j\}_{j=1}^n, \{\beta'_j\}_{j=1}^n\right)$. In order to have $\E_x \cal E \cdot\wt {\cal G}_{\, \tilde \kappa}\ne 0$, we must have that
\be\label{kadfl}
{\text{each $\gamma$, $\beta$ or $\beta'$ atom is connected with another $\gamma$, $\beta$ or $\beta'$ atom through a dashed line}},
\ee
and for any fixed $j$,
\be\label{kadfl2} 
{\text{the  $\beta_j$ ($\beta'_j$) atom must be connected with an atom that is not $\beta'_j$ ($\beta_j$) through a dashed line}.}
\ee
Now for any graph $\cal E \cdot \wt {\cal G}_{\, \tilde \kappa}$ satisfying the above two conditions, we merge the  $\gamma$, $\beta$ and $\beta'$ atoms that are connected through dashed lines, and call the new graph $\wt {\cal G}^{\cal E}_{\tilde \kappa}$ (which also includes the dashed lines).
Then we rename the merged $\gamma$, $\beta$ or $\beta'$ atoms in $\wt {\cal G}^{\cal E}_{\tilde \kappa}$ as $\al_1, \; \al_2,\; \cdots, \al_{s'}$, which are all different from each other (i.e. $\al_i -\times- \al_j$ for $1\le i \ne j \le s'$). Note that the solid edges between $\beta$ and $\beta'$ atoms can either become solid edges between the $\al$ atoms or become $f_1$ weights on the $\al$ atoms. Therefore $\E_x  \wt {\cal G}^{\cal E}_{\, \tilde \kappa}$ can be written as a sum of graphs of the following form: 
\be\label{Gpin23}
\cal G_{\, \tilde \kappa, \; \kappa}:= \quad \parbox[c]{0.30\linewidth}{\includegraphics[width=\linewidth]{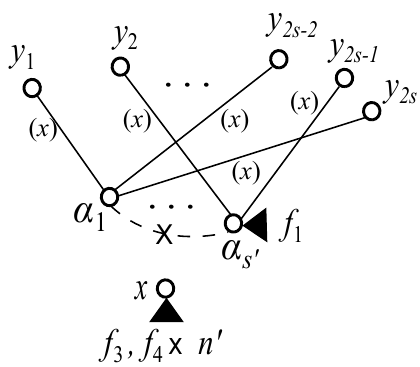}}
\ee
In sum, we have shown that for $\wt {\cal G}_{\, \tilde \kappa}$ in \eqref{Gpin22},
\be\label{rubicz-1}
 \E_x  \wt {\cal G}_{\, \tilde \kappa}=\sum_{\cal E}\E_x \wt {\cal G}^{\cal E}_{\tilde \kappa} =\sum_{\kappa} \sum_{\vec \al }C_{\, \tilde \kappa, \; \kappa}(\vec \al)\cal G_{\, \tilde \kappa, \; \kappa},
\ee
where $C_{\, \tilde \kappa, \; \kappa}(\vec \al)$ comes from 
\eqref{sjkyrtr} and satisfies 
\be\label{rubicz0}
C_{\, \tilde \kappa, \; \kappa}(\vec \al)=O\left(\prod_{j=1}^{s'(\kappa)}\left(W^{-d/2}\right)^{\#\text{dot}(\al_k)}\right){\bf 1}\left(\max_k |\al_k- x|=\OO(W)\right) .
\ee
Here $\#\text{dot}(\al_k)$ denotes the total number of dotted lines connected with the $\al_k$ atom in $\wt{\cal G}^{\cal E}_{\, \tilde \kappa}$. 
The $\E_xH_{xx}^{n'}$ term may make the coefficients even smaller, but we will not consider it in the following proof. 

So far, we have obtained the form in \eqref{Gzah}. It is easy to see that the conditions (i)-(vi) below \eqref{Gzah} hold. It remains to verify the conditions (vii) and (viii). Clearly by \eqref{kadfl}, we have that in $\wt{\cal G}^{\cal E}_{\, \tilde \kappa}$,
\be\label{rubicz}
\#\text{dot}(\al_k)\ge 2, \quad 1\le k \le s'.
\ee
In $\wt {\cal G}_{\, \tilde \kappa}$, each $\gamma$, $\beta$ and $\beta'$ atom is connected with 1 dotted line and 1 solid line. Therefore, we must have
\be\label{sdf;ju}
\#\text{dot}(\al_k)\ge \deg(\al_k ), \quad \text{ in $\wt{\cal G}^{\cal E}_{\, \tilde \kappa}$}.
\ee
On the other hand, we know that  $\deg(\al_k ) $ can be strictly smaller than $\#\text{dot}(\al_k)$. This happens only when the ending atoms of a solid edge are both equal to $\al_k$ and this solid edge then becomes an $f_1$ weight on atom $\al_k$. Moreover, we know that this solid edge can only be $Edge(\beta_k, \beta'_{k})$ in $\wt {\cal G}_{\, \tilde \kappa}$ for some $1\le k\le n$. Note that if there is a dashed line between $\beta_k$ and $\beta_k'$, then by \eqref{kadfl2} $\beta_k$ must be connected with another non-$\beta_k'$ atom through a dashed line. Due to this observation, we must have that
$$\deg(\al_k )<2\implies \#\text{dot}(\al_k)>2.$$
Together with \eqref{sdf;ju}, we get
$$\deg(\al_k )\ne 2\implies  \#\text{dot}(\al_k) \ge 3.$$
Combining with \eqref{rubicz0} and \eqref{rubicz}, one can see that the conditions (vii) and (viii) hold. This completes the proof of Lemma \ref{lemEk}. 
\end{proof} 
 
Now we return to Step 2. We apply Lemma \ref{lemEk} to $Q_{x_i} \cal G^{in}_{ \gamma ,i}=\left( 1-P_{x_i}\right) \cal G^{in}_{ \gamma ,i}$ for $1\le i\le p$, where the atom $x_i$ plays the role of atom $x$ in Lemma \ref{lemEk}. Then we can write  
\be\label{jhyaz}
Q_{x_i} \left(\cal G_{ \gamma,i} \right) =\cal G^{out}_{ \gamma ,i} \cdot \sum_{\kappa} \sum_{\vec \al_i}C^{\gamma,\kappa}_{\vec \al_i, x_i} \cal G^{in, \kappa}_{ \gamma ,i} (\vec \al_i, x_i),\quad \vec \al_i=(\al^1_i, \al^2_i,\cdots),
\ee
where for each fixed $\kappa$, $\cal G^{in, \kappa}_{ \gamma ,i}$ and $C^{\gamma,\kappa}_{\vec \al, x_i}$ satisfy the  conditions in Lemma \ref{lemEk}. 
Now the right-hand side of \eqref{jhyaz} is a linear combination of colorless graphs. Then taking product, we obtain that 
\be\label{prod_Gk}
\cal E _D\prod_{i=1}^{p}Q_{x_i} \left(\cal G_{ \gamma,i} \right) = \cal E _D
\sum_{ \kappa_1,\ldots, \kappa_p} \sum_{\vec \al_1, \cdots, \vec \al_p}\prod_{i=1}^p \left(C^{\gamma,\kappa_i}_{\vec \al_i, x_i}\right)   \prod_{i=1}^{p}\left(\cal G^{out}_{ \gamma ,i}\cdot \cal G^{in, \kappa_i}_{ \gamma ,i}\right).
\ee
Now we simplify the notations as
$$
{\bm \kappa}:=\{\kappa_1, \kappa_2\cdots, \kappa_p\}, \quad {\bm \al}:=\{\vec \al_1, \vec \al_2, \cdots, \vec\al_p\},\quad \vec x:=\{x_1, x_2, \cdots, x_p\},\quad \cal G_{\gamma, {\bm \kappa}}:=\prod_{i=1}^{p}\left(\cal G^{out}_{ \gamma ,i}\cdot \cal G^{in, \kappa_i}_{ \gamma ,i}\right).
$$
Then we can write (\ref{prod_Gk}) as
$$
 \cal E _D\prod_{i=1}^{p}Q_{x_i} \left(\cal G_{ \gamma,i} \right) = \cal E _D \sum_{{\bm \kappa}} \sum_{{\bm \al} } C^{\gamma, {\bm \kappa}}_{{\bm \al}, \vec x} \cdot   \cal G_{\gamma, {\bm \kappa}}({\bm \al}, \vec x),
$$
where 
$$
C^{\gamma, {\bm \kappa}}_{{\bm \al}, \vec x}=O\left(   \left(W^{-d}\right)^{\#\text { of  $\al$ atoms }}\right){\bf 1}\left(\max_{i,j} |\al_i^j-x_i|=\OO(W)\right) .
$$
Now let $\cal E^{\gamma, {\bm \kappa}}$ be a dashed-line partition of the atoms in $\cal G_{\gamma, {\bm \kappa}}$ such that
\begin{enumerate}
\item the restriction of $\cal E ^{\gamma, {\bm \kappa}}$ to the dashed-line partition of $\{x_i\}$ is equal to $\cal E_D$;
\item in $\cal E^{\gamma, {\bm \kappa}}$, we have $x_i-\times- \al_i^j$ and $\al_i^{j}-\times- \al_i^{j'}$, which are consistent with the dashed and $\times$-dashed lines in $\cal G^{in, \kappa}_{ \gamma ,i} (\vec \al_i, x_i)$.
\end{enumerate}
In this case, we shall also say that $\mathcal E^{\gamma,\bm\kappa}$ is a {\it{dashed-line extension}} of $\mathcal E_D$ that is consistent with $\prod_{i=1}^p \cal G^{in, \kappa}_{ \gamma ,i} (\vec \al_i, x_i)$. Then to prove \eqref{yikezz33}, we only need to show that for any fixed $\gamma$, ${\bm \kappa}$ and $\cal E^{\gamma, {\bm \kappa}}$, 
\be\label{yikezz44}
\E \cal F_0 \prec \OO_\tau\left(\Gamma^2\Phi^2 +1\right)^{ p},
\ee
where
\be\label{yikezz44F}
 \cal F_0:= \sum_{\vec x}^{(\star)}   \sum_{{\bm \al}} \wt  C^{\gamma, {\bm \kappa}}_{{\bm \al}, \vec x}  \cdot \cal E^{\gamma, {\bm \kappa}}\cdot   \cal G_{\gamma, {\bm \kappa}}({\bm \al}, \vec x),
 \ee
with
\begin{align*}
\wt  C^{\gamma, {\bm \kappa}}_{{\bm \al}, \vec x} \; = \; &  C^{\gamma, {\bm \kappa}}_{{\bm \al}, \vec x} \cdot  \prod_{i=1}^pb_{x_i}
 \; = \;   O\left(   \left(W^{-d}\right)^{\#\text { of  $\al$ atoms }}\right){\bf 1}\left(\max_{i,j} |\al_i^j-x_i|=\OO(W)\right)  .
\end{align*}

We shall call a linear combination of graphs a {\bf forest}. 
The above $\cal F_0$ is a forest, but with some special structures. We now pick some important structures that are useful for our proof, and call the forest with the desired structures a {\bf standard forest}. Then the proof of Lemma \ref{Qpart} is reduced to showing that for a standard forest, its expectation is always bounded by the right-hand side of \eqref{yikezz44}. 

\begin{definition}[Simple free molecule] 
In a colorless graph $ \cal G$, $\cal M_i$ is called a simple free molecule if it is a free molecule and satisfies
 \begin{enumerate}
 \item $\deg(\cal M_i)=4$ (cf. \eqref{degree});
 \item there is NO dashed edge inside $\cal M_i$; 
 \item there is NO off-diagonal solid edge inside $\cal M_i$;
 \item there is NO light weight inside $\cal M_i$. 
  \end{enumerate}
Here we say that a dashed/solid edge is inside $\cal M_i$ if the ending atoms of this edge are both in $\cal M_i$.
\end{definition}
\begin{definition}[Standard Forest]\label{def: SF} We call $\cal F$ a standard forest if $\cal F$ can be written as
\be\label{fagao}
\cal F=  \sum_{\vec x}^{(\star)}   \sum_{{\bm \al}}
 C _{{\bm \al}, \vec x}  \cdot \cal E\cdot  \cal G ({\bm \al}, \vec x),
\ee
where $\mathcal E$ is a dashed-line partition of all the atoms (including $\star$), and the coefficient $C _{{\bm \al}, \vec x}$ and the graph $ \cal G ({\bm \al}, \vec x)$ satisfy the following properties.
\begin{enumerate}
\item $\cal G ({\bm \al}, \vec x)$ has a $\star$ atom and $p$ molecules: $\cal M_1$, $\cal M_2, \cdots , \cal M_p$, where for each $1\le i\le p$, 
$$
\cal M_i=\{x_i, \; \al_i^1,\;\al_i^2,\; \al_i^3\cdots \}.
$$
Moreover, there are $\times$-dashed edges between all the atoms within one single molecule $\cal M_i$.

\item $ \cal G ({\bm \al}, \vec x)$ has no colors, dotted edges or $f_5$ type of light weights. 
\item $ \cal G ({\bm \al}, \vec x)$ satisfies the IPC nested property. 
\item If $\cal M_i$ is a free molecule, i.e., $\cal M_i\notin \cal Pol_1\cup \cal Pol_2$, then we have (recall \eqref{degree})
\be\label{deg_mi}
\deg (\cal M_i)\in 2\N, \quad \deg (\cal M_i) \ge 4.
\ee

 \item We have
 \be\label{Calphak}
 C _{{\bm \al}, \vec x} =
  \OO\left(   \left(W^{-d}\right)^{\#\text { of  $\al$ atoms }}\right){\bf 1}\left(\max_{i,j} |\al_i^j-x_i| \le (\log N)^{\OO(1)} W \right) .
\ee

\item\label{tianqi} If  $\cal M_i$ is a simple free molecule (which is defined right below), then the charges of the solid edges connected with $\cal M_i$ must be 
\be\label{3113}
  \text{3 positive + 1 negative, or 1 positive + 3 negative,}
\ee
and we must have 
\be\label{3114}
\deg(\al_i^j)\in \{0,\; 2\}, \quad \deg (x_i)\in  \{0,\; 2, \; 4\}.
\ee 
  \end{enumerate}
\end{definition}

By Lemma \ref{lemEk}, one can see that we obtain a standard forest by removing the colors. 

\begin{lemma}\label{panglei}
The forest $\cal F_0$ in \eqref{yikezz44F}, which is derived from \eqref{jhyaz}, is a standard forest in the sense of Definition \ref {def: SF}. 
\end{lemma}
\begin{proof} One can easily check that the conditions (i)-(v) in Definition \ref{def: SF} hold for $\mathcal F_0$ using Lemma \ref{lemEk}. 
From Lemma \ref{lemEk} and the condition (vi) below \eqref{yikezz33}, we know that \eqref{3113}  holds for each molecule in $\cal F_0$ whose degree is equal to 4. Hence the simple free molecules in $\cal F_0$  satisfy \eqref{3113}.  For \eqref{3114}, if $\cal M_i$ is a simple free molecule with no $\al$ atoms, then we have $\deg (x_i)=4$ by definition. On the other hand, if there are some $\al$ atoms in this $\cal M_i$, then by the conditions (ii) and (vii) in Lemma \ref{lemEk}, we must have $\deg (\al_i^j)=2$ and $\deg(x_i)\in \{0,2\}$. Here the $\deg(x_i)=0$ case comes from the $P_{x_i} \cal G^{in}_{ \gamma ,i}$ part as in Lemma \ref{lemEk}, and the $\deg(x_i)=2$ case comes from the product of a term from $P_{x_i} \cal G^{in}_{ \gamma ,i}$ with a term from other $\cal G_{\gamma,j}$ with $j\ne i$. In sum, the simple free molecules in $\cal F_0$  satisfy \eqref{3114}. Therefore the forest $\cal F_0$ in \eqref{yikezz44F}, which is derived from \eqref{jhyaz}, is a standard forest. 
\end{proof}


%

We have the following high probability bound on the standard forests, where the simple free molecules play an important role. 

\begin{lemma}\label{kaofu22}
Suppose \eqref{GM1.5} and \eqref{GM1} hold. Let $\cal F$ be a standard forest of the form \eqref{fagao}. Assume \eqref{num_pol} holds and there are $t_s$ ($0\le t_s \le t$) simple free molecules. Then
  \be\label{asdl;po;235}
 |\cal F| \prec \OO_\tau \left( \Phi^{t_2 + 2t  - t_s }\Gamma^{t_2 +2t}\right) \ .
\ee 
\end{lemma}

Note that compared with \eqref{yikezz44}, there is no $\E$ acting on $\cal F$, and the above bound holds even without the condition \eqref{3113}. On the other hand, it has $t_s$ fewer $\Phi$ factors than the right-hand side of \eqref{yikezz44}. In order to get these factors, the condition \eqref{3113} becomes essential. 

\begin{proof} [Proof of Lemma \ref{kaofu22}]

We first count the number of off-diagonal solid edges between molecules. 
By the definition of polymers, one can see that in the following cases the solid edges must connect non-equivalent atoms under $\cal E$: 
 \begin{enumerate}
\item one ending atom is in free molecule;
\item one ending atom is in $\cal Pol_2$ and the other one is not. 
\end{enumerate}
In the following graph, we draw the solid edges that belong to these two cases. 
 \be\label{asdfrgha}
  \parbox[c]{0.65\linewidth}{\includegraphics[width=\linewidth]{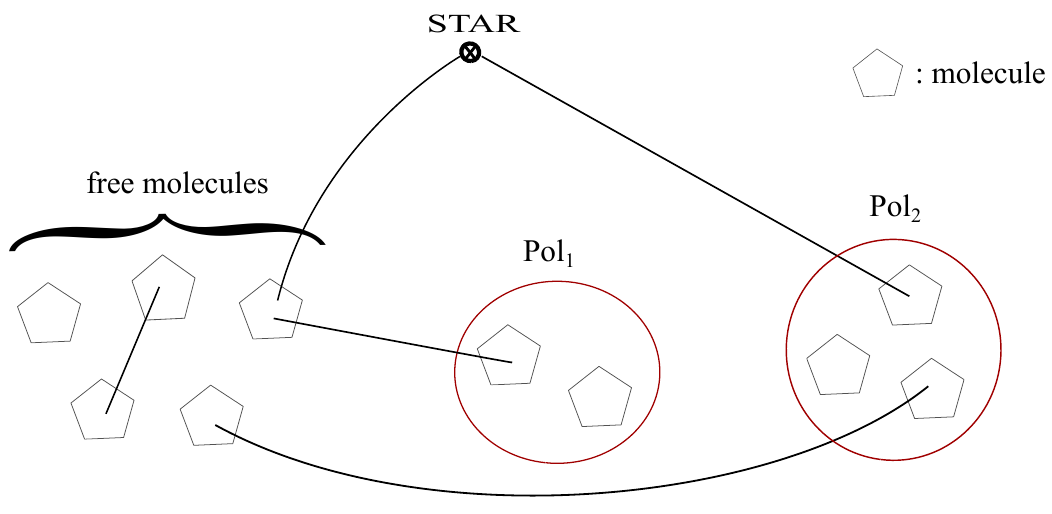}}
\ee
Then it is easy to calculate that the total number of solid edges of the above types (i)-(ii) is (cf. \eqref{degree})
\begin{equation}\label{above_types}
\frac12\left({\rm deg} (\cal Pol_2)+\sum_{s \text{ free}}{\rm deg}(\cal M_s)+{\rm deg}\Big(\{\star\}\cup \cal Pol_1\Big)\right).
\end{equation}
Because of the IPC nested structure, we know ${\rm deg} (\cal Pol_2)\ge 2t_2$, since there are $2t_2$ separate paths starting from molecules in $\cal Pol_2$. Similarly, we have ${\rm deg}\left(\{\star\}\cup \cal Pol_1\right)\ge 2(t+t_2)$, since there are $2(t+t_2)$ separate paths from free molecules and $\cal Pol_2$ to the $\star$ atom and $\cal Pol_1$. Thus we have that 
\be
 \eqref{above_types}  \ge 2t_2+3t+ \frac{1}{2}\sum_{s \text{ free}}\left({\rm deg}(\cal M_s)-4\right).\label{above_types2}
\ee
Furthermore by (\ref{Calphak}) and \eqref{chsz2}, we can bound each solid edge in \eqref{asdfrgha} by $\OO_\prec(\Psi_{x_i x_j})$ if it connects $\cal M_i$ and $\cal M_j$, or $\OO_\prec(\Psi_{x_i\star})$ if it connects $\cal M_i$ and the $\star$ atom.

Now for the graph $\cal G({\bm\al},\vec x)$ in \eqref{fagao} we can define its $\Psi$-graphs $\Psi(\cal G({\bm\al},\vec x), \cal E, {\bm\xi})$ as in Definition \ref{Psi-graph}, except that we now {\it keep all the {\bf  light weights} in $\cal G({\bm\al},\vec x)$ such that each of them represents a $\Phi$ factor in the $\Psi$-graphs.} Without loss of generality, we assume that the free molecules in $\cal E\cdot  \cal G ({\bm \al}, \vec x)$ are $\cal M_1,\cdots, \cal M_t$. Then these molecules are also free $\Psi$-molecules in the $\Psi$-graphs and we denote them by ${\cal M}^{\Psi}_{s}=\cal M_s$, $1\le s \le t$. We denote the new free $\Psi$-molecules by ${\cal M}^{\Psi}_{t+1}, \cdots, {\cal M}^{\Psi}_{t+r_2}$, and the non-free $\Psi$-molecules by ${\cal M}^{\Psi}_{t+r_2 +1}, \cdots, {\cal M}^{\Psi}_{t+r_2+r_1}$, where $0\le 2r_2\le t_2 $ and $0\le r_1 \le t_1$. 

As in Lemma \ref{lem F1}, we have 
\be \nonumber 
\cal E\cdot  \cal G ({\bm \al}, \vec x) \prec \OO_\tau\Big(\sum_{\bm \xi} \Psi(\cal G({\bm\al},\vec x), \cal E, {\bm\xi})\Big) ,
\ee
where each $\Psi(\cal G, \cal E, {\bm\xi})$ satisfies the IPC nested property with the $\Psi$-molecules ${\cal M}^{\Psi}_1, \cdots , {\cal M}^{\Psi}_{t+r_2+r_1}$. Now to conclude \eqref{asdl;po;235}, it suffices to show that for any ${\bm \xi}$,
\be \label{bound_sf}
\sum_{\vec x}^{(\star)}   \sum_{{\bm \al}} C _{{\bm \al}, \vec x} \cdot \Psi(\cal G({\bm\al},\vec x), \cal E, {\bm\xi}) \prec \OO_\tau \left( \Phi^{t_2 + 2t  - t_s }\Gamma^{t_2 +2t}\right) .
\ee 
By Lemma \ref{zuomeng}, $\Psi(\cal G, \cal E, {\bm\xi})$ satisfies the ordered nested property, hence there exists $\pi\in S_{t+r_2}$ such that \eqref{suz} holds for the free $\Psi$-molecules. Without loss of generality, we assume that $\pi=(1,2,\cdots, t+r_2 )\in S_{t+r_2}$. Then we can repeat the arguments between \eqref{yikezz2EE2} and \eqref{phir0}. Using \eqref{above_types2} and \eqref{ljsadfl} for $1\le s \le t+r_2$, we can get that 
$$\Psi(\cal G, \mathcal E, {\bm\xi})\prec \OO_{\wt\tau}\Big(\Phi^{2t_2+3t+ \frac{1}{2}\sum_{s \text{ free}}\left({\rm deg}(\cal M_s)-4\right)-2t-2r_2} \prod_{1\le s\le t+r_2} \Psi_{\wt x_{s} y_s}\Psi_{\wt x_{s} \wt y_s}\Big),$$
where for any $1\le s \le t+r_2$, $\wt x_s \in {\cal M}^{\Psi}_s$ and 
$$ y_s, \wt y_s\in   \{\star\} \cup \left(\cup_{s'<s} \{\wt x_{s'}\}\right) \cup \left(\cup_{t+r_2 <s'\le t+r_2 + r_1 } {\cal M}^\Psi_{ s' }\right) \ .$$
Plugging it into the left-hand side of \eqref{bound_sf} and using the same argument as in \eqref{phir0}, we obtain that
\begin{align*}
 \sum_{\vec x}^{(\star)}   \sum_{{\bm \al}} C _{{\bm \al}, \vec x} \cdot \Psi(\cal G({\bm\al},\vec x), \cal E, {\bm\xi}) &\prec \OO_{\wt\tau} \left(\Phi^{2t_2-2r_2+t+ \frac{1}{2}\sum_{s \text{ free}}\left({\rm deg}(\cal M_s)-4\right)}\cdot\sum_{\wt x_1,\cdots, \wt x_{t+r_2}}^{(\star)}  \prod_{1\le s\le t+r_2} \Psi_{\wt x_{s} y_s}\Psi_{\wt x_{s} \wt y_s} \right) \\
 & \prec \OO_{\tau} \left(\Phi^{2t_2-2r_2+t+ \frac{1}{2}\sum_{s \text{ free}}\left({\rm deg}(\cal M_s)-4\right)} \Gamma^{2(t+r_2)}\right) .
\end{align*}
Furthermore, by considering the internal structure of free molecules, it is easy to improve this bound to 
\be\label{internal_Psi}
\sum_{\vec x}^{(\star)}   \sum_{{\bm \al}} C _{{\bm \al}, \vec x} \cdot \Psi(\cal G({\bm\al},\vec x), \cal E, {\bm\xi}) \prec  \OO_{\tau} \left(\Phi^{2t_2-2r_2+t+ \frac{1}{2}\sum_{s \text{ free}}\left({\rm deg}(\cal M_s)-4\right)} \Gamma^{2(t+r_2)}\Phi^{a}\right) ,
\ee
where $a$ is the number of free molecules that satisfy at least one of the following conditions:
\begin{enumerate}
\item   there exists one dashed line inside the free molecule;
\item   there exists one off-diagonal solid edge inside the free molecule;
\item   there exists one light weight inside the free molecule.
\end{enumerate}
Note that in the first case, the sum of $C_{ {\bm \al}, \vec x}$ over the atoms in the free molecule gives an extra factor $W^{-d}$ due to loss of free indices. 

With \eqref{deg_mi}, it is easy to see that 
$$\frac{1}{2}\sum_{s \text{ free}}\left({\rm deg}(\cal M_s)-4\right) + a \ge t-t_s,$$
the number of {\it non-simple} free molecules.  Moreover, since $\Phi \ll 1$ and $\Gamma\ge 1$ by \eqref{GM1.5}, the right-hand side of \eqref{internal_Psi} increases as $r_2$ increases. Then with $2r_2 \le t_2$, we can further bound \eqref{internal_Psi} by
$$\sum_{\vec x}^{(\star)}   \sum_{{\bm \al}} C _{{\bm \al}, \vec x} \cdot \Psi(\cal G({\bm\al},\vec x), \cal E, {\bm\xi}) \prec  \OO_{\tau} \left(\Phi^{t_2+t} \Gamma^{2t+t_2}\Phi^{t-t_s}\right) = \OO_{\tau} \left(\Phi^{t_2+2t-t_s} \Gamma^{t_2+2t}\right).$$
This proves \eqref{bound_sf}.
\end{proof}

Note that we always have $t_2+2t\le 2p$ under \eqref{num_pol}. Then as in \eqref{heur}, one can see that \eqref{yikezz33} will follow from Lemma \ref{kaofu22} if we can write the left-hand side of \eqref{yikezz33} into a linear combination of $\OO(1)$ many standard forests, where each of them has no simple free molecule. 
By Lemma \ref{panglei}, $\cal F_0$ in \eqref{yikezz44F} is a colorless standard forest. Now to prove \eqref{yikezz44}, it suffices to prove the following lemma. 

 \begin{lemma}\label{kaofu}
Suppose the assumptions in Lemma \ref{Qpart} hold. Let $\cal F$ be a colorless standard forest of the form \eqref{fagao}. Then for any fixed $D>0$, we have 
   \be\label{likjuhh}
   \E \cal F=\sum_\kappa \E \cal F_\kappa+\OO(N^{-D}),
\ee 
where $\cal F_\kappa$ are colorless standard forests containing {\it zero} simple free molecules. Moreover, the total number of $\mathcal F_\kappa$ is of order $\OO(1)$.
\end{lemma}

For the following proof of Lemma \ref{kaofu}, it suffices to assume that there exists at least one simple free molecule in $\cal F$. Then our proof consists of an induction argument on the number of simple free molecules. More precisely, we will remove the simple free molecules one by one with the following two steps: step 3 in Section \ref{step3}, and step 4 in Section \ref{charged}. 

 \subsection{Proof of Lemma \ref{Qpart}: step 3.}\label{step3}\
 In this step, we turn all the simple free molecules into {\it regular simple free molecules}, in which all the atoms have degree 2 such that we can apply \eqref{charged_intro1}.

\vspace{5pt}

\noindent{\bf Step 3: Regular simple free molecules}. We now pick a simple free molecule, say $\cal M_i$, in the colorless standard forest. Due to \eqref{3114}, there are only 3 possible cases:
  \be\label{fresim}
   \parbox[c]{0.75\linewidth}{\includegraphics[width=\linewidth]{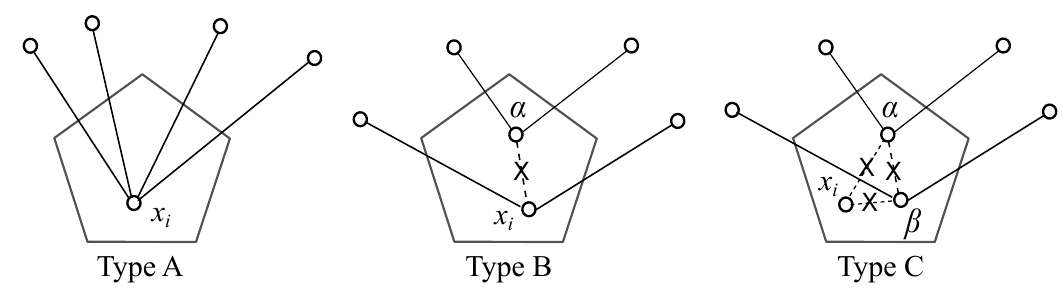}}
\ee
 where we did not draw the weights and the $\al$ atoms with zero degree. We shall call the type C molecules, i.e. the molecules without any solid edge connected with $x_i$ and with only degree 2 atoms, as {\it{regular simple free molecules}}.
 The purpose of this step is to prove the following lemma. 
 
 
 \begin{lemma}\label{sjiyyz} {Suppose} the assumptions in Lemma \ref{Qpart} hold. Let $\cal F$ be a colorless standard forest of the form \eqref{fagao}. For any fixed $D>0$, we have
 $$\E \cal F=\sum_\kappa \E \cal F_\kappa+\OO(N^{-D}),$$
where $\cal F_\kappa$ are colorless standard forests containing only regular (type C) simple free molecules. 
Moreover, the total number of $\mathcal F_\kappa$ is of order $\OO(1)$.
 \end{lemma}

 \begin{proof}
{Recall that $\cal F$ in \eqref{fagao} is built with graphs $\cal G ({\bm \al}, \vec x)$.} Suppose for some $i_0$, $\cal M_{i_0}$ is a type $A$ or $B$ simple free molecule in $\mathcal F$. By the definition of simple free molecules, atom $x_{i_0}$ is not equal to any other atoms in graph $\cal G ({\bm \al}, \vec x)$. For a solid edge, weight or light weight that is not attached to the atom $x_{i_0}$, we use \eqref{Gp1}, \eqref{Gp2} and \eqref{Gp6} to write it as a sum of two parts: one part is independent of the atom $x_{i_0}$; {the other part has two solid edges connected with the atom $x_{i_0}$ and may have a new atom, call it a $\beta$ atom}. Corresponding to these two parts, we can write $\cal G$ as a sum of two parts, say $\cal G=\cal G_0^{(1)}+\cal F_1$. Then for the graph $\cal G_0^{(1)}$, we again expand one of its solid edges or (light) weights that is not connected with the atom $x_{i_0}$, and write it as a sum of two parts, say $\cal G_0^{(1)}=\cal G_0^{(2)}+\cal F_2$. Continuing this process until for some $\cal G_0^{(k)}$, $k\in \mathbb N$, $x_{i_0}$ is added into the independent set of all edges, weights and light weights that are not connected with the atom $x_{i_0}$. Then we rename $\cal G_0\equiv \cal G_0^{(k)}$ and write
%
\be\label{zouasdl}
 \cal G ({\bm \al}, \vec x)= \cal G_0({\bm \al}, \vec x)+\sum_\kappa \sum_{{\bm \beta}} C^\kappa({\bm \beta})\cal G_\kappa({\bm \al}, \; {\bm \beta},\; \vec x) ,
 \ee
where
 \be\label{coeffs} C^\kappa({\bm \beta})=\OO\left((W^{-d})^{\# \; \beta \; atoms}\right){\bf 1 }\left(\max_{i,j} |\beta_i^j-x_i|\le (\log N)^{\OO(1)}W  \right),\ee
and $\cal G_\kappa$ denotes all the other $\cal F$ graphs which have two more solid edges visiting $x_{i_0}$.
Here the new $\beta$ atoms can only come from the expansions in \eqref{Gp6}, and the $C^\kappa({\bm \beta})$ comes from the $s$ coefficients. If the expansion happened for a weight on atom $x_{i}$, then the new atoms $\beta_{i}^j $ satisfy $\beta_{i}^j- x_{i}=\OO(W)$. Similarly, if the expansion happened for a weight on atom $\al^{\star}_{i} \in \mathcal M_i$, then the new atoms $\beta_{i}^j $ satisfy $\beta_{i}^j- \al^{\star}_{i} =\OO(W)$, which implies $\beta_{i}^j- x_{i} \le (\log N)^{\OO(1)}W$. We then include these new $\beta^j_{i}$ atoms into the molecule $\cal M_{i}$ of $\cal G_\kappa$.

So far we have explained how to get \eqref{zouasdl}. Now we consider the dashed-line partitions $\cal E_\kappa$ of the atoms in $\mathcal G_\kappa$, and we shall use $\mathcal E_\kappa \succ \mathcal E$ to mean that $\mathcal E_\kappa$ is an extension of the dashed-line partition $\mathcal E$ of the atoms in $\cal F$. 
Then we have
 $$
  \cal E\cdot \cal G ({\bm \al}, \vec x)=   \cal E\cdot \cal G_0({\bm \al}, \vec x)+\sum_\kappa \sum_{{\bm \beta}} {\sum_{\mathcal E_\kappa \succ \mathcal E}} C^\kappa({\bm \beta})  \cal E_\kappa\cdot \cal G_\kappa({\bm \al}, \; {\bm \beta},\; \vec x).
 $$
Then corresponding to the standard forest $\mathcal F$ in \eqref{fagao}, we define the forests
 \be\label{forest_regular}
 \cal F_0:=  \sum_{\vec x}^{(\star)}   \sum_{{\bm \al}}
 C _{{\bm \al}, \vec x}  \cdot \cal E\cdot  \cal G_0 ({\bm \al}, \vec x),\quad 
 \cal F_{\kappa, \mathcal E_\kappa}:=  \sum_{\vec x}^{(\star)}   \sum_{{\bm \al}} \sum_{{\bm \beta}}  C _{{\bm \al}, \vec x} C^\kappa({\bm \beta})\cal E_\kappa\cdot  \cal G_\kappa({\bm \al}, \; {\bm \beta},\; \vec x)  ,
 \ee
where the molecules are chosen as
$$\cal M_i=\{x_i, \; \al_i^1,\;\al_i^2,\; \al_i^3, \cdots, \beta_i^1,\;\beta_i^2,\; \beta_i^3,\cdots \},\quad 1\le i\le p.$$ 
As we assumed above,  $\cal M_{i_0}$ is a simple free molecule {in $\cal F$}. Clearly, $\cal F_0$ is still a standard forest and $\cal M_{i_0}$ is a simple free molecule in $\cal F_0$. On the other hand, we claim that: 
 \begin{itemize}
\item[(a)] $\cal F_{\kappa,\mathcal E_\kappa}$ is a standard forest;
\item[(b)] $\cal M_{i_0}$ is not a simple free molecule in $\cal F_{\kappa,\mathcal E_\kappa}$ anymore;
\item[(c)] the molecules which are not simple free in $\cal F$ are still not simple free in $\cal F_{\kappa,\mathcal E_\kappa}$. 
\end{itemize} 
Now we prove these statements.

\vspace{5pt}
\noindent{\bf Proof of (b):} Recall that we have obtained $\cal G_\kappa({\bm \al}, \; {\bm \beta},\; \vec x)$ by expanding the solid edges, weights and light weights in $ \cal G ({\bm \al}, \vec x)$ with respect to atom $x_{i_0}$ using \eqref{Gp1}, \eqref{Gp2} or \eqref{Gp6}. Then each of these edges, weights or light weights becomes either (1) the same component with $x_{i_0}$ added to the independent set, or (2) two solid edges connected with atom $x_{i_0}$ plus some weights. Therefore, for any $1\le j\le p$,  $\deg(\cal M_{j})$ does not decrease from $\cal F$ to $\cal F_{\kappa,\mathcal E_\kappa}$. Moreover, there must exist some component in $\cal F$ that turns into case (2) in $\cal F_{\kappa,\mathcal E_\kappa}$, which gives $\deg(\cal M_{i_0})\ge 6$. Hence $\cal M_{i_0}$ is not simple free any more in $\cal F_{\kappa,\mathcal E_\kappa}$, and the above statement (b) holds. 
 
\vspace{5pt}

\noindent{\bf Proof of (c):} If $\cal M_{j} $ is not a free molecule in $\cal F$, then $\cal M_{j} $ is still not free in $\cal F_{\kappa,\mathcal E_\kappa}$ since $\mathcal E_\kappa$ is an extension of $\mathcal E$. Now assume that $\cal M_{j}$ is a non-simple free molecule in $\cal F$. Then we have the following four cases, which can be proved easily with the expansions in \eqref{Gp1}, \eqref{Gp2} and \eqref{Gp6}.
 \begin{itemize}
\item If $\deg(\cal M_{j})\ge 6$ in $\cal F$, then $\deg(\cal M_{j})\ge 6$ in $\cal F_{\kappa,\mathcal E_\kappa}$ since $\deg(\cal M_{j})$ does not decrease from $\cal F$ to $\cal F_{\kappa,\mathcal E_\kappa}$.
\item If there is a dashed line inside $ \cal M_{j}$ in $\cal F$, then $ \cal M_{j} $ still contains this dashed line in $\cal F_{\kappa,\mathcal E_\kappa}$.
\item If there is an off-diagonal solid edge inside $ \cal M_{j}$ in $\cal F$, then either $ \cal M_{j} $ still contains this off-diagonal solid edge in $\cal F_{\kappa,\mathcal E_\kappa}$, or $\deg(\cal M_{j})$ increases by at least 2 such that $\deg(\cal M_{j})\ge 6$. 
\item If $ \cal M_{j}$ contains a light weight in $\cal F$, then either $\cal M_{j}$ still has this light weight in $\cal F_{\kappa,\mathcal E_\kappa}$, or $\deg(\cal M_{j})$ increases by at least 2 such that $\deg(\cal M_{j})\ge 6$ in $\cal F_{\kappa,\mathcal E_\kappa}$.
\end{itemize}
Therefore the statement (c) holds.

\vspace{5pt}
\noindent{\bf Proof of (a):} We verify the conditions (i)-(vi) in Definition \ref{def: SF} one by one.  
 \begin{itemize}
\item (i) and (ii) are trivial. (v) is due to \eqref{coeffs}.

\item (iii) is due to Lemma \ref{Lumm} and Lemma \ref{Lumm22}.

\item For any $\cal M_{j}$,  $\deg(\cal M_{j})$ increases by $0$ or $2$ from $\cal F$ to $\cal F_{\kappa,\mathcal E_\kappa}$. Hence the condition (iv) holds.

\item For condition (vi),
we assume that $\cal M_{j}$ is a simple free molecule in $\cal F_{\kappa,\mathcal E_\kappa}$. Then it is must be also a simple free molecule in $\cal F$ by the statement (c) we just proved. By the expansion rules and the fact that $\deg (\cal M_{j})$ does not change, it is easy to see that the solid edges connected with $\cal M_{j}$ do not change from $\cal F$ to $\cal F_{\kappa,\mathcal E_\kappa}$ except that the atom $x_i$ may be added into the independent sets of these edges.
Therefore \eqref{3113} and \eqref{3114} hold for $\cal F_{\kappa,\mathcal E_\kappa}$. \nc
\end{itemize}
Therefore the statement (a) holds.

 \vspace{5pt}
 
By the above statements (a)-(c), we know that the number of simple free molecules in $\cal F_{\kappa,\mathcal E_\kappa}$ is strictly smaller than that of $\cal F$. Now we consider the $\cal F_0$ term. In fact, assuming $\cal M_{i_0}$ is a type A or B simple free molecule, we will show that for any fixed $D>0$, 
\be\label{jiaqian}
\E_{x_{i_0}}\cal F_0=\sum_{\kappa}\cal F_{\kappa} + \OO_\prec(N^{-D}),
\ee
where each $\cal F_{\kappa}$ is a standard colorless forest, in which the total number of type $A$ and type $B$ simple free molecules is strictly smaller than that of $\cal F$. 
Then with mathematical induction, we can finish the proof of Lemma \ref{sjiyyz} by relabelling the standard colorless forests.
 
Now we prove \eqref{jiaqian}. In $\cal G_0$, all the edges and weights are independent of the atom $x_{i_0}$, except for the ones connected with atom $x_{i_0}$ directly. Then we can write
$$
\cal G_0=\cal G_0^{in}\cdot \cal G_0^{out}, \quad \E_{x_{i_0}} \cal G_0=\cal G_0^{out} \cdot \E_{x_{i_0}} \cal G_0^{in},
$$
where $\cal G_0^{in}$ consists of the edges and weights attached to atom $x_{i_0}$. 
Now applying Lemma \ref{lemEk}, we can write $\E_{x_{i_0}} \cal G_0^{in}$ as a linear combination of graphs with new atoms $\al_{new}^j$ as in \eqref{Gzah}: 
$$
\E_{x_{i_0}} \cal G_0= \sum_\kappa \sum_{\vec \al_{new}} C^{\kappa}(\vec \al_{new} )\cdot \left(\cal G^{in}_\kappa (\vec \al_{new}) \cdot \cal G^{out}_0\right)+\OO_\prec(N^{-D}),
 $$
where $ C^{\kappa}(\vec \al_{new} )$ satisfies
$$ C^\kappa(\vec \al_{new} )=\OO(W^{-d})^{ \# \; new \; \al \; atoms}{\bf 1 }\left(\max_{j} |\al_{new}^j-x_{i_0}|\le (\log N)^{\OO(1)}W  \right).$$
 We include these new atoms into the molecule $\cal M_{i_0}$ such that {from $\cal G_0^{in}$ to $\cal G^{in}_\kappa (\vec \al_{new})$, only the internal structure of $\cal M_{i_0}$ changes}. Thus we have
\be\label{jiaqian23}
\E_{x_{i_0}}\cal F_0=\sum_{ \kappa} \sum_{\vec x}^{(\star)}   \sum_{{\bm \al}}\sum_{\vec \al_{new}}\left(C _{{\bm \al}, \vec x} \cdot C^{\kappa}(\vec \al_{new} )\right)\cdot \cal E\cdot \left(\cal G^{in}_\kappa (\vec \al_{new}) \cdot \cal G^{out}_0\right)+\OO_\prec(N^{ -D}).
\ee
Again, let $\cal E_{\kappa}$ denote the dashed-line partition extensions of $\cal E$. We then define the forests
$$
\wt {\cal F}_{\kappa,\mathcal E_\kappa}=\sum_{\vec x}^{(\star)}   \sum_{{\bm \al}}\sum_{\vec \al_{new}}\left(C _{{\bm \al}, \vec x}\cdot C^{\kappa}(\vec \al_{new} )\right)\cdot \cal E_{\kappa }\cdot \left(\cal G^{in}_\kappa (\vec \al_{new})\cdot \cal G^{out}_0\right).
$$
Then we have 
$$
\E_{x_{i_0}}\cal F_0=\sum_{\kappa}{\sum_{\mathcal E_\kappa \succ \mathcal E}}\wt{\cal F}_{\kappa,\mathcal E_\kappa}+\OO_\prec(N^{-D}),
$$
where $\wt{\cal F}_{\kappa,\mathcal E_\kappa}$ are standard colorless forests. Since $\deg(x_{i_0})=0$ in $\wt{\cal F}_{\kappa,\mathcal E_\kappa}$, $\cal M_i$ is not a type $A$ or type $B$ simple free molecule in $\wt{\cal F}_{\kappa,\mathcal E_\kappa}$ anymore. Moreover, for all the other molecules, their types do not change from $\mathcal F_0$ to $\wt{\mathcal F}_{\kappa,\mathcal E_\kappa}$.
Therefore in $\wt{\cal F}_{\kappa,\mathcal E_\kappa}$, the total number of type $A$ and type $B$ simple free molecules is strictly smaller than that of $\cal F_0$. This completes the proof of \eqref{jiaqian} by relabelling the standard colorless forests.
\end{proof} 


 \subsection{Proof of Lemma \ref{Qpart}: step 4.}\label{charged}\
With Lemma \ref{sjiyyz}, it remains to prove Lemma \ref{kaofu} under the assumption that there are only type $C$ simple free molecules in $\cal F$. 
 
 \begin{lemma}\label{yuizhy}  Suppose the assumptions in Lemma \ref{Qpart} hold. Let $\cal F$ be a colorless standard forest of the form \eqref{fagao} and with $t_s$ simple free molecules for some $t_s \ge 1$. Moreover, we assume that they are all type $C$ regular simple free molecules in the sense of \eqref{fresim}. Then for any fixed $D>0$, we have
 $$\E \cal F=\sum_\kappa \E \cal F_\kappa  + \OO(N^{-D}), $$
where each $\cal F_\kappa$ is a colorless standard forest that contains at most $(t_s -1)$ simple free molecules. 
Moreover, the total number of $\mathcal F_\kappa$ is of order $\OO(1)$.
 \end{lemma} 
Clearly, together with Lemma \ref{sjiyyz}, the above Lemma \ref{yuizhy} shows that Lemma \ref{kaofu} holds by induction and hence completes the proof of Lemma \ref{Qpart}.  

For the standard forest $\cal F$ in Lemma \ref{yuizhy}, in each simple free molecule there are exactly 2 atoms that are connected with solid edges, as shown in \eqref{fresim}. Thus the condition \eqref{3113} in Definition \ref{def: SF} implies that one of the atom must be connected with two solid edges of the same charge. {Here we define the} charged atom {to be} an atom whose total charge with respect to the solid edges is not neutral. In this section, we focus on the charged atoms in type C simple free molecules.  

\begin{definition}[Simple charged atom] \label{Def_charge}We call a degree 2 charged atom in a simple free molecule a simple charged atom. 
 \end{definition}
 

 
\noindent{\bf $\mathcal Q_{\al}$ operation:} Let $\al_i$ be a simple charged atom in a simple free molecule $\cal M_i$. For a standard forest $\mathcal F$ of the form \eqref{fagao}, we define an operator $ \cal Q_{\al_i}$, which ``paints" the edges connected with atom $\al_i$ using the color $Q_{\al_i}$:
 $$
 \cal Q_{\al_i}\left(\cal F\right): =\sum_{\vec x}^{(\star)}   \sum_{{\bm \al}}
 C _{{\bm \al}, \vec x}  \cdot \cal E\cdot  \cal Q_{\al_i}\left( \cal G ({\bm \al}, \vec x)\right).
 $$   
The operation of $\mathcal Q_{\alpha}$ can be described graphically as
   \be\label{f,ku}
    \parbox[c]{0.60\linewidth}{\includegraphics[width=\linewidth]{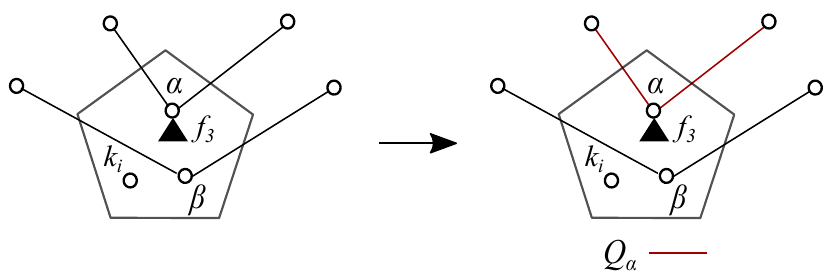}}
   \ee
{Be careful that $\cal Q_{\al}$ is different from $Q_{\alpha}$ acting on the graph: $\cal Q_{\al}$ only paints the edges connected with atom $\al$, while $Q_{\alpha}$ acting on the graph paints the whole graph with $Q_\al$ color. }

\vspace{5pt}

The proof of Lemma \ref{yuizhy} consists of the following two parts. 

\vspace{5pt}
\noindent{\bf Step 4A: Painting the simple charged atoms.} In this step, we apply \eqref{charged_intro1} to the two edges connected with a simple charged atom, i.e., we paint these two edges with some $Q$-color. Rigorously speaking, we shall prove the following lemma. 

 
 \begin{lemma}\label{yuizhy765} 
 Under the assumptions of Lemma \ref{yuizhy}, for any fixed {$D>0$}, we can write 
 \be
 \label{lametn} \cal F=\sum_\kappa \cal F_\kappa+\sum_{\tilde \kappa}    \cal Q_{\al_\star}\left(\cal F_{\tilde \kappa}\right)+\OO_\prec(N^{-D}),
 \ee
where $\cal F_\kappa$ and $\cal F_{\tilde \kappa}$ are colorless standard forests, and $\al_\star$ is some simple charged atom in $\cal F_{\tilde \kappa}$. Furthermore, each $\cal F_\kappa$ contains at most $(t_s-1)$ simple free molecules, and each $\cal F_{\tilde \kappa}$ contains at most $t_s$ simple free molecules. Here the total number of $\cal F_\kappa$ and $\cal F_{\tilde \kappa}$ is of order $\OO(1)$. 
\end{lemma} 
 
 \vspace{5pt}
\noindent{\bf Step 4B: Annihilation of $Q$-colored simple free molecules.} As discussed in the introduction, taking expectation over the $Q$-colored graphs in \eqref{lametn} will decrease the number of simple free molecules. We state it in terms of the following lemma.
 
 \begin{lemma}\label{yuiz3i87} Under the assumptions of Lemma \ref{yuizhy}, if $\al_\star$ is a simple charged atom in $\cal F$, then 
   $$\E \cal Q_{\al_\star} (\cal F)=\sum_\kappa \E \cal F_\kappa + \OO(N^{-D}), $$
where each $\cal F_\kappa$ is a colorless standard forest that contains at most $(t_s-1)$ simple free molecules. Moreover, the total number of $\cal F_\kappa$ is of order $\OO(1)$. 
\end{lemma} 

Clearly, Lemma \ref{yuizhy765} and Lemma \ref{yuiz3i87} together prove Lemma \ref{yuizhy}. 


\begin{proof}[Proof of Lemma \ref{yuizhy765}]  
We assume that $\al_\star$ is a simple charged atom in a simple free molecule $\cal M_{i_0}$ of $\cal F$. By definition, we have
\be\label{nyuan}
\cal F=  \sum_{\vec x}^{(\star)}   \sum_{{\bm \al}}
 C _{{\bm \al}, \vec x}  \cdot \cal E\cdot  \cal G ({\bm \al}, \vec x),
\ee
and there are two same charged solid edges connected with $\al_\star$ in  $\cal G ({\bm \al}, \vec x)$. Up to the choice of directions and charges, we can assume that these two solid edges are 
\be\label{zjxiang}
G^{(\dot I)}_{\al_\star\gamma_1}G ^{(\ddot I )}_{\al_\star\gamma_2},
\ee
where $\dot I$ and $\ddot  I$ are independent sets of these two solid edges, and $\gamma_1$, $\gamma_2$ are atoms outside $\cal M_{i_0}$. In the following proof, we only focus on this case, but the proof works for all the other choices of directions and charges. We can rewrite \eqref{zjxiang} with the following lemma. We postpone its proof to Appendix \ref{pf qingqz}. 

\begin{lemma}\label{qingqz} 
Given atoms $y,y'$ and independent sets $\dot I, \ddot  I$. Let $x\notin\{y,y'\}\cup \dot I\cup \ddot I$, and $J$ be any set of atoms such that 
$$
\{y,y'\}\cup \dot I\cup \ddot I\subset J ,\quad \quad x\notin J. 
$$
Define the notations
$$
\dot G:=G^{(\dot I)}, \quad \ddot G:= G^{(\ddot I)}, \quad 
\dot \Lambda_x :=\dot G_x -M_x , \quad \ddot \Lambda_x :=\ddot G_x -M_x .
$$
Then {under the assumptions of Lemma \ref{Qpart}}, for any fixed $D>0$ we can write  
\be \label{shangwmnzi}  
\dot G_{xy}\ddot G_{xy'}= Q_x \left(\dot G_{xy}\ddot G_{xy'}\right) +\sum_{w\notin \{x\}\cup J} 
c_{xw} Q_w \left(\dot G_{wy}\ddot G_{wy'}\right) +\sum_\kappa \sum_{\vec\al}C^{\kappa}_{\vec \al }\,\cdot \cal G_\kappa(\vec \al , x, y, y')+\OO_\prec(N^{ -D }),
 \ee 
and
 \be\label{shangwmnws} 
Q_x\left(\dot G_{xy}\ddot G_{xy'}\right)= \dot G_{xy}\ddot G_{xy'}+
 \sum_{y}\wt c_{xw}\dot G _{wy}\ddot G _{wy'}+\sum_\kappa \sum_{\vec \al}{\wt C}^{\kappa}_{\vec \al }\,\cdot \cal {\wt G}_\kappa(\vec \al , x, y, y')+\OO_\prec(N^{ -D }),
\ee
where $\vec \al$ denotes the new atoms, the deterministic coefficients satisfy 
 $$
 c_{xw},\, \wt c_{xw}=\OO(W^{-d}){\bf 1}_{|x-w|\le (\log N)^2W}, \quad C^{\kappa}_{\vec \al },\;  {\wt C}^{\kappa}_{\vec \al } =\OO\left(\left(W^{-d}\right)^{\#\text{ of $\al$ atoms}}\right)  {\bf 1}\left(\max_l |x-\al_l |\le (\log N)^2W\right),
 $$
and $\cal G_{\kappa}$, $\cal {\wt G}_{\kappa}$ are colorless graphs look like 
  \be\label{kjlksj8iu}
   \parbox[c]{0.80\linewidth}{\includegraphics[width=\linewidth]{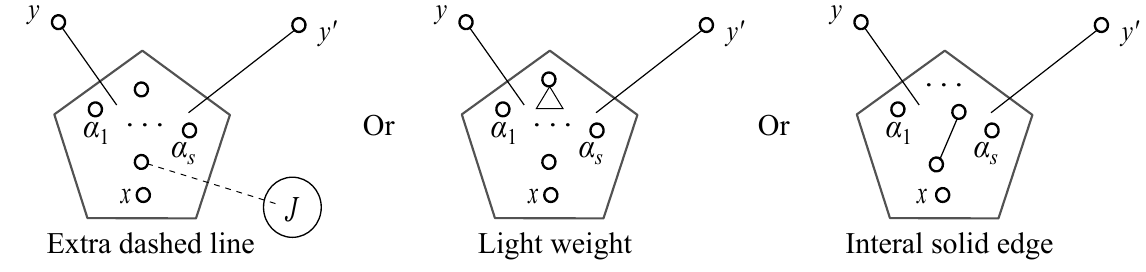}}
   \ee 
More precisely, $\cal G_{\kappa}$, $\cal {\wt G}_{\kappa}$ are  colorless graphs with atoms 
$ x,\al_1,  \al_2,\cdots $, and $y,y'$. For simplicity of presentation, we shall call $\cal M:=\{x, \al_1,  \al_2,\cdots\}$ a molecule (which is consistent with our previous definition). Then in each $\cal G_{\kappa}$ or $\cal {\wt G}_{\kappa}$, there are two solid edges connecting atoms $y$ and $y'$ to the atoms in $\cal M$, and there are no dotted lines. Moreover, at least one of the following three cases holds.
 \begin{itemize}
\item There exists at least one dashed line connecting an atom in $\cal M$ to an atom in $J$.  
Here by convention, we assume that the atoms in $J$ are also included in the graphs $\cal G_\kappa$ or $\cal {\wt G}_{\kappa}$.  
\item There exists at least one light weight on the atoms in $\cal M$. 
\item  There exists at least one off-diagonal solid edge between atoms in $\cal M$. 
\end{itemize}
Finally, the total number of $\cal G_{\kappa}$ and $\cal {\wt G}_{\kappa}$ graphs is of order $\OO(1)$. 
 \end{lemma}


\vspace{5pt}

In applying Lemma \ref{qingqz} to \eqref {zjxiang}, we let $\al_\star$ play the role of $x$ in \eqref{shangwmnzi}, $\tilde \al_\star$ play the role of $w$, $\gamma_{1}$ and $\gamma_2$ play the role of $y$ and $y'$, $J$ be the set of all atoms in $\cal G ({\bm \al}, \vec x)$ of \eqref{nyuan}, and we label the new atoms by $\vec\beta_{i_0}$. Then we obtain that
\be\label{zjxiang23}
\begin{split}
\dot G_{\al_\star\gamma_1}\ddot G_{\al_\star\gamma_2} & =
Q_{\al_\star}\left(\dot G_{\al_\star\gamma_1}\ddot G_{\al_\star\gamma_2}\right) +\sum_{\tilde \alpha_\star \notin \{\al_\star\}\cup J} 
c_{\al_\star\tilde \alpha_\star} Q_{\tilde \alpha_\star}\left( \dot G_{\tilde \alpha_\star \gamma_1}\ddot G_{\tilde \alpha_\star\gamma_2} \right) \\
& +\sum_\kappa \sum_{\vec \beta_{i_0}} C^{\kappa}_{\vec \beta_{i_0} }\, \cal G_\kappa\left(\vec \beta_{i_0} , \al_\star, \gamma_1, \gamma_2\right) +\OO_\prec(N^{ -D}),
\end{split}
\ee
with 
$$ 
 c_{\tilde \al_\star} =\OO(W^{-d}){\bf 1}\left(|\tilde \alpha_\star-\al_\star|\le (\log N)^2W\right), \ \
 C^{\kappa}_{\vec \beta_{i_0} } =O\left(\left(W^{-d}\right)^{\#\text{ of $\beta$ atoms}}\right)  
 {\bf 1}\left(\max_l |\beta_{i_0}^l -\al_\star |\le (\log N)^2W\right).
 $$ 
 Now we plug \eqref{zjxiang23} into the graph $\cal G ({\bm \al}, \vec x)$, and we include these new $\beta$'s atoms into the molecule $\cal M_{i_0}$, which contains $x_{i_0}$ and $\alpha_\star$. Since $|\al_\star-x_{i_0}|\le (\log N)^{\OO(1)} W $, we have 
 $$ 
 c_{\tilde \al_\star} =\OO(W^{-d}){\bf 1}\left(|\tilde \al_\star-x_{i_0}  |\le (\log N)^{\OO(1)} W \right),$$
 and
 $$
 C^{\kappa}_{\vec \beta_{i_0} } =O\left(\left(W^{-d}\right)^{\#\text{ of $\beta$ atoms}}\right)  
 {\bf 1}\left(\max_l  |\beta_{i_0}^l -x_{i_0}|\le (\log N)^{\OO(1)} W \right).
 $$ 
Now with these coefficients and \eqref{zjxiang23}, we can write $\cal F$ in \eqref{nyuan} as 
\be\label{xjlajzh}
\cal F= \cal Q_{\al_\star}\cal F + \cal Q_{\tilde \al_\star} \cal {\wt F} +\sum_\kappa \sum_{\cal E_\kappa \succ \cal E}\cal F_{\kappa, \mathcal E_\kappa}+\OO_\prec(N^{-D}),
\ee
where $\wt{ \cal F}$ and $\cal F_{\kappa, \mathcal E_\kappa}$ are colorless forests defined as follows:
$$
  \wt{ \cal F} : =\sum_{\vec x}^{(\star)}   \sum_{{\bm \al}}\sum_{\tilde\al_\star} \left( C _{{\bm \al}, \vec x}\cdot c_{\al_\star \tilde\al_\star}\right) \cdot \wt {\cal E}\cdot \wt{  \cal G} ({\bm \al}, \tilde \al_\star, \vec x) , \qquad  \wt{  \cal G} ({\bm \al}, \tilde \al_\star, \vec x) :=
   \frac{{  \cal G} ({\bm \al},  \vec x) }{\dot G_{\al_\star \gamma_1}\ddot G_{\al_\star\gamma_2} } \dot G_{\tilde \al_\star \gamma_1}\ddot G_{\tilde \al_\star\gamma_2} ,
$$
where $\wt {\cal E}$ is the unique dashed-line partition extension of $\cal E$ in which $\tilde\al_\star$ is not equal to any other atoms in $\wt{  \cal G} ({\bm \al}, \tilde \al_\star, \vec x) $; 
$$  \cal F_{\kappa,\mathcal E_\kappa}: =\sum_{\vec x}^{(\star)}   \sum_{{\bm \al}} \sum_{\vec \beta_{i_0}}
\left( C _{{\bm \al}, \vec x} \cdot  C^{\kappa}_{\vec \beta_{i_0} } \right) \cdot   \cal E_{\kappa}\cdot  \cal G_{\kappa}  ({\bm \al}, \vec \beta_{i_0}, \vec x),  \qquad \cal G_{\kappa}  ({\bm \al}, \vec \beta_{i_0}, \vec x)  : = \frac{ {  \cal G} ({\bm \al},  \vec x) }{\dot G_{\al_\star \gamma_1}\ddot G_{\al_\star\gamma_2} }\cal G_\kappa(\vec \beta_{i_0} , \al_\star, \gamma_1, \gamma_2 ) ,
$$
where $\cal E_\kappa$ is a dashed-line partition extension of $\cal E$. 
Here we recall that $\cal G$ is the graph in \eqref{nyuan} and $\cal G_\kappa$ is the graph in \eqref{zjxiang23}. 

For the above $\wt F$ and $\cal F_{\kappa,\mathcal E_\kappa}$, it is easy to prove that they are standard colorless forest. Furthermore, compared with $\cal F$, there is no new simple free molecules appearing in $\cal {\wt F}$ and $\cal F_{\kappa,\mathcal E_\kappa}$, and $\cal M_i$ is not simple free anymore in $\cal F_{\kappa,\mathcal E_\kappa}$. Hence $\cal F_{\kappa,\mathcal E_\kappa}$ has at most $(t_s-1)$ simple free molecules, and $\cal {\wt F}$ has at most $t_s$ simple free molecules. Now with \eqref{xjlajzh}, we can finish the proof of Lemma \ref{yuizhy765}. Note that for the second term on the right-hand side of \eqref{xjlajzh}, we need to switch the names of $\al_\star$ and $\tilde \al_\star$, and relabel the standard forests.
\end{proof}

Next we prove Lemma \ref{yuiz3i87} using Lemma \ref{qingqz}.
 
  \begin{proof}[Proof of Lemma \ref{yuiz3i87}. ]
 We consider 
\be \nonumber
 \mathcal Q_{\al_\star}\left(\cal F\right) =  \sum_{\vec x}^{(\star)}   \sum_{{\bm \al}} C _{{\bm \al}, \vec x}  \cdot \cal E\cdot  \mathcal Q_{\al_\star}\left(\cal G ({\bm \al}, \vec x)\right),
\ee
where $\al_\star$ is a simple charged atom in a simple free molecule $\cal M_{i_0}$. In particular, atom $\al_\star$ is not equal to any other atoms in the graph $\cal G ({\bm \al}, \vec x)$, and there are only two edges painted with the $Q_{\al_\star}$ color in $ \cal Q_{\al_\star} (\cal F)$. 
As we did in the proof of Lemma \ref{sjiyyz}, we expand all {\it colorless} solid edges, weights and light weights in $ \cal G ({\bm \al}, \vec x)$ with respect to the atom $\al_\star$ as follows. 
 \begin{itemize}
\item We use \eqref{Gp1} and \eqref{Gp2} to expand the solid edges, $f_{1,2}$ weights and $f_{1,2}$ light weights which are not connected with atom $\al_\star$. We use \eqref{Gp6} to expand the $f_{3,4}$ weights and $f_{3,4}$ light weights which are not attached to atom $\al_\star$. 

\item Since $\cal M_{i_0}$ is a simple free molecule, the weights on atom $\al_\star$ must be normal weights (i.e., not light weights). For these weights, we write them as
 $$G_{x_{i_0}x_{i_0}}=M_{x_{i_0}}+ (G_{x_{i_0}x_{i_0}}-M_{x_{i_0}}), \quad 
\left(G_{x_{i_0}x_{i_0}}\right)^{-1}=M^{-1}_{x_{i_0}}+ \left(\left(G_{x_{i_0}x_{i_0}}\right)^{-1}-M^{-1}_{x_{i_0}}\right),$$ 
i.e. we expand $f_{1,2}$ weights into an order 1 deterministic quantity plus $f_{1,2}$ light weights. For $f_{3,4}$ weights on atom $\al_\star$, we keep them unchanged, since they are already independent of the $\al_\star$ atom. 
\end{itemize}

Then as in the proof of Lemma \ref{sjiyyz}, we can expand the edges and weights, which are not connected with the atom $\al_\star$ directly, step by step: 
$$\cal G=\cal G_0^{(1)}+\cal F_1=\cal G_0^{(2)}+\cal F_1+\cal F_2=\cdots,$$
until we get that
\be\label{zouasdl22}
 \cal G ({\bm \al}, \vec x)= \cal G_0({\bm \al}, \vec x)+\sum_\kappa \sum_{{\bm \beta}} C^\kappa({\bm \beta})\cal G_\kappa({\bm \al}, \; {\bm \beta},\; \vec x),
\ee
where 
$$C^\kappa({\bm \beta})=O\left((W^{-d})^{\# \; \beta \; atoms}\right){\bf 1 }\left(\max_{i,j} |\beta_{i}^j-x_{i}| \le (\log N)^{\OO(1)}W  \right).$$
Note that \eqref{zouasdl22} corresponds to \eqref{zouasdl} in the proof of Lemma \ref{sjiyyz}. In $\cal G_0$, expect for the two solid edges connected with ${\al_\star}$, all the other edges and weights are independent of the $\al_\star$ atom.
Thus we immediately get that
\be\label{mpzx}
\E_{\al_\star}\mathcal Q_{\al_\star}\left(\cal G_0({\bm \al}, \vec x)\right)=0.
\ee
As explained in the proof of Lemma \ref{sjiyyz}, the new $\beta$ atoms in $\cal G_\kappa$ come from the expansions in \eqref{Gp6}, and the coefficients $C^\kappa({\bm \beta})$ come from the $s$ coefficients in \eqref{Gp6}. If the expansion happened for a weight on atom $x_{i}$, then the new atom $\beta_{i}^j $ satisfies  $\beta_{i}^j- x_{i}=\OO(W)$. Similarly, if the expansion happened for a weight on atom $\al$ in molecule $\cal M_i$, then the new atom $\beta_{i}^j$ satisfies  $\beta_{i}^j- \al =\OO(W)$, which implies $|\beta_{i}^j- x_{i} |\le (\log N)^{\OO(1)}W$. We then include these new $\beta^j_{i}$ atoms into the molecule $\cal M_{i}$ of $\cal G_\kappa$. 

We define
$$\cal F_{\kappa,\mathcal E_\kappa}: = \sum_{\vec x}^{(\star)}   \sum_{{\bm \al}} \sum_{{\bm \beta}} 
\left( C _{{\bm \al}, \vec x} \cdot C^\kappa({\bm \beta})\right) \cdot \cal E_\kappa\cdot  \cal G_\kappa({\bm \al}, \; {\bm \beta},\; \vec x).$$
Note that it has the same structure as the forest $\cal F_{\kappa, \cal E_\kappa}$ in \eqref{forest_regular} in the proof of Lemma \ref{sjiyyz}. Suppose in $\cal G({\bm \al},\vec x)$, the two solid edges connected with $\al_\star$ are $G^{(\dot I)}_{\al_\star\beta_\star}G^{(\ddot I)}_{\al_\star\gamma_\star}$. Then we define
$${\cal G}^{Q}_\kappa({\bm \al}, \; {\bm \beta},\; \vec x):=\frac{\cal G_\kappa({\bm \al}, \; {\bm \beta},\; \vec x)}{G^{(\dot I)}_{\al_\star\beta_\star}G^{(\ddot I)}_{\al_\star\gamma_\star}}Q_{\al_\star}\left(G^{(\dot I)}_{\al_\star\beta_\star}G^{(\ddot I)}_{\al_\star\gamma_\star}\right),$$
and the forest
$$\cal F^{Q}_{\kappa,\mathcal E_\kappa}: = \sum_{\vec x}^{(\star)}   \sum_{{\bm \al}} \sum_{{\bm \beta}} 
\left( C _{{\bm \al}, \vec x} \cdot C^\kappa({\bm \beta})\right) \cdot \cal E_\kappa\cdot  \cal G^Q_\kappa({\bm \al}, \; {\bm \beta},\; \vec x).$$
With \eqref{mpzx}, we obtain that
$$
\E  \cal Q_{\al_\star} (\cal F)=\sum_\kappa \sum_{\mathcal E_{ \kappa} \succ \mathcal E}\E  \cal F^Q_{\kappa,\mathcal E_\kappa}.
$$
Note that $\cal F^Q_{\kappa, \cal E_\kappa}$ has the same structure as $\cal F_{\kappa, \cal E_\kappa}$ defined above except that $\cal F^Q_{\kappa, \cal E_\kappa}$ has two $Q$-colored edges. 

For the colorless forest $\cal F_{\kappa,\mathcal E_\kappa}$, the statements (a)-(c) in the proof of Lemma \ref{sjiyyz} also hold here. Therefore we have that the number of simple free molecules in $\cal F_{\kappa,\mathcal E_\kappa}$ is strictly smaller than that of $\cal F$. Finally, we apply \eqref{shangwmnws} to the $Q$-colored solid edges $Q_{\al_\star}\left(G_{\al_\star\beta_\star}G_{\al_\star\gamma_\star}\right)$ with $\al_\star$ playing the role of atom $x$ in \eqref{shangwmnws}, and we add the new $w$ and $\al$ atoms in \eqref{shangwmnws} into the molecule $\cal M_{i_0}$ (i.e. the one containing $\al_\star$). 
Then for each $\cal F^{Q}_{\kappa,\mathcal E_\kappa}$, we can write 
  $$
 \cal F^Q_{\kappa,\mathcal E_\kappa}=\sum_{\kappa'} \cal F_{\kappa,\mathcal E_\kappa, \kappa'} + \OO_\prec(N^{-D}),
 $$
where $\cal F_{\kappa,\mathcal E_\kappa, \kappa'}$ are colorless standard forests by replacing $Q_{\al_\star}\left(G_{\al_\star\beta_\star}G_{\al_\star\gamma_\star}\right)$ with the terms on the right-hand side of \eqref{shangwmnws}. Moreover, it is easy to see that each $\cal F_{\kappa,\mathcal E_\kappa, \kappa'}$ contains at most the same number of simple free molecules as $\cal F_{\kappa,\mathcal E_\kappa}$, i.e., each $\cal F_{\kappa,\mathcal E_\kappa, \kappa'}$ contains at most $(t_s-1)$ simple free molecules. This completes the proof of Lemma \ref{yuiz3i87} by relabelling the standard forests.
\end{proof}

 \appendix
\section{Proof of \eqref{zlende}}\label{appd}

With Taylor expansion, we can write
\begin{equation}\label{taylor}
 (1-|m|^2 S)^{-1}S = \left( 1-|m|^{2K}S^{K}\right)^{-1}\sum_{k=0}^{K-1}|m|^{2k}S^{k+1}.
 \ee
Since $\|S\|_{l^\infty \to l^\infty} = 1$ and $|m|\le 1 - c\eta$ for some constant $c>0$ by \eqref{msc}, it is easy to see that by taking $K=\eta^{-1} $ in \eqref{taylor}, we have
\be\label{taylor1}
0\le  \left[(1-|m|^2 S)^{-1}S\right]_{xy} \le C \max_{x, y}\sum_{k=1}^{\eta^{-1}} (S^{k})_{xy}.
\ee
Since $S$ is a doubly stochastic matrix, $(S^{k})_{xy}$ can be understood through a $k$-step random walk on the torus $\Z_N^d$. We first prove the following lemma. Here different from the previous proof, for any vector $v\in \R^d$ we denote $| v|\equiv \| v\|_2$. 

\begin{lemma}\label{reg RW2} 
Let $B_n = \sum_{i=1}^n X_i$ be a random walk on $\mathbb Z^d$ with $i.i.d.$ steps $X_i$ that satisfy the following conditions: (i) $X_1$ is symmetric; (ii) $|X_1| \le L$ almost surely; (iii) there exists constants $C^*,c_*>0$ such that 
\begin{equation}\label{core_condition2}
c_* L^{-d} {\mathbf 1} _{|x| \le  c_* L}\le  \mathbb P(|X_1|=x)  \le C^* L^{-d} {\mathbf 1} _{|x| \le L}, \quad  j\in \Z.
\end{equation}
Let $\Sigma$ be the covariance matrix of $X_1$ with $\Sigma_{ij}=\mathbb E [ (X_1)_i (X_1)_j ]$. Assume that $n\in \mathbb N$ satisfies 
\begin{equation}\label{nL_relation}
 \log n \ge c_0\log L
\end{equation}
for some constant $c_0>0$. Then for any fixed (large) $D>0$, we have
\be\label{RW_diffusion2}
\mathbb P\left(B_n = x\right) = \frac{1+\oo_n(1)}{(2\pi n)^{d/2} \sqrt{\det(\Sigma)}}e^{-\frac1{2} x^T (n\Sigma)^{-1} x}+\OO(L^{-D}),
\ee
for large enough $L$ (and $n$).
\end{lemma}

\begin{proof} [Proof of Lemma \ref{reg RW2}]
Note that \eqref{RW_diffusion2} is in accordance with the central limit theorem. Our proof below is in fact a variant of the proof of CLT with characteristic functions. 

Combining condition (ii), i.e., $|X_i|\le L$, with a large deviation estimate, with \eqref{nL_relation}, we get that
$$\P\left(|B_n| \ge L n^{1/2+\tau}\right) = \OO(L^{-D}), $$
for any fixed (small) $\tau >0$ and (large) $D>0$. Thus to prove \eqref{RW_diffusion2}, we only need to focus on the case 
$$|x|=\OO(L n^{1/2+\tau_0})$$ 
for some small enough constant $\tau_0>0 $. In the following proof, we always make this assumption.

For $p\in\R^d $ with $0<|p|\le L^{-1}n^{-1/2+\tau_0}$, we have 
$$ \log \E e^{ \ii p \cdot X_1}= - \frac12 p^T \Sigma p + \sum_{k\ge 3} \frac{\kappa_k(\hat p)}{k!} (\ii |p|)^k,$$
where $\hat p= p/|p|$ and $\kappa_k(\hat p)$ is the $k$-th cumulant of $\hat p\cdot X_1$. It gives that
$$\frac{1}{n}\log \E e^{ \ii p \cdot  B_n}= - \frac12 p^T \Sigma p + \sum_{k\ge 3} \frac{\kappa_k(\hat p)}{k!} (\ii |p|)^k. $$
By the condition (\ref{core_condition2}), it is easy to verify that
\be\label{operator_sigma}
C^{-1}L^{2}\le \Sigma \le CL^2
\ee
in the sense of operators, and 
$$ |\kappa_k(\hat p)|\le  C^k k! L^k, \quad k\in \N, \quad \hat p \in \mathbb S^d,$$
for some constant $C>0$. 
Then for $|p|\le L^{-1}n^{-1/2+\tau_0}$, we have 
\be\label{ganzz} 
\E e^{\ii p \cdot B_n}=e^{-\frac12 n p^T\Sigma p}\Big(1+\sum_{3 \le k \le K_D}\al_k(\hat p)( Ln^{1/2} |p|)^k\Big) + \OO(L^{-D}), 
 \ee
where $\al_k\in \C$ are coefficients (independent of $p$) satisfying
 \be\label{amp}
  \al_k(\hat p)=\OO(n^{1-k/2})=\OO(n^{-k/6}), \quad k\ge 3, 
  \ee
and $K_D=\OO(1)$ is a fixed integer depending only on $D$ and the constant $c_0$ in \eqref{nL_relation}. 
 
Now we estimate $\E e^{\ii p\cdot B_n}$ for large $p$. 
Because of the existence of the core in \eqref{core_condition2}, it is easy to see that for some constant $c>0$, 
$$| \E e^{ \ii p\cdot X_1}|\le 1-c\min\{1,(L|p|)^2\},\quad   L^{-1}n^{-1/2+\tau_0} \le |p|\le  \pi,$$
 which implies that for some $c>0$,
 $$\left| \E e^{\ii p \cdot B_n}\right|\le e^{-cn^{\tau_0}},\quad |p|\ge L^{-1}n^{-1/2+\tau_0}. $$
 Together with \eqref{ganzz}, with 
 $$
 y:=(n\Sigma)^{-1/2}x, \quad |y|=\OO(n^{\tau_0}), \quad q:=(n\Sigma)^{1/2}p ,
 $$ 
and $H_n$ being  the Hermite polynomials,  we have 
\begin{equation}\nonumber
\begin{split}
\P\left(B_n =x \right)& = \frac{1}{(2\pi)^d}\int_{|p| \le L^{-1}n^{-1/2+\tau_0}} \rd p \, e^{-\ii p\cdot x}e^{-\frac12 n p^T\Sigma p}\Big(1+\sum_{3 \le k \le K_D}\al_k(\hat p)( Ln^{1/2} |p|)^k\Big) +\OO(L^{-D})\\
& =\frac{1}{(2\pi)^d\sqrt{ n^d \det(\Sigma)}} \int_{|L\Sigma^{-1/2}q|\le n^{\tau_0}} \rd q \, e^{ -\ii q \cdot y}e^{-\frac{q^2}{2}}\left(1+\sum_{3\le k\le K_D}\alpha_k(\hat p) |L\Sigma^{-1/2}q|^k\right)+\OO(L^{-D})\\
& =\frac{1}{(2\pi)^d\sqrt{ n^d \det(\Sigma)}} \int_{q\in \mathbb R^d} \rd q \, e^{ -\ii q \cdot y}e^{-\frac{q^2}{2}}\left(1+\sum_{3\le k\le K_D}\alpha_k(\hat p) |L\Sigma^{-1/2}q|^k\right)+\OO(L^{-D})\\
&=\frac{1}{(2\pi n)^{d/2} \sqrt{\det(\Sigma)} }\left(1+\sum_{3\le k \le K_D}\OO(n^{-({1}/{6} - \tau_0) k} ) \right) e^{-\frac{y^2}{2}}+ \OO(L^{-D})
\end{split}
\end{equation}
where in the third step we used $C^{-1/2}|q|\le |L\Sigma^{-1/2}q|\le C^{1/2}|q|$ by \eqref{operator_sigma} and approximated the $\int_{|L\Sigma^{-1/2}q|\le n^{\tau_0}}$ with $\int_{q\in \R}$ up to an error $\OO(L^{-D})$ due to the factor $e^{-q^2/2}$, and in the last step we used \eqref{amp}, $|y|=\OO(n^{\tau_0})$ and stationary approximation to bound the integrals. This proves \eqref{RW_diffusion2}.
\end{proof} 
%
%

Now we can give a proof of \eqref{zlende}.
\begin{proof}[Proof of \eqref{zlende}]
Fix any small constant $\tau>0$. We now bound the sum in \eqref{taylor1}. Let $B_n = \sum_{i=1}^n X_i$ be a random walk on $\Z_N^d$ with $i.i.d.$ steps $X_i$, with distribution $\mathbb P(X_1 = y-x)=s_{xy}$. Then it is easy to see that
$$(S^k)_{xy}=\mathbb P(B_k = y - x).$$

For $1\le k \le N^{\tau}$, with \eqref{bandcw1} we can bound 
\be\label{Sk1} (S^k)_{xy} \le \mathbf 1_{|x-y| \le C_s k W} W^{-d} \lesssim \frac{N^{(d-2)\tau}}{W^2\langle x-y\rangle^{d-2}}.\ee
For $N^{\tau} \le k \le N^{2-\tau}/W^2$, we have a large deviation estimate
\be\nonumber
\P\left(|B_k| \ge |x-y| \right) \le \exp\left( - \frac{c|x-y|^2}{k^2 W^2}\right)
\ee
for some constant $c>0$. In particular, with high probability, $b_x$ can be regarded as a random walk on the full lattice $\Z^d$ if $k \le N^{2-\tau}/W^2$, and we can apply \eqref{RW_diffusion2} to get that 
\begin{align}\label{Sk2}
(S^k)_{xy}=\mathbb P( B_k= y-x) \lesssim \frac{1}{k^{d/2} W^d}e^{-\frac{c}{2kW^2} |y-x|^2 }+\OO(N^{-D}),
\end{align}
 for some constant $c>0$ and for any large constant $D>0$. Finally, for $N^{2-\tau}/W^2 \le k \le \eta^{-1}$, using $\|S\|_{l^\infty \to l^\infty} \le 1$ we get that
 \be\label{Sk3}
 (S^k)_{xy} \le \max_{x,y}(S^{N^{2-\tau}/W^2})_{xy} \le \frac{1}{N^{d-d\tau/2}}
 \ee
where we used \eqref{Sk2} in the last step. Applying \eqref{Sk1}-\eqref{Sk3} to \eqref{taylor1}, we obtain that
\begin{align*}
\sum_{k=1}^{\eta^{-1}} (S^{k})_{xy} \lesssim \frac{N^{(d-2)\tau}}{W^2\langle x-y\rangle^{d-2}} +\frac{N^{d\tau/2}}{N^{d}\eta}+ \sum_{N^\tau \le k\le N^{2-\tau}/W^2} \frac{1}{k^{d/2} W^d}\mathbf 1_{k\ge N^{-\tau}|x-y|^2/W^2} + \OO(N^{-D}),
\end{align*}
where it is easy to verify that
\begin{align*}
\sum_{N^\tau \le k\le N^{2-\tau}/W^2} \frac{1}{k^{d/2} W^d}\mathbf 1_{k\ge N^{-\tau}|x-y|^2/W^2} &\lesssim \frac{1}{W^d}\mathbf 1_{|x-y|\le N^
\tau W} + \frac{1}{W^d (N^{-\tau}|x-y|^2/W^2)^{d/2-1}}\mathbf 1_{|x-y|\ge N^
\tau W} \\
&\lesssim \frac{N^{(d-2)\tau}}{W^2 \langle x-y\rangle^{d-2}}.
\end{align*}
This finishes the proof of \eqref{zlende} since $\tau$ can be arbitrarily small and $D$ can be arbitrarily large.
\end{proof}

\section{Proof of Lemma \ref{qingqz}}\label{pf qingqz}
 We fix $x,y, y'$ in the proof. For simplicity, we ignore ``$x$" from the coefficients $c_{xw}$ and $\wt c_{xw}$. 
With \eqref{Gp3}, we can write  
 $$
 G_{xy}G_{xy'}=M_x^2\sum^{(x)}_{s,s'} H_{xs} H_{xs'}G^{(x)}_{sy}G^{(x)}_{s'y'}+\left(2\Lambda_x M_x+\Lambda_x ^2\right)\frac{G_{xy}}{G_{xx}}\frac{G_{xy'}}{G_{xx}}, \quad \Lambda_x:= G_{xx}-M_x.
 $$
Similarly for $x\notin\{y,y'\}\cup \dot I\cup \ddot I$, together with \eqref{Gp1}, we get that
\begin{equation}\label{hangkong}
\begin{split}
& \E_x \dot G_{xy}\ddot G_{xy'}
 =M_x^2\sum_{w} s_{xw}\dot G^{(x)}_{wy}\ddot G^{(x)}_{w y'}
 + \E_x \left(\left(\dot \Lambda_x M_x+\ddot\Lambda_x  M_x+\ddot \Lambda_x \dot\Lambda_x\right)\frac{\dot G_{xy}}{\dot G_{xx}}\frac{\ddot G_{xy'}}{\ddot G_{xx}}\right)\\
& =M_x^2\sum_{y} s_{xw}\dot G _{wy}\ddot G _{wy'}
-M_x^2\sum_{y} s_{xw}\left(\dot G _{wy}\frac{\ddot G_{wx}\ddot G_{xy'}}{\ddot G_{xx}}
+ \frac{\dot G_{wx}\dot G_{xy}}{\dot G_{xx}}\ddot G_{wy'}
- \frac{\dot G_{wx}\dot G_{xy}}{\dot G_{xx}} \frac{\ddot G_{wx}\ddot G_{xy'}}{\ddot G_{xx}}\right)\\
&+ \E_x \left(\left(\dot \Lambda_x M_x+\ddot\Lambda_x  M_x+\ddot \Lambda_x \dot\Lambda_x\right)\frac{\dot G_{xy}}{\dot G_{xx}}\frac{\ddot G_{xy'}}{\ddot G_{xx}}\right).
\end{split}
\end{equation}
Now we define the ``error" part as
\be\label{sdfuyys}
\begin{split}
\cal B_x : & = -M_x^2\sum_{y} s_{xw}\left(\dot G _{wy}\frac{\ddot G_{wx}\ddot G_{xy'}}{\ddot G_{xx}}
+ \frac{\dot G_{wx}\dot G_{xy}}{\dot G_{xx}}\ddot G_{wy'}
- \frac{\dot G_{wx}\dot G_{xy}}{\dot G_{xx}} \frac{\ddot G_{wx}\ddot G_{xy'}}{\ddot G_{xx}}\right)\\
&+ \E_x \left(\left(\dot \Lambda_x M_x+\ddot\Lambda_x  M_x+\ddot \Lambda_x \dot\Lambda_x\right)\frac{\dot G_{xy}}{\dot G_{xx}}\frac{\ddot G_{xy'}}{\ddot G_{xx}}\right), \quad \text{ if } x\notin \{y,y'\}\cup {\dot I\cup \ddot I},
\end{split}
\ee
and
\be\nonumber
\cal B_x:=\E_x\dot G_{xy}\ddot G_{xy'}-M_x^2\sum_{y} s_{xw}\dot G _{wy}\ddot G _{wy'}, \quad \text{ if } x\in \{y,y'\}\cup {\dot I\cup \ddot I}.
\ee
With the above definition and \eqref{hangkong}, we have for any $x\in \Z_N^d$,
\be\label{yyxj}
\dot G_{xy}\ddot G_{xy'}=
M_x^2\sum_{y} s_{xw}\dot G _{wy}\ddot G _{wy'}+Q_x\left(\dot G_{xy}\ddot G_{xy'}\right)+\cal B_x.
\ee
It implies (with $\{y,y'\}\cup \dot I\cup \ddot I\subset J$) 
\begin{align*}
\dot G_{xy}\ddot G_{xy'} & =\sum_{y}\left[(1-M^2S)^{-1}\right]_{xw} \left(Q_w\left(\dot G_{wy}\ddot G_{wy'} \right)
+  \cal B_w\right)\\
&=\sum_{w\notin J}  \left[(1-M^2S)^{-1}\right]_{xw} \left(Q_w\left(\dot G_{wy}\ddot G_{wy'} \right)
+  \cal B_w\right)\\
& +  \sum_{w\in J}  \left[(1-M^2S)^{-1}\right]_{xw}
\sum_{w'}  \left( 1-M^2S  \right)_{ww'} \left(\dot G_{w' y}\ddot G_{w' y'}\right).
\end{align*}
Then using (\ref{tianYz}), we obtain that for any fixed $D>0$, 
\begin{equation}\label{suyys}
\begin{split}
\dot G_{xy}\ddot G_{xy'} -Q_x \left(\dot G_{xy}\ddot G_{xy'} \right)
  &= \sum_{w\notin J}
   c_{w}   Q_w\left(\dot G_{wy}\ddot G_{wy'}  \right)
+ { \cal B_x}+  \sum_{j\notin J}
   c_{w} \cal B_w \\
& +  
\sum_{w\in J } c_{w}   \dot G_{wy}\ddot G_{wy'}  
+ \sum_{w\in J}\sum_{w'} d_{ww'} \dot G_{w' y}\ddot G_{w' y'} 
+\OO_\prec(N^{ -D}),
\end{split}
\end{equation}
for some coefficients satisfying
$$c_{w}=\OO(W^{-d}){\bf 1}_{|x-w|\le (\log N)^2W} ,\quad d_{ww'}= \OO(W^{-2d}){\bf 1}_{|x-w|+|x-w'|\le (\log N)^2W}.
$$
Furthermore, by the definition of $\cal B_w$, we have 
\begin{equation*}
\begin{split}
\sum_{j\notin J}
   c_w    \cal B_w &=\sum_{w\notin J}  \sum_{v} c'_{w} s_{wv} \left(\dot G _{vy}\frac{\ddot G_{vw}\ddot G_{wy'}}{\ddot G_{ww}} + \frac{\dot G_{vw}\dot G_{wy}}{\dot G_{ww}}\ddot G_{vy'} - \frac{\dot G_{vw}\dot G_{wy}}{\dot G_{ww}} \frac{\ddot G_{vw}\ddot G_{wy'}}{\ddot G_{ww}}\right) \\
&+\sum_{w\notin J}c_w \E_w \left(\left(\dot \Lambda_w M_{w}+\ddot\Lambda_w  M_{w}+\ddot \Lambda_w \dot\Lambda_w\right)\frac{\dot G_{wy}}{\dot G_{ww}}\frac{\ddot G_{wy'}}{\ddot G_{ww}}\right),
\end{split}
\end{equation*}
for some coefficients
\be\label{cj12}
c'_w =\OO(W^{-d}){\bf 1}_{|x-w|\le (\log N)^2W}.
\ee
Therefore, up to the error term $\OO_\prec(N^{-D})$, $\dot G_{xy}\ddot G_{xy'}-Q_x\left(\dot G_{xy}\ddot G_{xy'}\right)$ is equal to (see the explanation below)
  \be\label{tiank}
   \parbox[c]{0.78\linewidth}{\includegraphics[width=\linewidth]{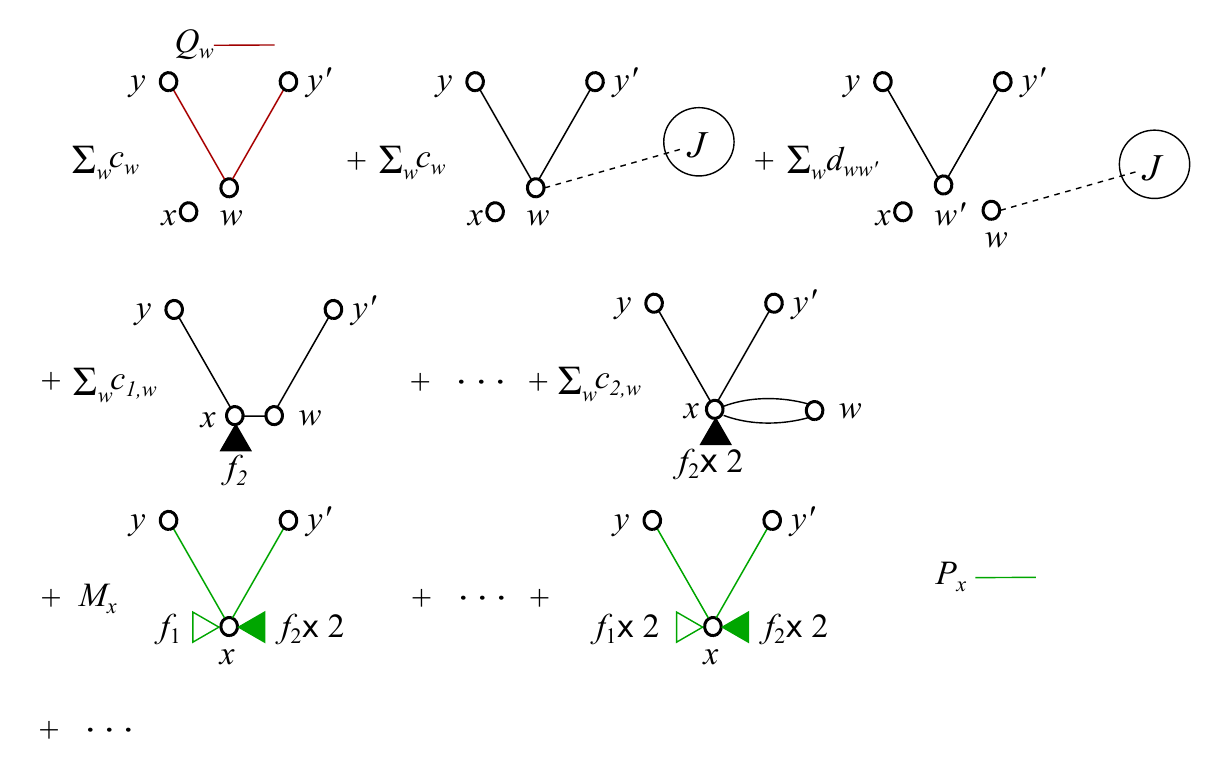}}
\ee
where $c_{1,w}$ and $c_{2,w}$ are some coefficients that also satisfy (\ref{cj12}). Here we have only drawn the dashed lines and ignored the $\times$-dashed lines. Moreover, the $y$ and $y'$ can be the same atom, but we did not draw this case. The first graph in the first row represents the first term on the right-hand side of \eqref{suyys}. The second and third graphs in the first row represent the two terms in the second line of \eqref{suyys}. The second row of \eqref{tiank} represents the the first row of \eqref{sdfuyys}, and  the third row of \eqref{tiank} represents the the second row of \eqref{sdfuyys}. The graphs of $\sum_{w\notin J}c_w \cal B_w $ have the same structures as the graphs in the second and third rows of \eqref{tiank}, and we used ``$\cdots$" to represent them in the fourth row.

Now the first graph in \eqref{tiank} gives the second term on the right-hand side of \eqref{shangwmnzi}. All the other graphs in the first and second rows of \eqref{tiank} can be included into the third term on the right-hand side of \eqref{shangwmnzi} by relabelling $w$, $w'$ as $\al$ atoms. It is easy to check that these graphs satisfy the conditions for $\cal G_\kappa$ below \eqref{kjlksj8iu}. Therefore to finish the proof of \eqref{shangwmnzi}, it remains to write the graphs in the third line of \eqref{tiank} into the form of the third term on the right-hand side of \eqref{shangwmnzi}.

Following the idea in the proof for Lemma \ref{lemEk}, we can write the graph with $P_x$ color into a sum of colorless graphs.  More precisely, 
using  
$$
\frac{\dot G_{xy}}{\dot G_{xx}}\frac{\ddot G_{xy'}}{\ddot G_{xx}}=
\sum_{\al_1,\al_2}H_{x\al_1}H_{x\al_2}\dot G^{(x)}_{\al_1 y}\ddot G^{(x)}_{\al_2 y'},\quad 
 \dot G _{xx}-M_x = \left(\dot {\cal Y}_{x}\right)^{-1}-M_x +  \sum_{m=1}^\infty (\cal Y_{x})^{-m-1} (\cal Z_x)^m ,
$$
and taking partial expectation $\E_x$, we can write the graphs in the third row of \eqref{tiank} as
\be\label{3rdrow}
\sum_\kappa \sum_{\vec \al}C^{\kappa}_{\vec \al }\cdot \cal G_\kappa(\vec \al , x,  y,y')+\OO_\prec(N^{-D}),
\ee
where
$$
C^{\kappa}_{\vec \al }=O\left(\left(W^{-d}\right)^{\#\text{ of $\al$ atoms}}\right){\bf 1}\left(\max_l{|x-\al_l |\le (\log N)^2W}\right),
$$
and 
$\cal G_\kappa(\vec \al , x, y,y')$ are colorless graphs which look like the graphs in \eqref{kjlksj8iu}.
 In fact, it is easy to check that $\cal G_\kappa$ either has a light weight (i.e. $f_4$ light weight on the atom $x$ or $f_6$ weight on some $\al$ atom) or there exists a solid line between $\al$ atoms. 
Hence (\ref{3rdrow}) can be written into the form of the third term on the right-hand side of \eqref{shangwmnzi} and satisfies the conditions below \eqref{kjlksj8iu}. This completes the proof of \eqref{shangwmnzi} in Lemma \ref{qingqz}. 

For \eqref{shangwmnws}, we use \eqref{yyxj} and see that it suffices to write $\cal B_x$ into the form of the third term on the right-hand side of \eqref{shangwmnws}, which have been done above. Thus we finish the proof of \eqref{shangwmnws}.
 

  \end{document}